\documentclass[preprint,11pt]{elsarticle}
\usepackage{amssymb}
\usepackage{amsmath}
\usepackage{eucal}
\usepackage{graphicx}
\usepackage{float}
\usepackage[font=footnotesize,labelfont=bf]{caption}
\DeclareFontFamily{U}{calligra}{}
\DeclareFontShape{U}{calligra}{m}{n}{<->callig15}{}
\usepackage{comment}
\usepackage{subcaption}
\usepackage{svg}

\usepackage{caption}
\usepackage{booktabs}
\usepackage{array}
\usepackage{tabularx}
\usepackage{amsthm}
\usepackage{multirow}
\usepackage{relsize}
\usepackage{caption}
\newcommand{\calE}{{\!\!\text{\usefont{U}{calligra}{m}{n}E}\,\,}}
\usepackage{algorithm}
\usepackage{algpseudocode}
\newcommand{\calC}{{\!\!\text{\usefont{U}{calligra}{m}{n}C}\,\,}}
\newcommand{\calT}{{\!\!\text{\usefont{U}{calligra}{m}{n}T}\,\,}}
\newtheorem{theorem}{Theorem}
\newtheorem{prop}{Proposition}

\usepackage{caption}
\usepackage{mathrsfs}
\usepackage{multirow}
\usepackage[toc]{appendix}
\newcommand\norm[1]{\left\lVert#1\right\rVert}
\usepackage{mathtools, stmaryrd}
\usepackage{xparse} \DeclarePairedDelimiterX{\Iintv}[1]{\llbracket}{\rrbracket}{\iintvargs{#1}}

\title{A Mass Preserving Numerical Scheme for Kinetic Equations that Model Social Phenomena}

\author{Yassin Bahid\textsuperscript{1,*},
Eduardo Corona\textsuperscript{1},
Nancy Rodriguez\textsuperscript{1}\\
\bigskip
\textbf{\textsuperscript{1}}Department of Applied Mathematics, University of Colorado, Boulder, Colorado, USA\\
\bigskip
*Corresponding Author: yassin.bahid@colorado.edu}

\begin{document}
\begin{frontmatter}

\begin{abstract}
In recent years, kinetic equations have been used to model many social phenomena. A key feature of these models is that transition rate kernels involve Dirac delta functions, which capture sudden, discontinuous state changes. Here, we study kinetic equations with transition rates of the form
$$
T(x,y,u) = \delta_{\phi(x,y) - u}.
$$
We establish the global existence and uniqueness of solutions for these systems and introduce a fully deterministic scheme, the \emph{Mass Preserving Collocation Method}, which enables efficient, high fidelity simulation of models with multiple subsystems. We validate the accuracy, efficiency, and consistency of the solver on models with up to five subsystems, and compare its performance against two state-of-the-art agent-based methods: Tau-leaping and hybrid methods. Our scheme resolves subsystem distributions captured by these stochastic approaches while preserving mass numerically, requiring significantly less computational time and resources, and avoiding variability and hyperparameter tuning characteristic of these methods.

\end{abstract}
\end{frontmatter}
\section{Introduction}


Kinetic equations have emerged as a powerful framework for modeling complex social dynamics. Originally developed in the study of gas dynamics \cite{cercignani1988boltzmann}, kinetic modeling provides a bridge between agent-based models and macroscopic descriptions, capturing interactions at the mesoscopic scale. The application of kinetic ideas to social dynamics began with the Weidlich master equation approach to opinion formation \cite{weidlich1971statistical}, followed by the kinetic model of Galam and his collaborators for social conflict \cite{galam1994towards}. Toscani later proposed a kinetic framework for opinion dynamics that incorporated both self-thinking and interaction effects \cite{toscani2006kinetic}. At the same time, Pareschi and Toscani extended this approach to study wealth distribution patterns \cite{pareschi2013interacting}.

Kinetic equations have since been applied across various domains in the social sciences. Borsche and collaborators developed models of migration flows based on population density and perceived location advantages \cite{borsche2014coupling}. Bellomo and collaborators used kinetic modeling to investigate criminal behavior and the effectiveness of law enforcement \cite{bellomo2014systems}. More recent studies have focused on the propagation of information and misinformation in online networks, namely, Franceschi {\it et al.} modeled the dynamics of fake news by accounting for user gullibility and fact-checking behavior \cite{franceschi2022spreading}. During the COVID-19 pandemic, kinetic models were also used to study the spread of the epidemics in heterogeneous populations, as shown by Dimarco {\it et al.} \cite{dimarco2020wealth}.

While kinetic models provide a flexible framework for representing social dynamics, their application to social systems introduces distinctive features and challenges that are not typically encountered in physical systems. One notable difference is the need to capture abrupt state changes, such as immediate decisions in response to crime opportunities or rapid escalations in unrest following provocation. In kinetic models, these changes are captured by the transition rates. Since changes in social systems are frequently sudden and triggered by specific events, they are best modeled with a Dirac delta function, which represents an instant shift rather than a gradual transition. Another key aspect of interest is the assumption that mass is preserved. In other words, the ``mass" of the system, representing the population, should remain constant unless explicitly accounted for by factors such as population growth or decline, e.g., as incorporated in \cite{bellomo2014systems}. The presence of these Dirac delta functions alongside strict enforcement of mass preservation poses unique challenges in the development of efficient and accurate numerical solvers. 

\emph{Numerical methods for kinetic equations:} the simulation of systems via kinetic models has a rich history in the physical and biological sciences; a variety of numerical methods have been developed to handle application-specific complexities and challenges. In rarefied gas dynamics, the Direct Simulation Monte Carlo (DSMC) technique, introduced by Bird \cite{bird1994molecular}, has become a standard approach. In the study of dense gases and liquids, Molecular Dynamics (MD) simulations are commonly used to capture interactions at the particle level \cite{frenkel2023understanding}. In biological systems, stochastic chemical kinetics are often modeled using the Gillespie algorithm, which accurately simulates the timing and sequence of reaction events \cite{gillespie, gillespie2007stochastic}. 

Many of these numerical techniques have been adapted to social systems, most notably agent-based simulations based on the Gillespie \cite{gillespie} and Tau-leaping \cite{gillespie2001approximate} algorithms, which use Monte Carlo methods to capture individual-level variability \cite{helbing2012social}. Hybrid approaches combining stochastic ABMs with deterministic solvers have also gained traction \cite{yates2020blending}. Yet challenges remain: ABMs are computationally expensive, scale poorly with population size, and are difficult to parallelize due to frequent interactions. A limitation of hybrid methods is the reliance on user-defined hyperparameters that govern the transition from deterministic to stochastic dynamics, which can complicate implementation and reduce consistency \cite{flegg2012two}.



Deterministic integration of kinetic equation models at both the system and subsystem levels offers a useful alternative to agent-based models (ABMs). In this work, we develop a numerical solver to address the specific computational challenges that arise in models typical of social applications.


{\it Key contributions}: This paper introduces the mass-preserving collocation method (MPCM), a fully deterministic scheme for efficiently and accurately integrating kinetic equations with sharp transition rates, in contrast to existing stochastic or hybrid methods. The MPCM is constructed to automatically preserve mass numerically and, unlike the Tau-leaping and hybrid methods, it does not require any hyperparameter tuning.

We demonstrate the effectiveness of the MPCM on a number of test case kinetic models featuring two subsystems and four representative transition rates, showing that it conserves mass and reproduces the expected qualitative behavior. For each of these, we compare its performance with that of two state-of-the-art algorithms: Tau-leaping and hybrid. Across all examples presented on this work, the MPCM method reliably matches system behavior and subsystem distributions while requiring orders of magnitude less computational resources. Finally, we show how MPCM allows us to simulate models with a larger number of subsystems for which competing approaches may become impractical. 


{\it Outline}: The paper is structured as follows. Section \ref{sec:KE} introduces a general framework of kinetic equations to model social dynamics. In Section \ref{sec:transitions}, we detail the specific transition rates considered and provide illustrative examples of their projected behavior. Section \ref{sec:scheme} introduces the mass-preserving collocation scheme (MPCM). Section \ref{sec:results} presents the simulation results for various kinetic systems using MPCM and compares them with Tau-leaping and hybrid methods, demonstrating that MPCM offers substantial computational advantages. 

\section{Kinetic Equations - Model Derivation}\label{sec:KE}

Kinetic equations, when used in sociological modeling, borrow ideas from statistical mechanics to help explain how social systems change over time. In this approach, individuals (or agents) are treated as particles that interact within a social space. To build such a model, we follow the framework from \cite{bellomo2013difficult}. In a large population, each person is seen as an active particle with a certain strategy, called {\it microstate}, which defines their individual state. Based on their microstate, agents are grouped into subsystems, each described by a distribution over possible microstates. Interactions between agents are modeled using ideas from game theory and sociology \cite{nowak2006evolutionary, short2010cooperation}. The overall behavior of the system is then described using integro-differential equations, which track how agents move between subsystems. Figure \ref{fig:kineq} illustrates this approach: agents are divided into subsystems, each linked to a microstate. When the population is large, we focus less on individual behavior and more on how the sizes of these subsystems change over time.
 
\begin{figure}[H]
    \centering
    \includegraphics[width=0.9\linewidth]{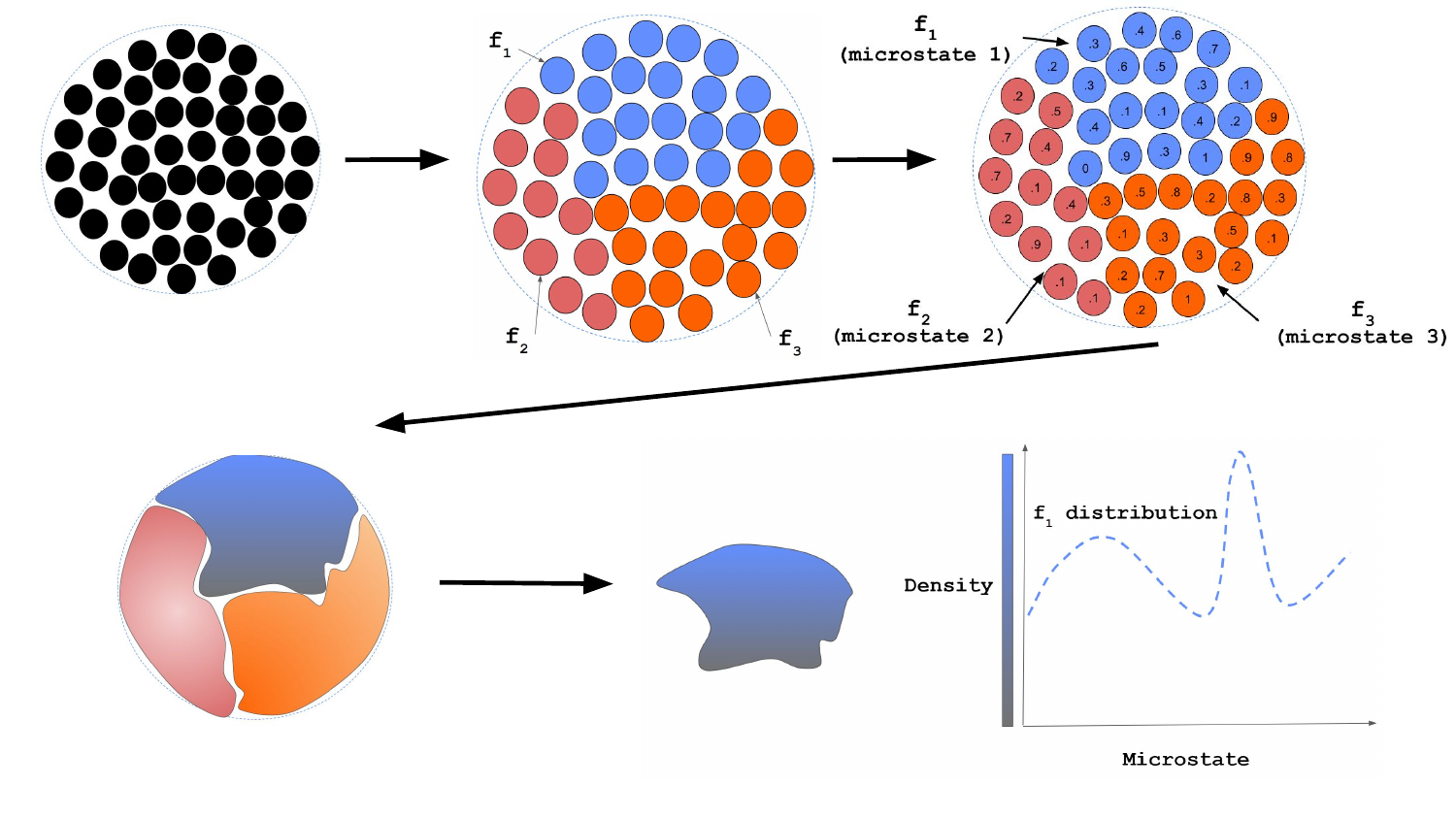}
    \caption{The transition from agent-based models to kinetic equations involves shifting from tracking individual agents to grouping them into subsystems with assigned microstates. The system is then described using density distributions for each subsystem. The objective is to capture macroscopic behavior while preserving essential microscopic dynamics. This approach bridges discrete agent interactions and continuous statistical representations.}
    \label{fig:kineq}
\end{figure}

\subsection{Functional Subsystems}\label{sec: Func Subsystems}
In a social system, functional subsystems represent groups based on roles, traits, or positions, such as social status, income, or beliefs. They help us understand how interactions between individuals influence the spread of these traits and lead to shared norms or opinions. Microstates describe the individual situation of each person, while subsystems group similar individuals and influence how the whole system changes over time.
For a kinetic system with $n$ subsystems, the microstate of subsystem $i$ lies within a non-negative real interval $D_i \subset \mathbb{R}^{+}$. The distribution function of each subsystem is represented by
$$
f_i : [0,T] \times D_i \rightarrow \mathbb{R}^+, \;\text{for}\; i \in \mathbb{S},
$$
where $T$ is a final time, possibly $+\infty$ and $\mathbb{S}=\{1,2,\dots, n\}$. The function $f_i(t,u)$ denotes the microstate density of subsystem $i$ at time $t$, with microstate varying with $u$. It is assumed that $f_i(t,u) \geq 0$ for all subsystems $i$, microstates $u \in D_i$, and times $t \in [0,T]$.

\subsection{Micro-scale interactions}\label{sec:micscale}

Microscopic interactions link sociological theories to mathematical representations through rules that govern the evolution of agents. Two rates of micro-scale interactions are necessary. First, the {\it encounter rates}, $\eta_{ij}(x,y)$, real, non-negative bounded functions, describe the interaction rate between an agent of the subsystem $i$ with microstate $x$ and an agent from subsystem $j$ with microstate $y$. Second, the {\it transition rates}, $T_{ij}^k(x,y,z)$, represent the probability that an agent from subsystem $i$ with microstate $x$ transitions to subsystem $k$ with microstate $z$ after interacting with an agent from subsystem $j$ with microstate $y$, where $i, j, k \in \mathbb{S}$, and the respective microstates are $x \in D_i$, $y \in D_j$, and $z \in D_k$.


\subsection{Kinetic Systems}\label{sec:IDE}

The way micro-scale interactions, captured by the transition rates and encounter rates, influence the dynamics of the density functions defined in Section \ref{sec: Func Subsystems} can then be modeled with an integro-differential equation. Given $n$ subsystems $\{f_1, f_2, \dots, f_n\}$, the equation governing the dynamics of subsystem $i$ at time $t$ is given by
\begin{equation}\label{eq:integrodiff}
\begin{split}
    \partial_t f_i(t,u) =& \sum_{h,k} \int_{D_h} \int_{D_k}\eta_{hk}(x,y)T_{hk}^i(x,y,u)f_h(t,x)f_k(t,y)dxdy \\
    &- f_i(t,u)\sum_{k}\int_{D_k}\eta_{ik}(u,x)f_k(t, x)dx,
\end{split}
\end{equation}
for $i\in \mathbb{S}$ and $t>0.$ In equation \eqref{eq:integrodiff}, the infinitesimal change in the density of a given microstate $u$ in subsystem $i$ at time $t$ is determined by the net effect of inflow and outflow resulting from pairwise interactions between subsystems. 

The first term on the right-hand side of equation \eqref{eq:integrodiff}
%
%
represents the inflow of density into subsystem $i$ with microstate $u$. This inflow results from interactions between agents from subsystem $h$, in all possible microstates $x \in D_h$, and agents from subsystem $k$, in all microstates $y \in D_k$. The encounter rate $\eta_{hk}(x,y)$ determines how frequently these interactions occur, while the transition rate $T_{hk}^i(x,y,u)$ specifies the likelihood that such an interaction causes an agent from subsystem $h$ to transition into subsystem $i$ with microstate $u$.
The second term on the right-hand side
%
%
describes the outflow of density from subsystem $i$, specifically from microstate $u$, due to interactions with agents from subsystem $k$ in microstate $x \in D_k$. The interaction occurs at a rate $\eta_{ik}(u,x)$, and the magnitude of this term reflects the frequency and intensity of these interactions.

Note that the model given by equation \eqref{eq:integrodiff} does not include birth or death processes. Therefore, we expect the solutions to system \eqref{eq:integrodiff} to preserve a constant total mass. Mathematically, mass conservation implies that the total infinitesimal change in the density functions of all subsystems is zero
$$
\sum_{i\in \mathbb{S}} \int_{D_i} \partial_t f_i(t,u) du = 0,\quad \forall t \in [0,T],
$$
where $T$ is defined as before.  Mass is conserved in this system if and only if the transition rates satisfy the following condition:
\begin{equation}\label{eq:trans_condition}
\sum_{i\in \mathbb{S}} \int_{D_i} T_{hk}^i(x,y,u) du = 1,\quad \forall (x,y) \in D_h\times D_k,\ \forall  h,k\in \mathbb{S}.
\end{equation}

All kinetic models studied in this work satisfy condition \eqref{eq:trans_condition}. For ease of calculation, the total mass is normalized to one 
$$
\sum_{i \in \mathbb{S}} \int_{D_i} f_i(t,u) \, du = 1, \quad \forall t \in [0,T].
$$
Before presenting our proposed numerical method, we first state a result regarding existence and uniqueness of solutions to the system of equations in \eqref{eq:integrodiff}. Bellomo and collaborators prove the existence and uniqueness of solutions for a kinetic system with three subsystems in \cite{bellomo2014systems}. We generalize their result to a system of $n$ equations with general transition rates satisfying the condition \eqref{eq:trans_condition}.
We prove the existence of solutions in $\mathcal{C} = \{ (f_1, \dots, f_n) \mid f_i \in C^1([0,T] \times \mathbf{X}) \}$, where $T \in \mathbb{R}^+$ is an arbitrary final time, and $\mathbf{X} = \{ f = (f_1, f_2\dots, f_n) \mid f_i \in L^1(D_i) \ \forall i \in \mathbb{S} \}$. A function $f \in \mathcal{C}$ means that $f(\cdot, u)$ is continuously differentiable for all $u \in D_i$ and that $f(t, \cdot)$ is integrable for all $t \in [0, T]$. The proof can be found in \ref{app:prf}.

\begin{theorem}[Existence and Uniqueness of Solution]\label{Th:wellpos}
    Given $T$ a final time in $(0,+\infty]$, consider a kinetic system with $n$ subsystems
\begin{equation}\label{eq:integrodiffthrm}  
    \begin{cases}
        \partial_t f_i(t,u) &=\sum_{h,k} \int_{D_h} \int_{D_k}\eta_{hk}(x,y)T_{hk}^i(x,y,u)f_h(t,x)f_k(t,y)dxdy\vspace{2pt} \\
    &\quad - f_i(t,u)\sum_{k}\int_{D_k}\eta_{ik}(u,x)f_k(t, x)dx \vspace{2pt}\\
     f_i(0,u) &=\ f_i^0(u),
   \end{cases}
   \end{equation}
    for $i\in\mathbb{S},$ with positive encounter rates $\{\eta_{ij}\}_{i,j \in \mathbb{S}}$ and positive transition rates $\{T_{ij}^k\}_{i,j,k \in \mathbb{S}}$ that satisfy condition \eqref{eq:trans_condition}. Let $f_0 = (f_1^0,....,f_n^0)\in \{(f_1^0,....,f_n^0)|\; f_i^0 \in  L^1(D_i), \;f_i^0\ge 0 \;\text{for}\; i \in \mathbb{S}\}$ be the initial state of the system satisfying $$\sum_i \int_{D_i}f_i^0(u) du = 1.$$ 
    If all encounter rates are bounded, such that for some $C_{\eta} \in \mathbb{R}^+$
    \begin{align*}
         \max_{(x,y)\ \in \ D_h\times D_k} \eta_{hk}(x,y) < C_\eta,\quad \forall\ h,k \in \mathbb{S},
    \end{align*}
    then the system \eqref{eq:integrodiffthrm} has a unique solution $f \in \calC$.
\end{theorem}

\section{On Transition Rates for Social Dynamics}\label{sec:transitions}

We define a general class of transition rate kernels depending on a transition function $\phi$, obtaining necessary and sufficient conditions such that mass is conserved. 

\begin{prop}\label{prop:deltamasspreserve}
Let $D_i, D_j, D_k$ be measurable sets. 
Suppose that the transition kernel is given by
\[
T_{ij}^k(x,y,u) = \delta_{\phi(x,y)-u},
\]
where $\phi : D_i \times D_j \to \mathbb{R}$ is measurable. 
Then, for each $(x,y) \in D_i \times D_j$,
\[
\int_{D_k} T_{ij}^k(x,y,u)\,du =
\begin{cases}
1, & \text{if } \phi(x,y) \in D_k, \\[4pt]
0, & \text{if } \phi(x,y) \notin D_k.
\end{cases}
\]
Consequently,
\[
\int_{D_k} T_{ij}^k(x,y,u)\,du = 1 \quad \text{for all } (x,y) \in D_i \times D_j
\]
if and only if $\phi(D_i \times D_j) \subset D_k$.
\end{prop}

\begin{proof}
For fixed $(x,y) \in D_i \times D_j$, $T_{ij}^k(x,y,\cdot)$ is the Dirac mass centered at $u=\phi(x,y)$. 
Hence,
\[
\int_{D_k} \delta_{\phi(x,y)-u}\,du =
\begin{cases}
1, & \phi(x,y) \in D_k, \\[4pt]
0, & \phi(x,y) \notin D_k.
\end{cases}
\]

If $\phi(D_i \times D_j) \subset D_k$, then $\phi(x,y) \in D_k$ for all $(x,y)$, so the integral equals $1$ everywhere.  

Conversely, suppose $\int_{D_k} T_{ij}^k(x,y,u)\,du = 1$ for all $(x,y)\in D_i \times D_j$. If there existed $(x^*,y^*)$ with $\phi(x^*,y^*) \notin D_k$, then the corresponding integral would be $0$, a contradiction. Hence $\phi(D_i \times D_j) \subset D_k$. 
\end{proof}

\subsection{Sample Transition Rates}

We present four parametric families as representative choices for $\phi$ drawn from kinetic models of social dynamics, and explain the meaning and expected behavior for each case; these are summarized in Table \ref{tab:phitab}. In Section \ref{sec:results}, we apply our numerical scheme to integrate instances of equation \eqref{eq:integrodiff} where transition rates are defined by our sample $\phi$ functions and then analyze the resulting dynamics. We note that while all sample functions used in this work are bilinear or piecewise bilinear, the approach presented in this paper applies to a wide class of $\phi$ satisfying the conditions in Proposition \ref{prop:deltamasspreserve}. \\

\noindent
{\it Example 1: Lowering interactions (``Left").} The transition function
$$
\phi_L(x,y) = x - \gamma x y,
$$
for $\gamma>0$, represents a situation where an individual’s state $x$ is reduced through interaction with another individual in state $y$.
The term $\gamma x y$ shows that the reduction is stronger when both $x$ and $y$ are large, pulling the new state closer to 0. Figure \ref{fig:trans_left} illustrates this effect.
At the boundary $x = 0$, we have $\phi_L(0,y) = 0$, meaning the state cannot decrease further. At the other boundary, $x = 1$, the update is $\phi_L(1,y) = 1 - \gamma y$, so the decrease grows with larger $y$. The parameter $\gamma$ controls the rate at which an individual’s state goes toward $0$. \\ 


\noindent 
{\it Example 2: Reinforcing interactions (``Right").} The function
$$
\phi_R(x,y) = x + \gamma (1-x)(1-y),
$$
shown in Figure \ref{fig:trans_right}, describes dynamics that depend on how far both $x$ and $y$ are from their maximum value of 1. It represents a reinforcement effect: when both $x$ and $y$ are low, individuals are more likely to shift toward higher states.
At the boundary $x = 0$, we get $\phi_R(0,y) = \gamma (1-y)$. This means that the smaller $y$ is, the more strongly individuals at $x = 0$ move upward. If $y = 1$, the update vanishes, so $x = 0$ stays fixed. At the other boundary, $x = 1$, we have $\phi_R(1,y) = 1$, so no further increase is possible, making $x = 1$ a stable state. Over time, individuals are therefore drawn toward $x = 1$.

\begin{figure}[H]
    \centering
    \subcaptionbox{ $\phi_L$ \label{fig:trans_left}}{\includegraphics[width=.45\textwidth]{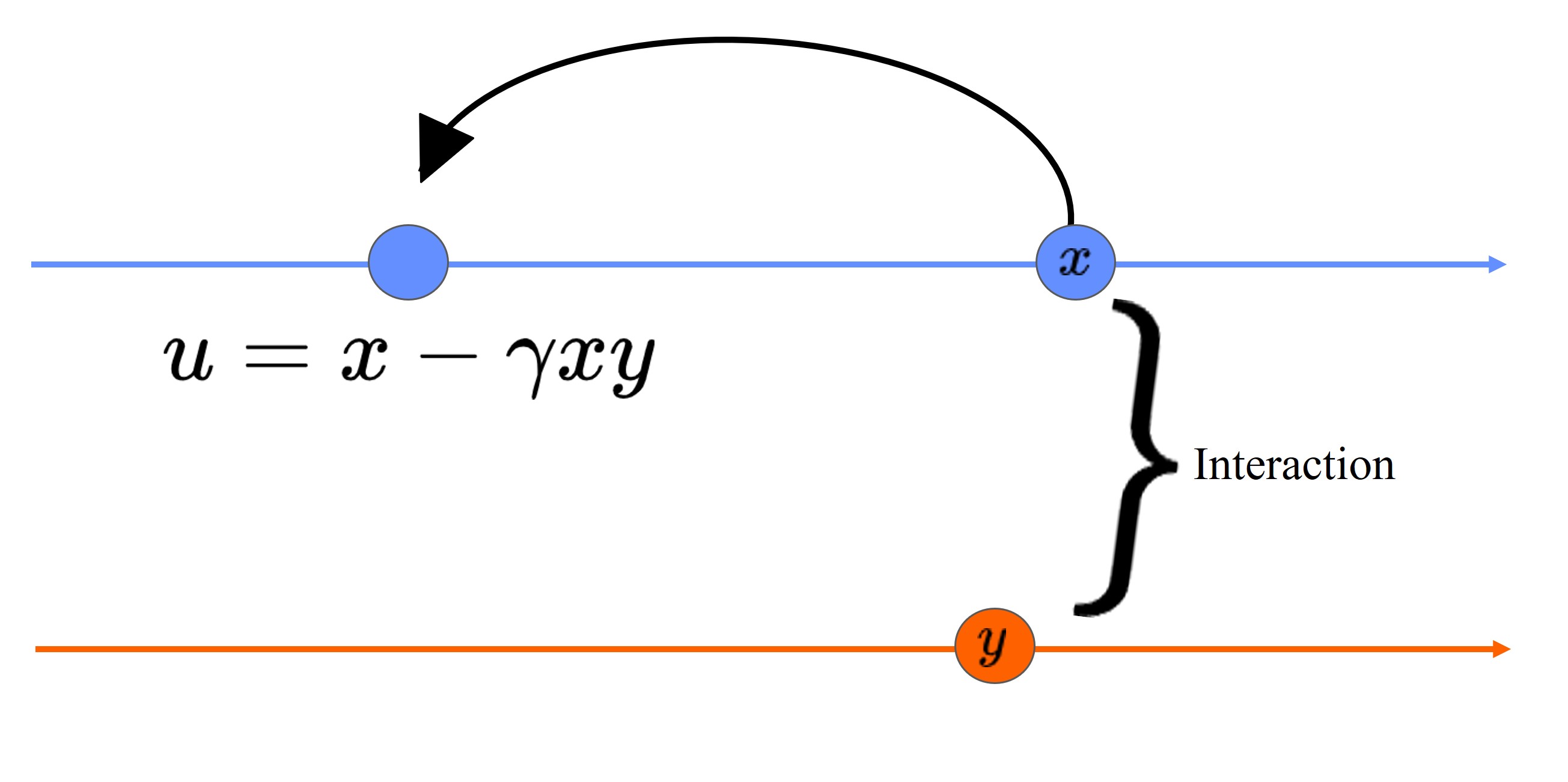}}\hspace{1em}%
    \subcaptionbox{$\phi_R$ \label{fig:trans_right}}{\includegraphics[width=.45\textwidth]{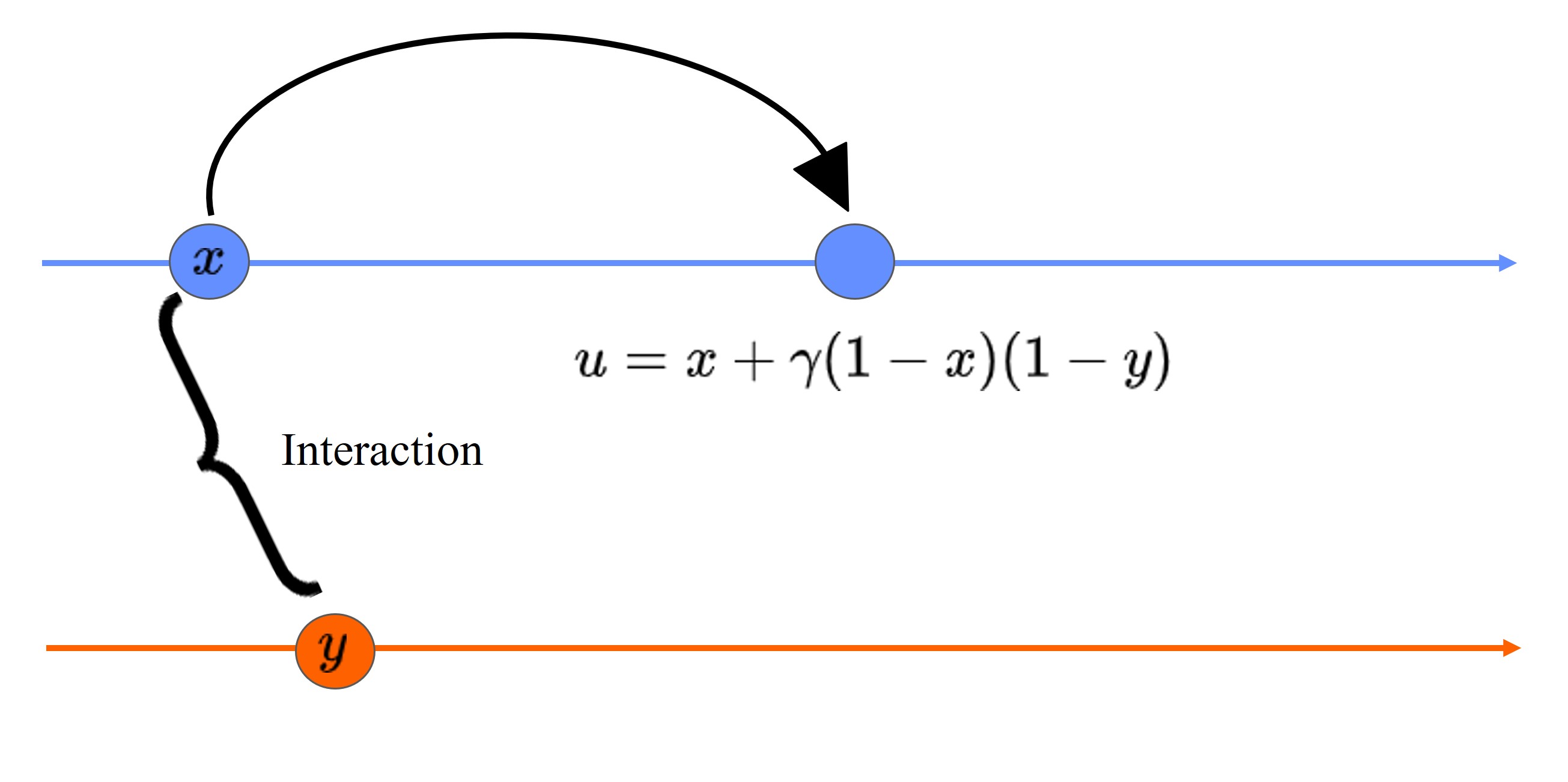}}\hspace{1em}%
    \caption{Illustrations of transition effects based on the first two functions listed in Table \ref{tab:phitab}. Each function models the rate at which an agent in subsystem $i$ with microstate $x$ will transition to a microstate $u$ in subsystem $k$. All transition functions depend on $\gamma$. (a) The transition function $\phi_L$, where an agent in subsystem $i$ with microstate $x$ transitions to a lower microstate $x - \gamma xy$, transitioning closer to 0. (b) A depiction of the transition function $\phi_R$, where an agent in subsystem $i$ with microstate $x$ transitions to a higher microstate $x + \gamma(1 - x)(1 - y)$, transitioning closer to 1.}

    \label{fig:errorruntime2}
\end{figure}

\noindent 
{\it Example 3: Attractive threshold interactions (``Toward").}  Consider the piecewise transition function

$$
\phi_T(x,y) = 
\begin{cases} 
x - \gamma xy, & x \leq a,\\ 
x + \gamma (1-x)(1-y), & x > a,
\end{cases}
$$
illustrated in Figure \ref{fig:trans_awa}. The threshold $a$ determines whether interactions push individuals downward or upward. For $x \leq a$, interactions reduce $x$ through $x - \gamma xy$, so individuals with lower states are pulled toward smaller values when paired with higher $y$. For $x > a$, interactions increase $x$ via $x + \gamma (1-x)(1-y)$, reinforcing higher states. At the boundaries, $\phi_T(0,y) = 0$ and $\phi_T(1,y) = 1$, so states cannot go below 0 or above 1. Over time, individuals move away from the threshold $a$ and are drawn toward the lower or upper boundaries.

{\it Example 4: Repulsive threshold interactions (``Away").} As a final example, consider the piecewise function

$$
\phi_A(x,y) = 
\begin{cases} 
x - \gamma (x-a)y, & x > a,\\ 
x + \gamma (a-x)(1-y), & x \leq a,
\end{cases}
$$

illustrated in Figure \ref{fig:trans_tow}. For $x > a$, interactions with $y$ pull individuals toward lower states, while for $x \leq a$, interactions push them toward higher states.  At the boundaries, $\phi_A(0,y) = \gamma a (1-y)$ moves individuals away from 0, and $\phi_A(1,y) = \gamma (1-a)(1-y)$ pulls them away from 1. Over time, the threshold $a$ determines the dynamics: states above $a$ drift downward, while states at or below $a$ tend to increase.

\begin{figure}[H]
    \centering
    \subcaptionbox{$\phi_T$ \label{fig:trans_tow}}{\includegraphics[width=.45\textwidth]{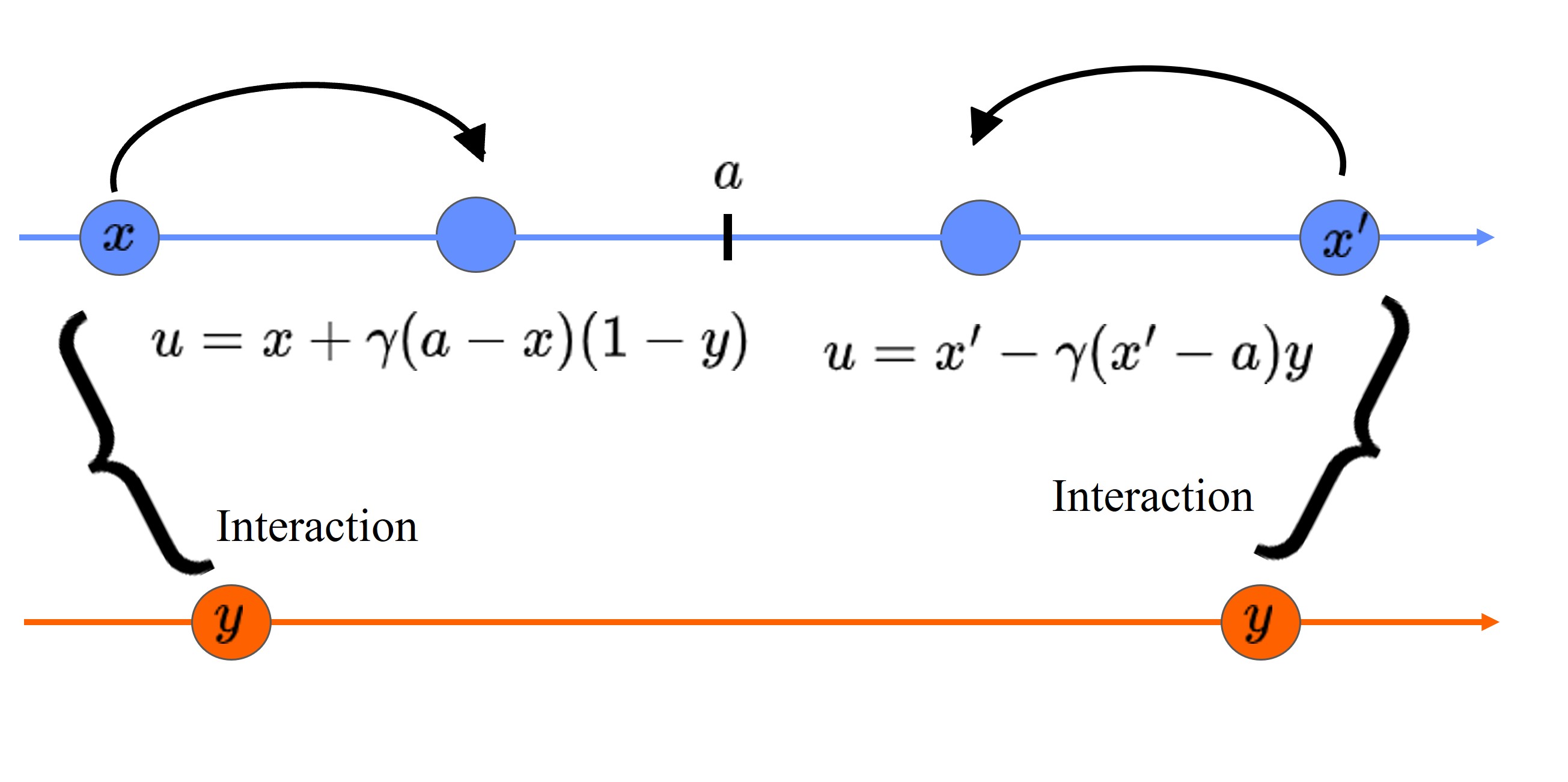}}\hspace{1em}%
    \subcaptionbox{$\phi_A$ \label{fig:trans_awa}}{\includegraphics[width=.45\textwidth]{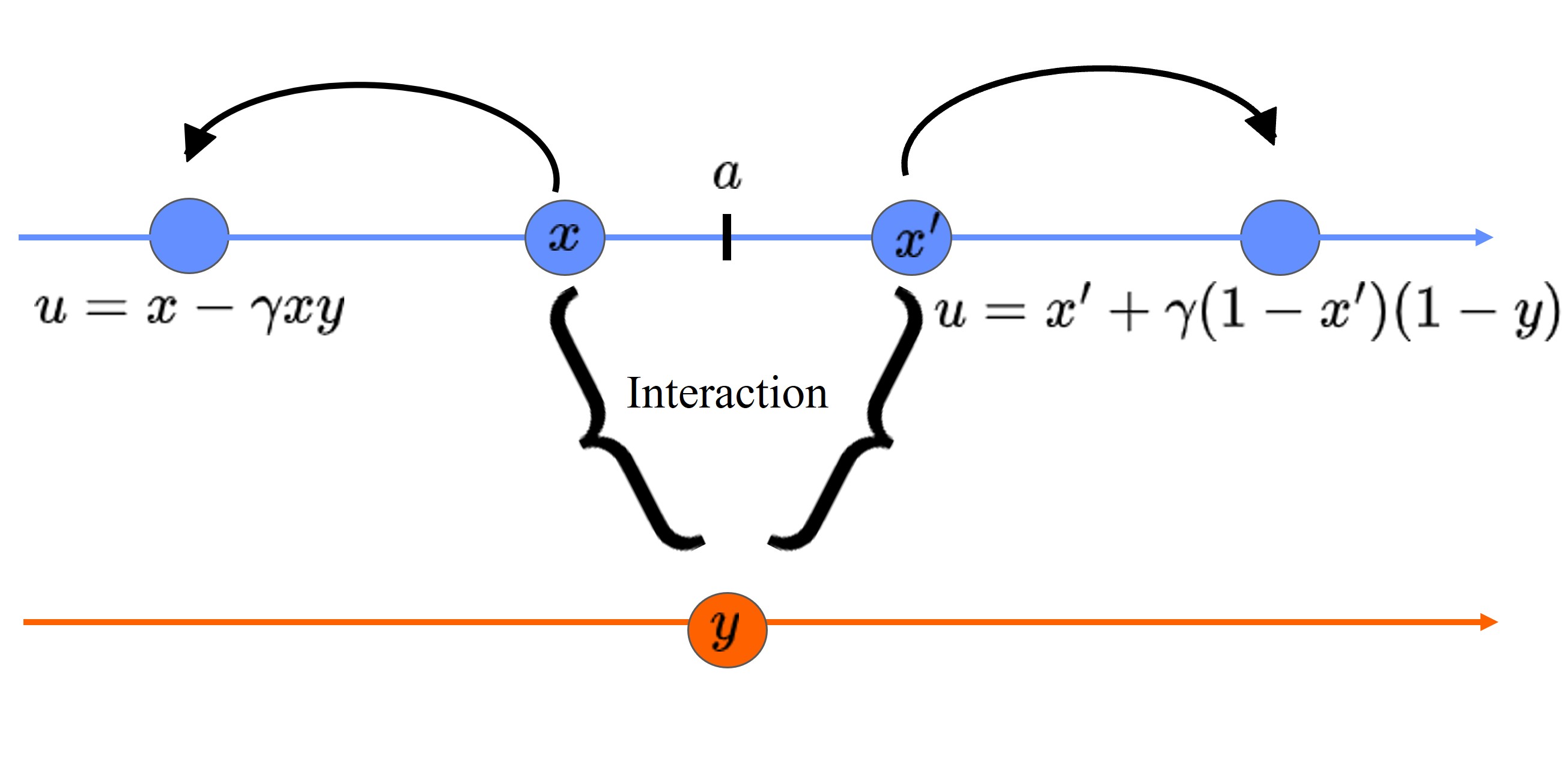}}\hspace{1em}%
    \caption{Illustrations of transition effects based on function $\phi$, which models an agent in subsystem $i$ with microstate $x$ moving to a microstate $u$ in subsystem $k$. For these functions, the direction of the transition depends on threshold $a$. (a) Transition function $\phi_T$: if $x < a$, the agent shifts to the right, and if $x > a$, the agent shifts to the left. In both cases, the agents move toward $a$. (b) Transition function $\phi_A$: if $x < a$, the agent shifts to the left, and if $x > a$, the agent shifts to the right. In both cases, the agents move away from $a$, toward $1$ if $x > a$ and toward $0$ if $x < a$.}
    \label{fig:errorruntime2}
\end{figure}
\begin{table}[!htbp]
    
    \begin{tabularx}{\linewidth}{|X|X|}
        \toprule
        Transition function $\phi$ & Interpretation of movement \\
        \midrule
        \midrule
        {\scriptsize$\phi_L(x,y) = x - \gamma xy$} & Agents in micro-state $x$ will transition to the \textbf{left} upon interacting with agents in micro-state $y$. \\
        \midrule
        {\scriptsize$\phi_R(x,y) = x + \gamma(1-x)(1-y)$} & All agents in micro-state $x$ will transition to the \textbf{right} upon interacting with agents in micro-state $y$. \\
        \midrule
        {\scriptsize $\phi_A(x,y) = \begin{cases} x - \gamma xy,& \text{if } x \leq a,\\ x + \gamma (1-x)(1-y), & \text{if } x > a \end{cases}$ }& Agents in micro-states to the left and right of constant $a$ transition \textbf{away from} $a$.\\
        \midrule
        {\scriptsize $\phi_T(x,y) = \begin{cases} x - \gamma (x-a)y, & \text{if } x > a,\\ x + \gamma (a-x)(1-y), & \text{if } x \leq a \end{cases}$ }& Agents in micro-states to the left and right of some constant $a$ transition \textbf{towards} $a$.\\
        \bottomrule
    \end{tabularx}
   \centering
    \caption{Four examples of transition functions and their physical interpretation. Each function represents a specific transition mechanism. The change is proportional to $\gamma.$}
    \label{tab:phitab}
\end{table}

\section{A Mass Preserving Scheme for Kinetic Equations Integration}\label{sec:scheme}

In this section, we develop a mass-preserving numerical scheme tailored to accurate integration of kinetic systems such as \eqref{eq:integrodiff}, with transition rates as defined in Section \ref{sec:transitions}. Given that spatial discretization of \eqref{eq:integrodiff} involves integrals defined by the encounter and transition rate kernels, the latter of which encode sharp transitions via an indicator function for the pre-image of $\phi$, we must choose a numerical integration technique with this in mind. Approaches typically used in the discretization of integral equations include the Nystrom method (numerical quadrature), as well as the Galerkin and collocation methods, which approximate the solution within a finite-dimensional subspace, such as polynomials, splines, or trigonometric functions \cite{hughes2003finite}. 

In this work, we develop a collocation method; this discretization scheme asks that the equation be satisfied by all functions in a finite-dimensional space at a specified set of collocation points. As compared to a Nystrom approach, collocation and Galerkin schemes avoid the potential need for specialized quadrature, instead requiring the precomputation of integrals for a chosen basis of the finite-dimensional subspaces. Collocation methods are generally easier to implement and can be more computationally efficient than Galerkin; they can, however, become unstable if the chosen points miss key features of the solution \cite{canuto2007spectral}. We note that the approach presented below could be readily adapted to a wide array of Galerkin schemes; however, we leave this development for future work.

\subsection{Spatial discretization}

We discretize integrals in equation \eqref{eq:integrodiff} by approximating unknown density functions $\{f_i\}_{i=1}^n$ for each functional subsystem with polynomials of degree $\leq N$ over their respective microstate domains $\{D_i\}_{i=1}^n$. Given that our collocation scheme enforces these equations at a set of collocation points, we opt to use the corresponding Lagrange interpolation basis $\{B_j^i\}_{j=0}^N$; our approximation to the density function takes the form 

$$f_i(t,u) = \sum_{j=0}^N a_{ij}(t) B_j^i(u), \quad i\in \{1,2,\cdots,n\}.$$

This transforms our problem from finding continuous functions $f_i$ to determining finite sets of time-dependent coefficients $a_{ij}(t)$. Substituting this representation into equation \eqref{eq:integrodiff} yields 
\begin{align}\label{NumIntDiff}
a_{ij}' &= \sum_{l,m}^n\sum_{p,q}^Na_{lp}a_{mq}\int_{D_l} \int_{D_m} \eta_{lm}(x,y)T_{lm}^i(x,y,u)B_p^l(x)B_q^m(y)dxdy \notag\\
&-  \sum_p^N a_{ip}B_p^i(u)\sum_l^n\sum_q^Na_{lq}\int_{D_l} \eta_{il}(u,y)B_q^l(y)dy, \quad \forall i \in \{1,...,n\} .
\end{align}

We note that using the Lagrange basis simplifies the left-hand side of this equation (generally $\sum_j^N a_{ij}'B_j^i(u)$). All interpolants and basis terms are evaluated using barycentric formulas to ensure computational efficiency and stability.

\subsection{Ensuring mass conservation}
Given a set of collocation points $\{x_j^i\}_{j=0}^N$ over $D_i$, Equation \eqref{NumIntDiff} thus reduces to a system of first-order ODEs. In our experiments, however, we find that integrating this system of ODEs can lead to considerable violations of mass conservation. \ref{app:naivefail} demonstrates that the approximation error can lead to substantial mass growth, contradicting the mass conservation established in Section \ref{sec:micscale}. As mentioned in Section \ref{sec:KE}, these violations render approximate solutions practically useless, as they are nonphysical.

To address this issue, we modify our scheme so that it satisfies a numerical mass preservation property while remaining convergent. We achieve this by adding a corrective term. Given an initial state $a_{ij}(0) = f_i^0(x_j^i)$ for some $f_i^0 \in L_1(D_i)$, the proposed scheme takes the following form
\begin{align}\label{ODESYS}
 a_{ij}'&= \sum_{l,m}^n\sum_{p,q}^N T_{ilmj}^{pq} a_{lp}a_{mq} \\
 &-  C(\{a_{ij}\})a_{ij}\sum_l^n\sum_{q}^N M_{ijlq} a_{lq},\quad \forall i \in \{1,...,n\},\ \forall j \in {0,1,\cdots,N} ,\notag
\end{align}
where 
\begin{align}\label{eqn:coeff}
\left\{\begin{array}{ll}
 T_{ilmj}^{pq} &= \mathlarger{\int}_{D_l} \mathlarger{\int}_{D_m} \eta_{lm}(x,y)T_{lm}^i(x,y,x_j)B_p^l(x)B_q^m(y)dxdy,\vspace{8pt}\\
    M_{ijlq} &= \mathlarger{\int}_{D_l} \eta_{il}(x_j,y)B_q^l(y)dy,
    \end{array}\right.
\end{align}
and
\begin{align}\label{eq:caij}
    C(\{a_{ij}\}) &= \frac{\sum\limits_{i}^{n} \sum\limits_j^N w_{ij} \sum\limits_{l,m}^{n} \sum\limits_{p,q}^N T_{ilmj}^{pq} a_{lp}a_{mq}}{\sum\limits_i^n \sum\limits_j^N  a_{ij}w_{ij}\sum\limits_l^n\sum\limits_q^N M_{ijlq} a_{lq}},
\end{align}
where $w_{ij}=\int_{D_i} B_j^i(u)du.$
Note that the term $C(\{a_{ij}\})$ serves as a correction factor. At each time step, it compensates for the error in total mass introduced by the approximation, ensuring that our solutions satisfy an exact discrete conservation property. 
These terms avoid singularities except for the trivial case where there are no interactions. Furthermore, as the number of collocation points increases, $C(\{a_{ij}\})$ approaches one. This behavior is illustrated in Figure \ref{fig:errmassmpcs} in \ref{app:naivefail}.
\begin{prop} 
Consider the discrete mass functional
$$
\mathcal{M}(\{a\})(t)=\sum_i\sum_j a_{ij}(t)w_{ij}. 
$$
The scheme presented in equation \eqref{ODESYS} with coefficients $ T_{ilmj}^{pq}$ and $M_{ijlq}$ defined in \eqref{eqn:coeff} and $C(\{a_{i,j}\})$ defined in \eqref{eq:caij} numerically preserves the mass, 
    $$
    \frac{d}{dt}\mathcal{M}(\{a\})(t) = 0.
    $$
\end{prop}
\begin{proof}
Computing the time derivative of $\mathcal{M}(\{a\})$ we get
\begin{align*}
    \frac{d}{dt}\mathcal{M}(\{a\})(t)&=\sum_{i}^n \sum_j^N a_{ij}'w_{ij} \\
    &= \sum_{i}^n \Bigg[\sum_{l,m}^n\sum_{p,q}^N T_{ilmj}^{pq} a_{lp}a_{mq}
    -  C(\{a_{ij}\})a_{ij}\sum_l^n\sum_{q}^N M_{ijlq} a_{lq} \Bigg]w_{ij}.
    \end{align*}
    Note that $$\sum_i^n  C(\{a_{ij}\})a_{ij} w_{ij} = \frac{B}{A},$$ where
    \begin{align*}
       A := \sum_i^n \sum_j^N w_{ij} \sum_{l,m}^n \sum_{p,q}^N T_{ilmj}^{pq} a_{jp}a_{lq}
       \;\;\text{and}\;
       B :=  \sum_i^N \sum_j^N  a_{ij}w_{ij}\sum_l^n\sum_q^N M_{ijlq} a_{lq}.
    \end{align*}
Thus,
\begin{align*}
 \frac{d}{dt}\mathcal{M}(\{a\})(t) &= \frac{1}{A}\bigg(1 - \frac{B}{A} \frac{A}{B}\bigg)= 0.
\end{align*}
Thus, we conclude that the mass is numerically preserved.

\end{proof}

A full proof of convergence for this scheme is left for future work. However, in Section \ref{sec:results}, we show strong evidence of the robustness, efficiency, and ability of the method to capture the qualitative dynamics of the transition rates in Section \ref{sec:transitions}.
In summary, we introduce a discretization method for integro-differential kinetic equations, which we call the {\bf Mass-Preserving Collocation Method (MPCM)}. A pseudocode description is given in Algorithm \ref{algo:coloc}. This method consists of two main steps:

\begin{itemize}
    \item A \emph{precomputation step} in which we compute coefficient arrays $T_{ilmj}^{pq}$ and $M_{ijlq}$ to high accuracy and store them. This involves the use of analytical solutions or numerical quadrature to compute integrals involving polynomial basis functions. This can be computationally expensive if numerical integration is required, but needs to be performed only once for a given choice of interaction and transition kernels, exploiting symmetries to reduce storage requirements. We note that the maximum storage requirements for these coefficient arrays are $\mathcal{O}(n^3 N^3)$; dependency on $n$ is likely to be lower due to the sparsity of subsystem connectivity networks, such as those shown in Figure \ref{fig:sys5}. 
    
    \item A \emph{solve step}, in which we set up an initial value problem of the form $A'(t) = F(A(t))$ with $A(t)$ a matrix with entries $a_{ij}(t)$ given initial conditions. The IVP can then be integrated in time using a chosen numerical ODE solver. 
\end{itemize} 
{\tiny
\begin{algorithm}[H]
\caption{Mass Preserving Collocation Method (MPCM)}\label{algo:coloc}
\begin{algorithmic}[1]
\Require $n, N$: number of subdomains and basis functions
\Require $\{D_i\}_{i=1}^n$: domain partitions
\Require $\{x_j^i\}_{j=0}^N, \{B_j^i\}_{j=0}^N$: collocation nodes and Lagrange basis
\Require $\eta_{lm}(x,y), T_{lm}^i(x,y,u)$: interaction kernels
\Require Initial coefficients $a_{ij}(0) = f_i^0(x_j^i)$
\Require Time step $dt$, final time $T$
\Ensure Approximate solution $f_i(t,u) \approx \sum_{j=0}^N a_{ij}(t)B_j^i(u)$

\State \textbf{Precomputation (via Gaussian quadrature):}
\State Compute
\[
T_{ilmj}^{pq} = \int_{D_l}\int_{D_m} \eta_{lm}(x,y)T_{lm}^i(x,y,x_j) B_p^l(x)B_q^m(y)\,dx\,dy
\]
\[
M_{ijlq} = \int_{D_l} \eta_{il}(x_j,y) B_q^l(y)\,dy,
\quad
w_{ij} = \int_{D_i} B_j^i(u)\,du
\]

\State \textbf{Time-stepping:} For $t=0:dt:T$
\Statex \quad Compute correction factor
\[
A = \sum_{i,j} w_{ij}\Big(\sum_{l,m,p,q} T_{ilmj}^{pq} a_{lp}a_{mq}\Big),
\;
B = \sum_{i,j} a_{ij} w_{ij}\Big(\sum_{l,q} M_{ijlq}a_{lq}\Big),
\;
C = A/B
\]
\Statex \quad For each $i,j$ update coefficient
\[
a_{ij}' = \sum_{l,m,p,q} T_{ilmj}^{pq} a_{lp}a_{mq}
- C \cdot a_{ij}\Big(\sum_{l,q} M_{ijlq}a_{lq}\Big)
\]
\Statex \quad Advance $a_{ij}(t+dt)$ using an ODE solver (e.g. Runge--Kutta).

\State \textbf{Output:} Approximate densities $f_i(t,u)$ reconstructed from coefficients.
\end{algorithmic}
\end{algorithm}
}

 \section{Numerical Experiments}\label{sec:results}
In this section, we conduct a series of experiments to study the performance of the MPCM. First, we test the ability of the method to resolve the dynamics of the subsystem for representative transition rate models, as described in Section \ref{sec:transitions}. In this initial set of experiments, our goals are to study the convergence properties and computational costs involved, as well as to assess how well the resulting approximate solutions capture the qualitative behavior expected of the solutions determined by moment analysis (see \ref{app:asymp}). 

To further validate our numerical method and demonstrate its utility in resolving interactions within large systems, we then compare its performance with what are arguably the state-of-the-art approaches in stochastic modeling of multi-particle interactions: the Tau-leaping method and a hybrid method. In these comparisons, we make choices specific to each algorithm to achieve similar high accuracy and mass preservation; due to the stochastic nature of these methods, this involves averaging over repeated realizations. We find that MPCM consistently matches subsystem dynamics obtained by the competing methods, is orders of magnitude faster, and requires significantly less parameter tuning to ensure mass preservation and accurate results. 

Finally, we test MPCM on a larger, more complex system involving five subsystems and six transition rates. Our results show that MPCM is an efficient and reliable tool for studying kinetic equations in social systems.

\paragraph{Experimental setup and metrics} We use Chebyshev collocation points of the first kind to obtain the equation for the coefficients $\{a_{ij}\}$. As the initial value problems (IVPs) obtained from the examples in this work are not noticeably stiff, we use an explicit fourth-order Runge-Kutta method to integrate the IVP for coefficients $a_{ij}$. To empirically monitor convergence, we introduce a self-convergence metric, which measures the difference between solutions computed at two levels of mesh refinement,  $N$ and $2N$, evaluated on a common set of discretization points. Specifically, we define the {\it self-convergence metric} as $\norm{f^{2N} - f^N}$, where 
\begin{align*}
    \norm{f} = \max_{t\in[0,T_f]} \sum_i \int_0^1 |f_i(t,u)| du,
\end{align*}
with $T_f$ denoting the final time. To compute this metric, we evaluate both $f^{2N}$ and $f^N$ at the same $N$-point Gaussian quadrature nodes for all time steps. The integrals are then approximated using Gaussian quadrature, and the maximum over time is taken to yield the final value of the self-convergence metric.

\subsection{Kinetic Models of Two Subsystems}\label{sec:KMresults}

We assess the performance of our scheme on kinetic systems comprised of two subsystems. Each subsystem is represented by a density function $f_1$ and $f_2$, defined on the microstate domain $[0,1]$. We run two rounds of experiments with this setup in order of increasing complexity. First, we only allow one subsystem's density to evolve in time, requiring a single transition rate to describe how it changes as a function of interactions with a second subsystem. We then perform tests that involve sets of two evolving subsystems, requiring two distinct transition rates. 

\paragraph{Tests with one dynamic subsystem} As stated above, $f_2$ is first assumed to be constant in time; we set the encounter rate $\eta_{12} = 1$. We then carry out four experiments, setting the transition rate for subsystem $1$ (due to its interaction with subsystem $2$) $T_{12}^1(x,y,u) = \delta_{\phi(x,y) - u}$, where $\phi$ is chosen as a member of each of the four parametric families representative of models of social dynamics, as listed in Table \ref{tab:phitab}. The resulting equation is
\begin{align}\label{EQ:oneparticles}
    \partial_t f_1(t,u) &= \int_0^1 \int_0^1 \delta_{\phi(x,y) - u}f_1(t,x)f_2(y)\;dxdy - f_1(t,u)\int_0^1 f_2(y)dy.
\end{align}
As illustrated in Figure \ref{fig:simonepart}, our numerical scheme accurately reproduces expected behavior across all transition functions $\phi$ described in Table \ref{tab:phitab} (see \ref{app:asymp}). Each row in this figure shows snapshots of the evolution of $f_1$ corresponding to each transition function $\phi$. Note that, starting from the same initial state, $\phi_L$ shifts microstate density to the left, $\phi_R$ shifts them to the right, $\phi_A$ pulls them away from value $a$, and $\phi_T$ pulls them toward $a$.
In \ref{app:diffgam}, we study the effect of $\gamma$ on the dynamics of the solution; as shown in Figure \ref{fig:simonepartdiffg}, $\gamma$ is directly proportional to the corresponding rate of transition of density moments, as reflected by asymptotic rates in \ref{app:asymp}. The consistent match between expected qualitative and quantitative behavior for these systems and the density trajectory integrated by MPCM validates our approach. 

In Figure \ref{fig:err}, we observe that the self-convergence metric converges to zero as the number of nodes $N$ grows across experiments at an increasing rate, suggesting spectral convergence. Figure \ref{fig:runtime1} shows total runtime as a function of $N$, which is consistently observed to be $O(N^3)$. We note that across simulations, the precomputation cost accounts for $70-80\%$ of the runtime and is responsible for the observed scaling; this is mostly due to computing the tensors $T_{ilmj}^{pq}$. For a given set of transition and interaction kernels, the precomputation step only needs to be done once. After that, multiple simulations can be run with different initial conditions. The computational cost and scaling of each solve then depend on the chosen integrator and corresponding parameters, such as timestep size or target accuracy.

\begin{figure}[H]
    \centering
    \begin{subfigure}{0.22\textwidth}
        \includegraphics[width=\textwidth]{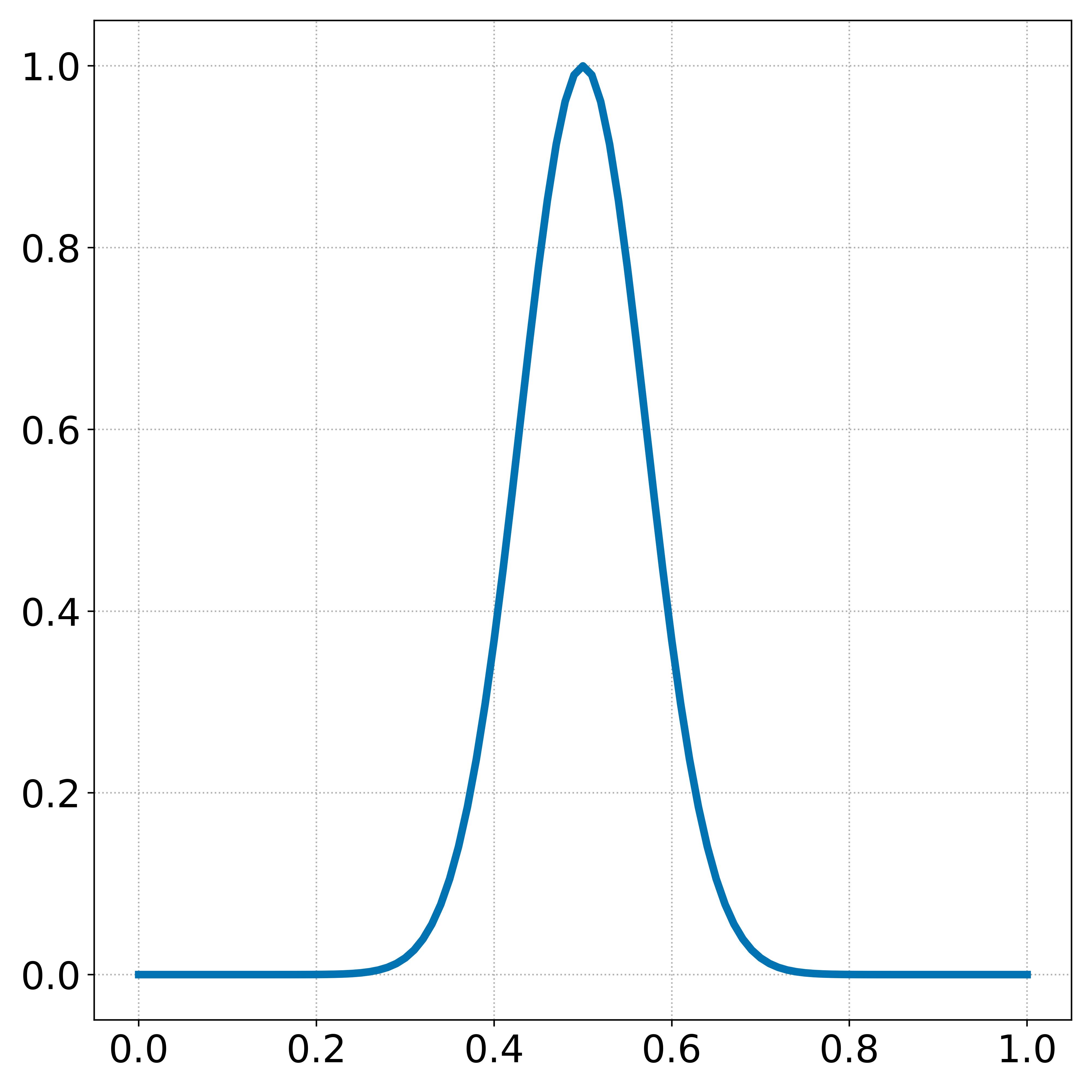}
          \caption{$\phi_L$ at $t = 0$} \label{fig:phiLt0}
          \end{subfigure}
        \begin{subfigure}{0.22\textwidth}
        \includegraphics[width=\textwidth]{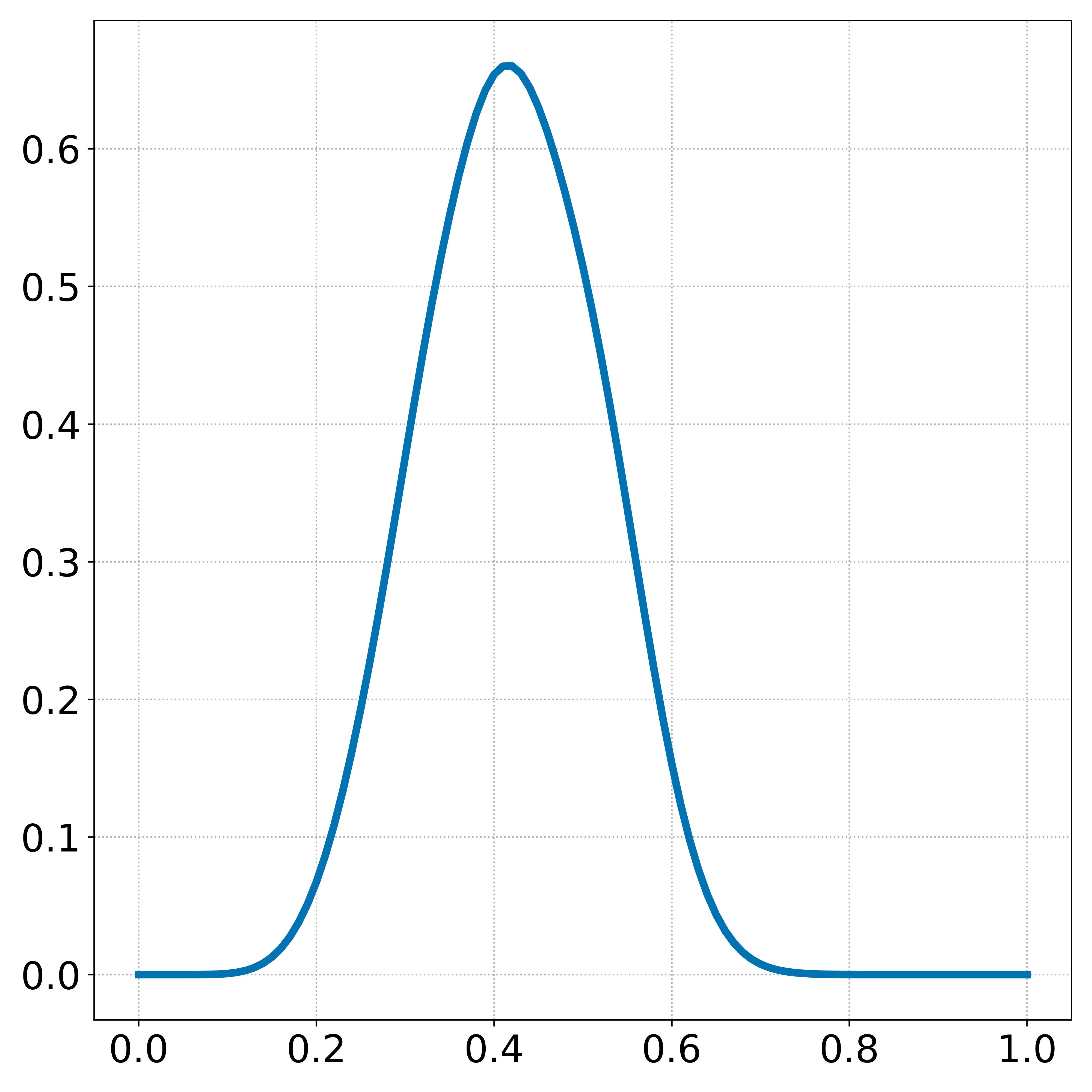}
          \caption{$\phi_L$ at $t = 5$}\label{fig:phiLt5}
          \end{subfigure}
        \begin{subfigure}{0.22\textwidth}
        \includegraphics[width=\textwidth]{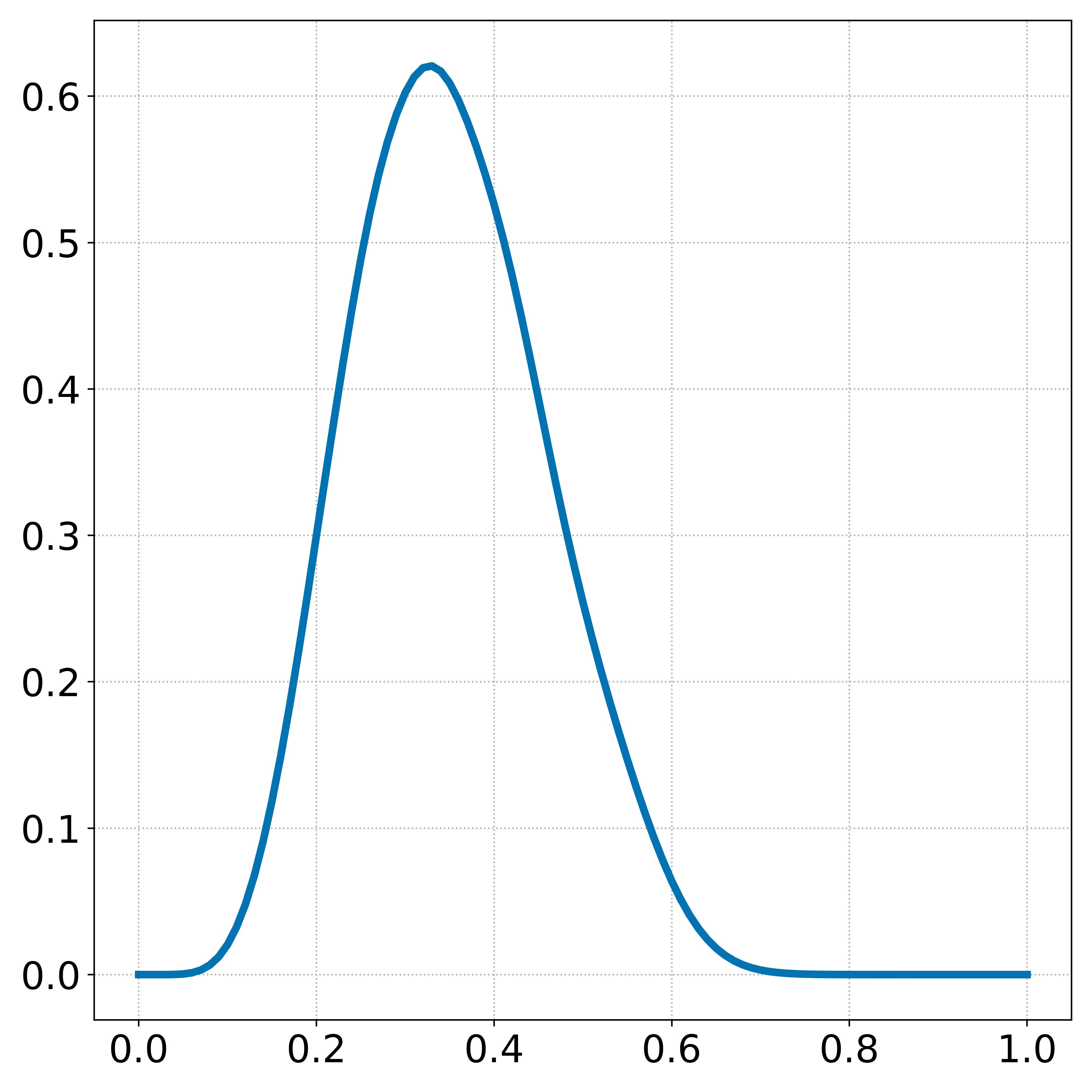}
          \caption{$\phi_L$ at $t = 10$}\label{fig:phiLt10}
          \end{subfigure}
          \begin{subfigure}{0.22\textwidth}
        \includegraphics[width=\textwidth]{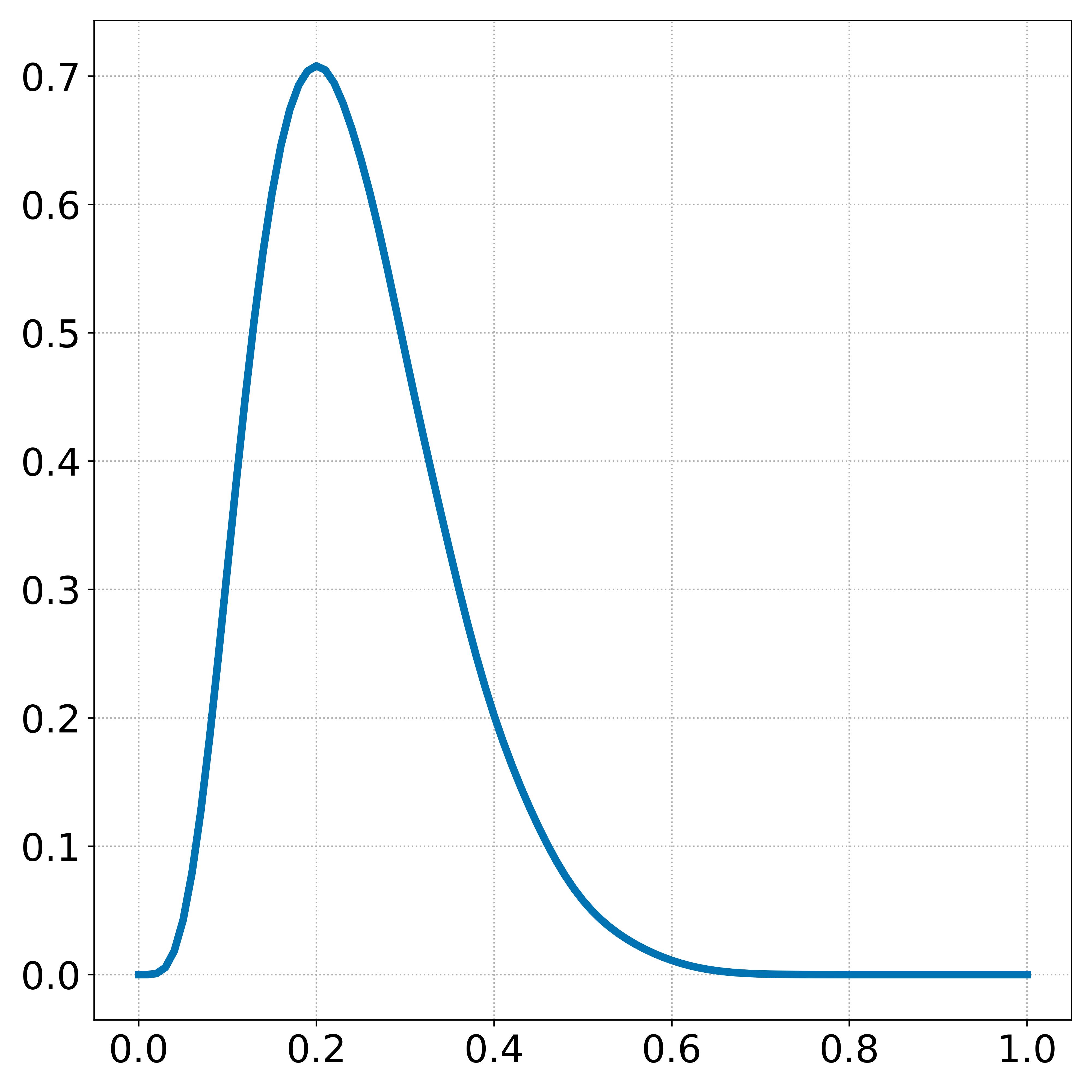}
          \caption{$\phi_L$ at $t = 20$}\label{fig:sphiLt20}
          \end{subfigure}\\
          %
           \begin{subfigure}{0.22\textwidth}
        \includegraphics[width=\textwidth]{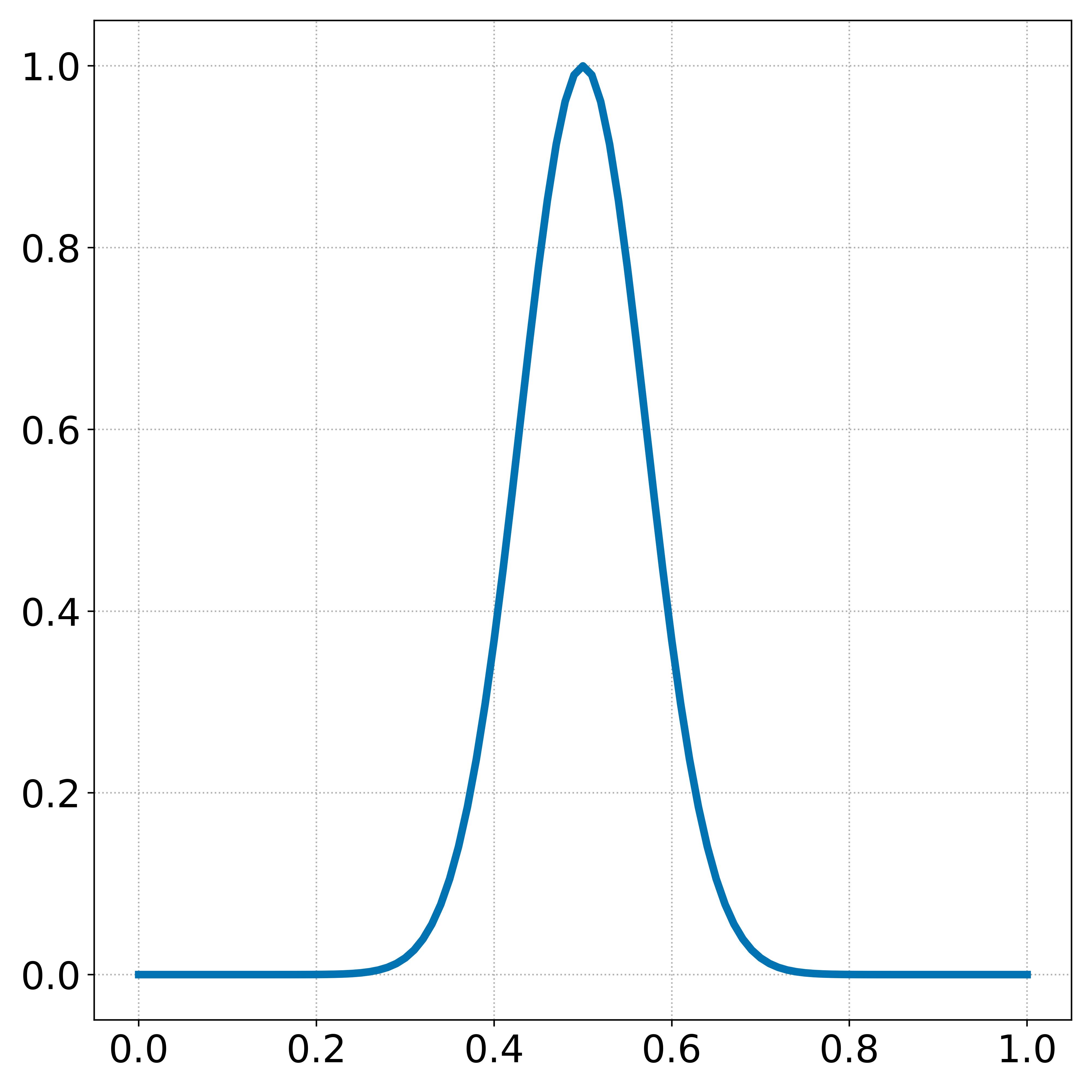}
        \caption{$\phi_R$ at $t = 0$}
        \label{fig:phiRt0}
          \end{subfigure}
        \begin{subfigure}{0.22\textwidth}
        \includegraphics[width=\textwidth]{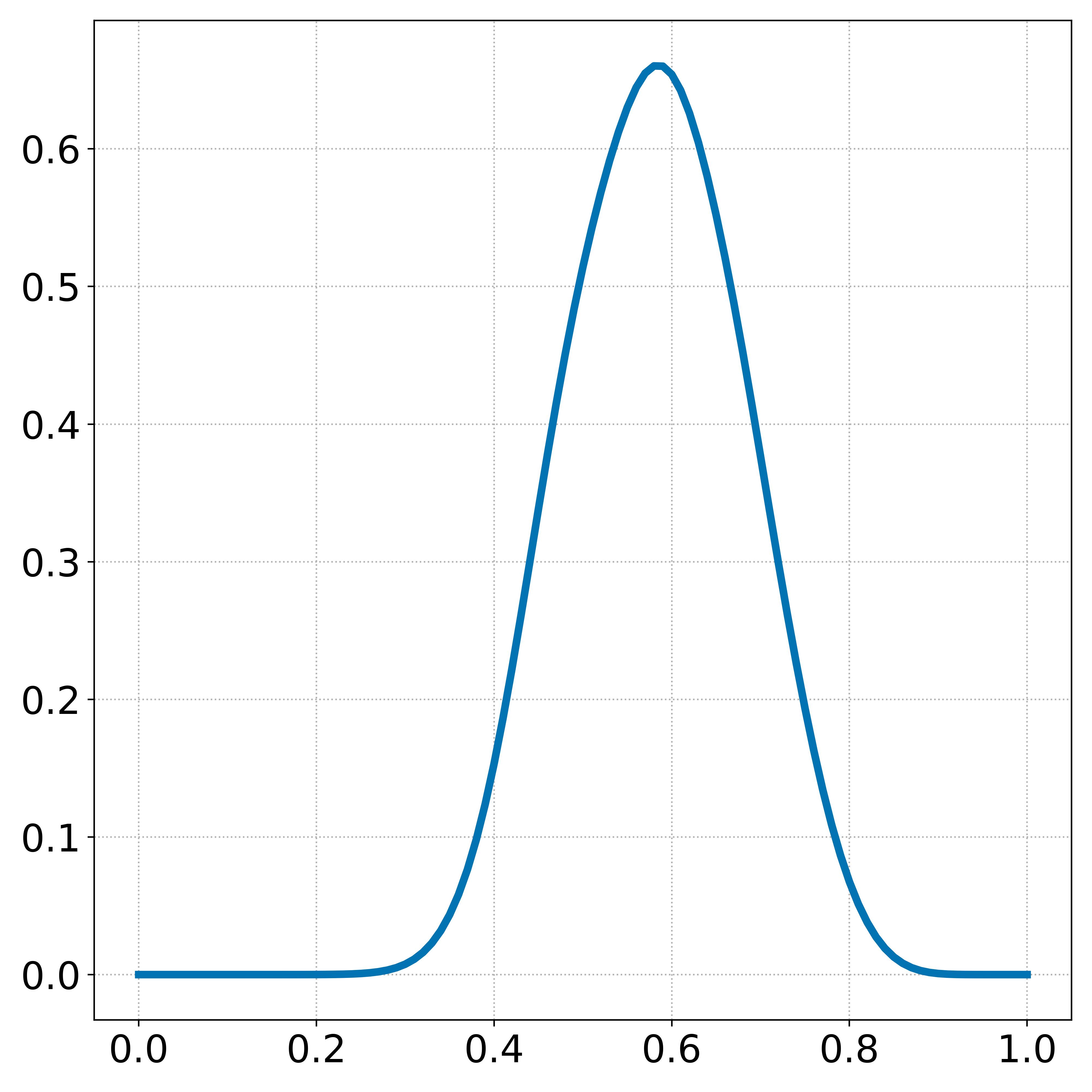}
          \caption{$\phi_R$ at $t = 5$}\label{fig:phiRt5}
          \end{subfigure}
        \begin{subfigure}{0.22\textwidth}
        \includegraphics[width=\textwidth]{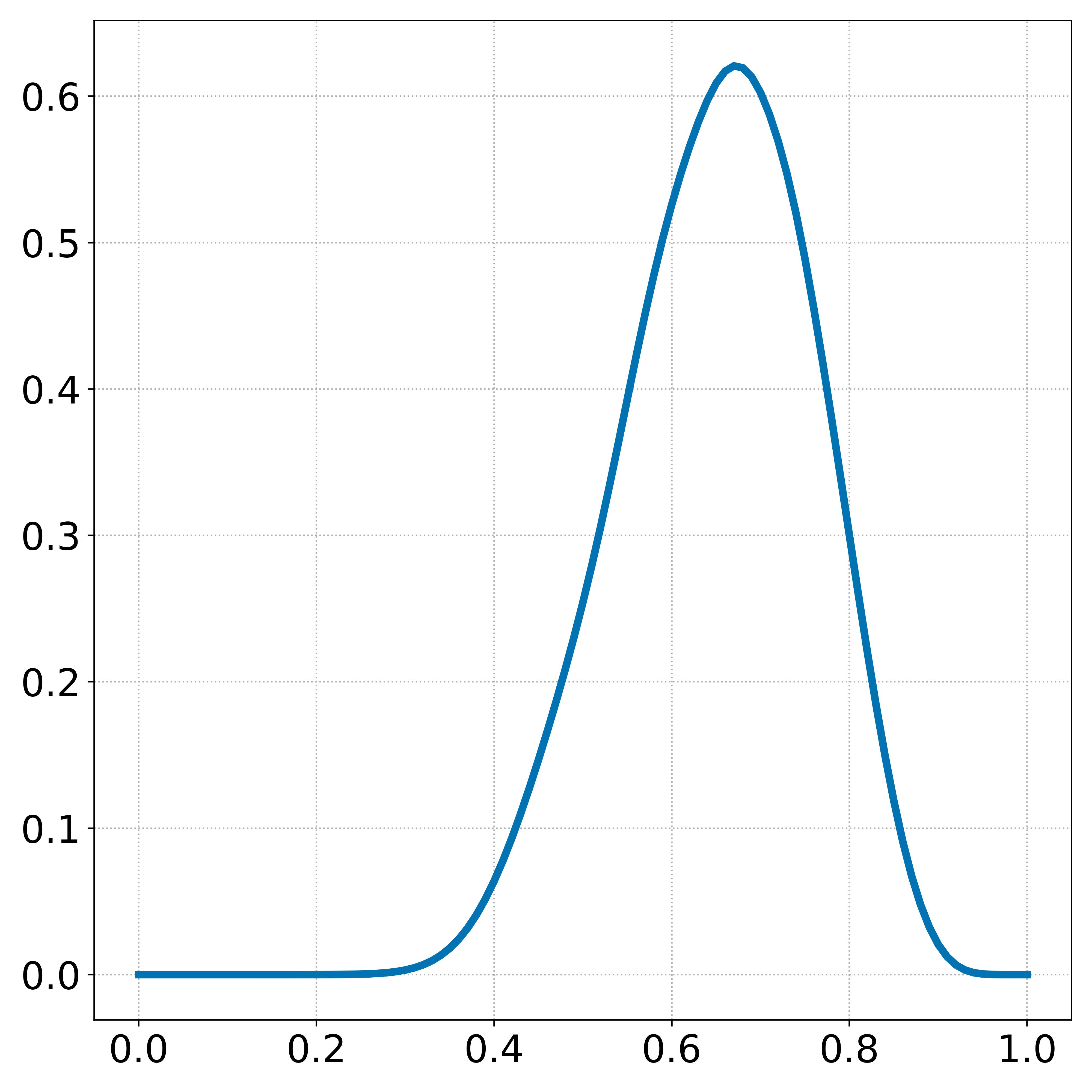}
          \caption{$\phi_R$ at $t = 10$}\label{fig:phiRt10}
          \end{subfigure}
          \begin{subfigure}{0.22\textwidth}
        \includegraphics[width=\textwidth]{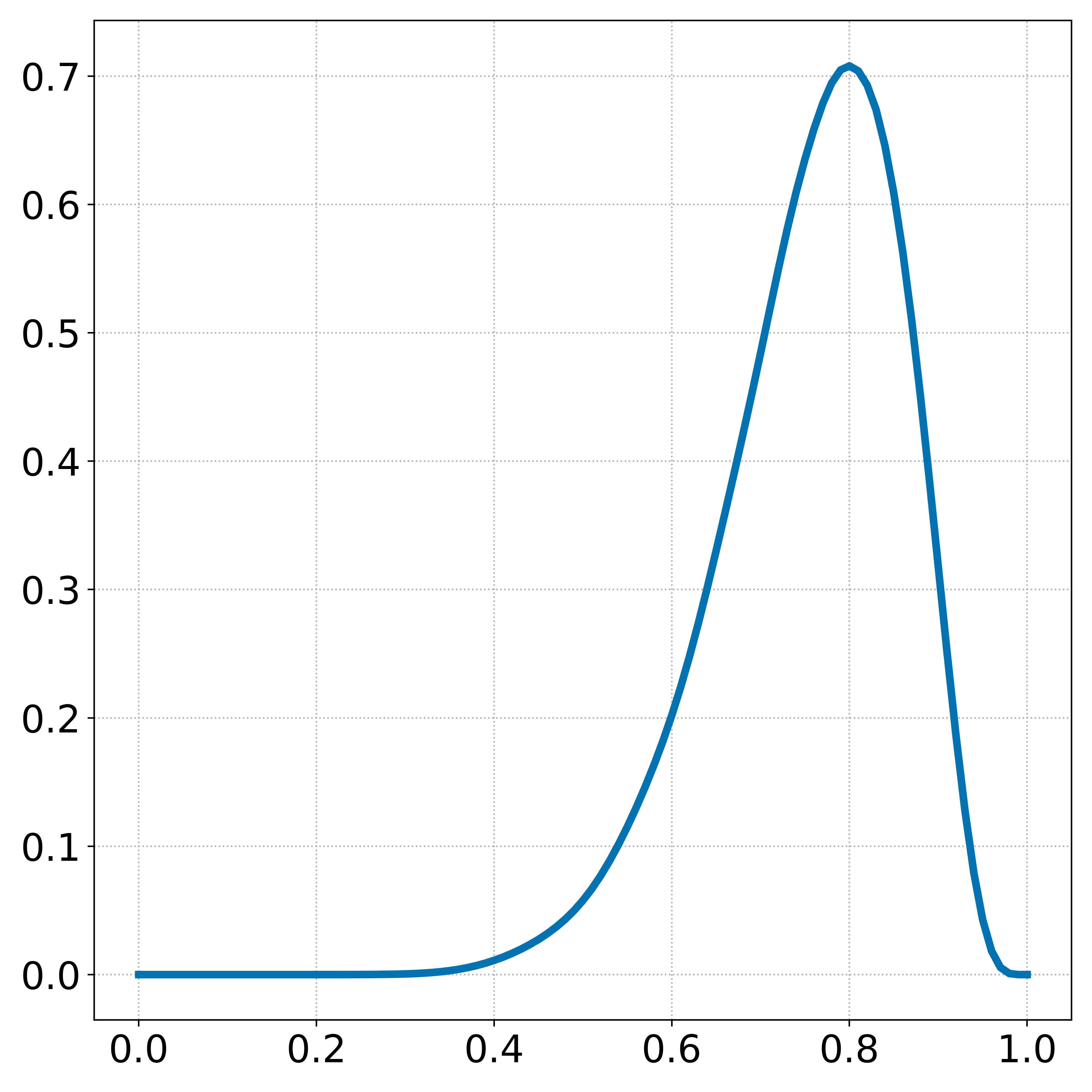}
          \caption{$\phi_R$ at $t = 20$}\label{fig:phiRt20}
          \end{subfigure}\\
           \begin{subfigure}{0.22\textwidth}
        \includegraphics[width=\textwidth]{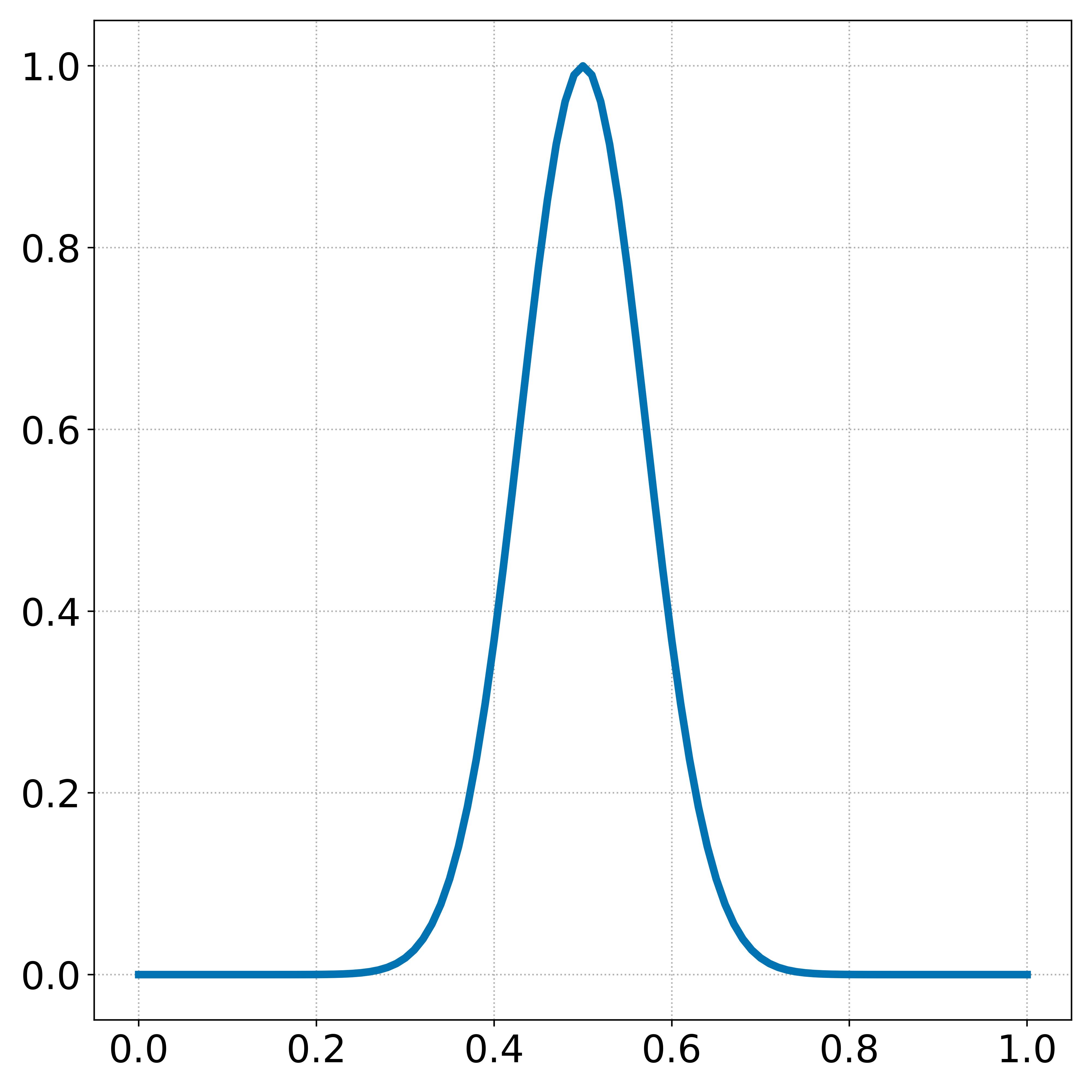}
          \caption{$\phi_A$ at $t = 0$}\label{fig:phiAt0}
          \end{subfigure}
        \begin{subfigure}{0.22\textwidth}
        \includegraphics[width=\textwidth]{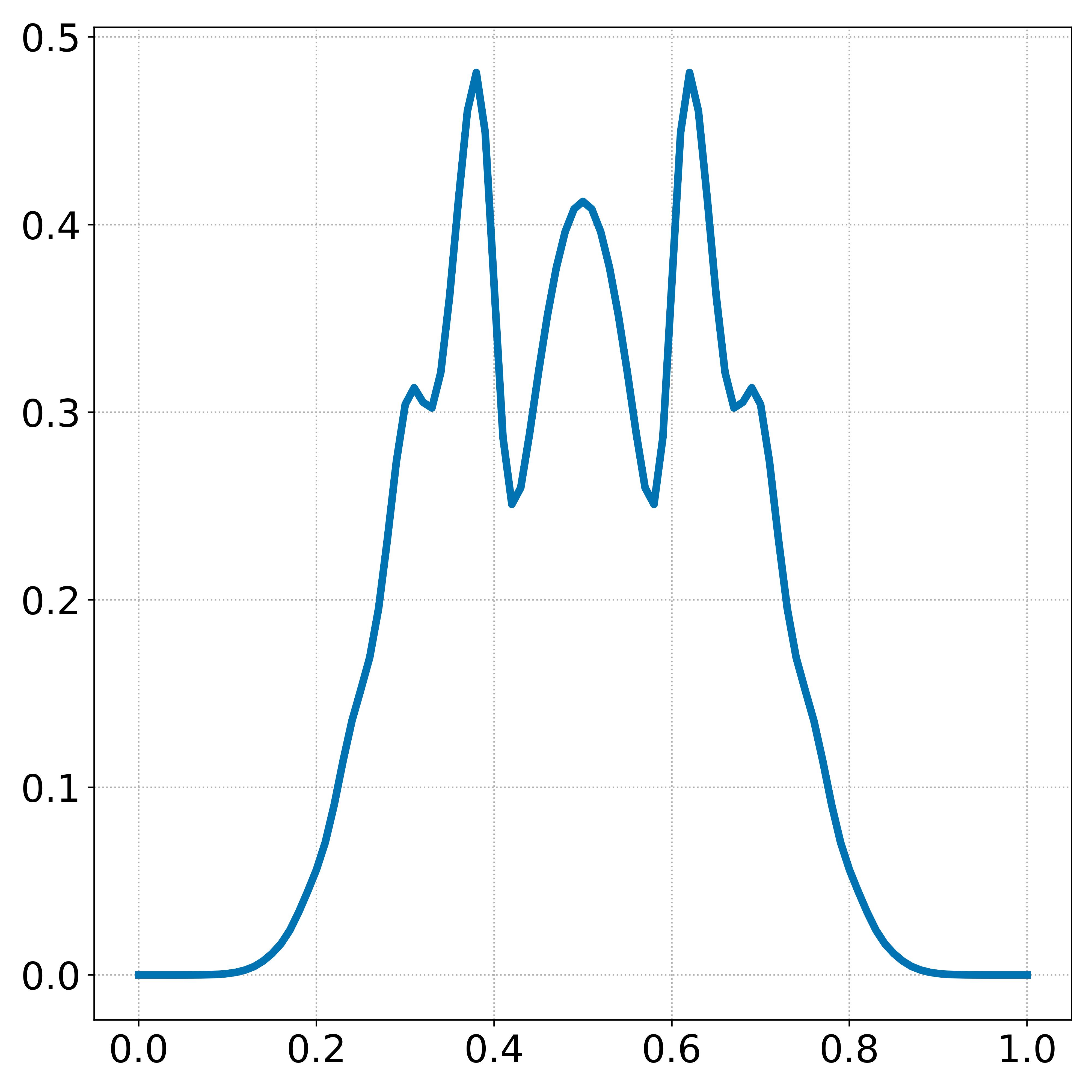}
          \caption{$\phi_A$ at $t = 5$}\label{fig:phiAt5}
          \end{subfigure}
        \begin{subfigure}{0.22\textwidth}
        \includegraphics[width=\textwidth]{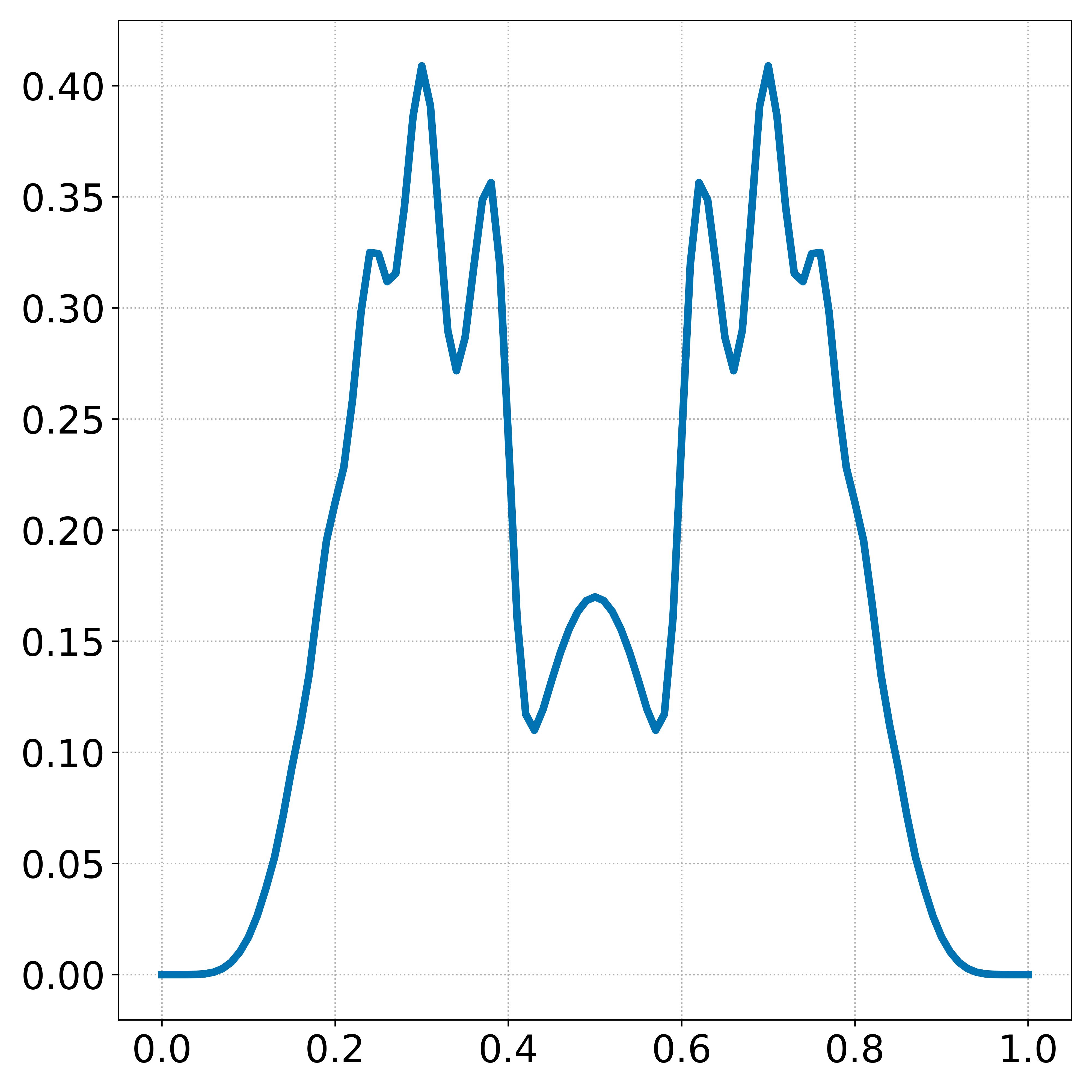}
          \caption{$\phi_A$ at $t = 10$}\label{fig:phiAt10}
          \end{subfigure}
          \begin{subfigure}{0.22\textwidth}
        \includegraphics[width=\textwidth]{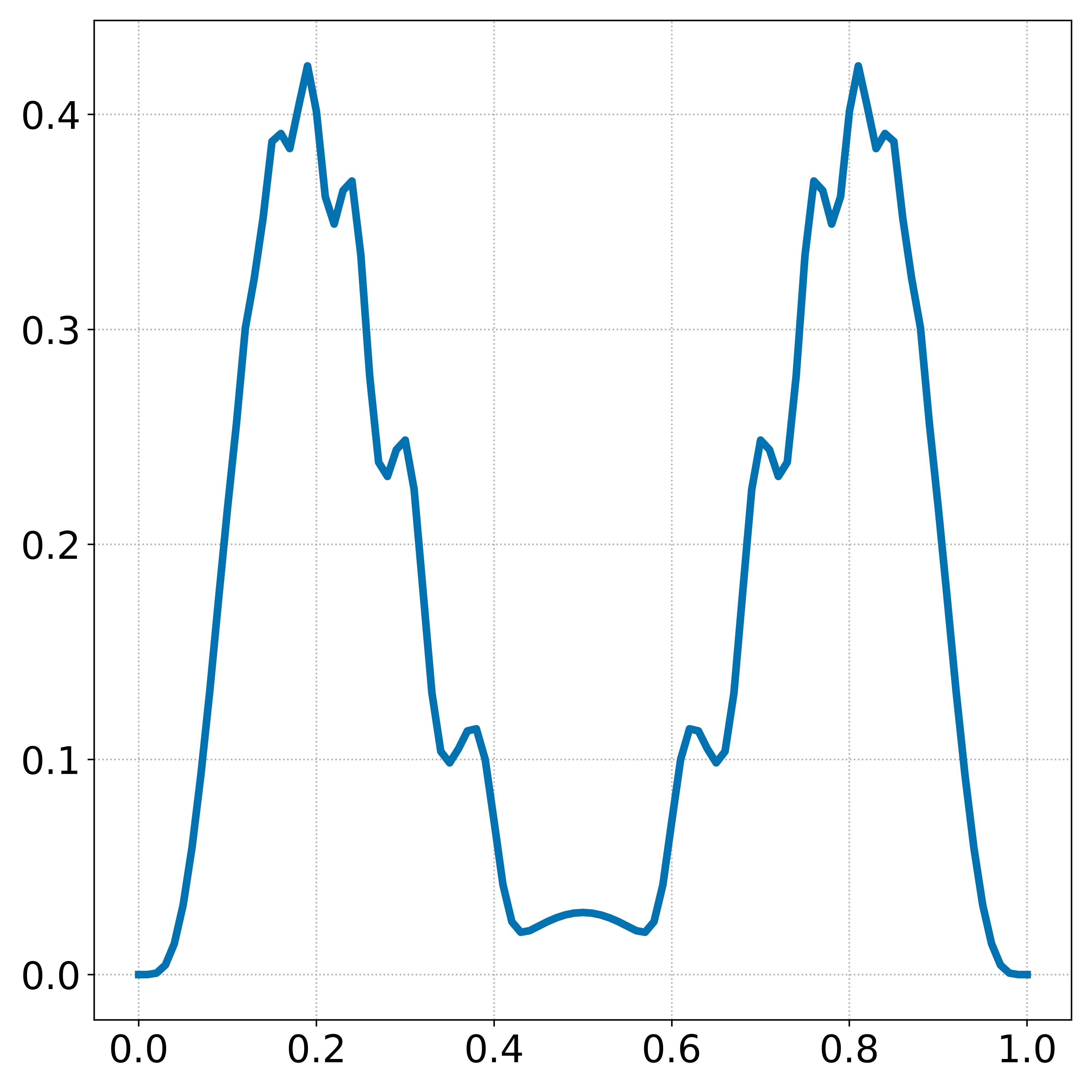}
          \caption{$\phi_A$ at $t = 20$}\label{fig:phiAt20}
          \end{subfigure}\\
           \begin{subfigure}{0.22\textwidth}
        \includegraphics[width=\textwidth]{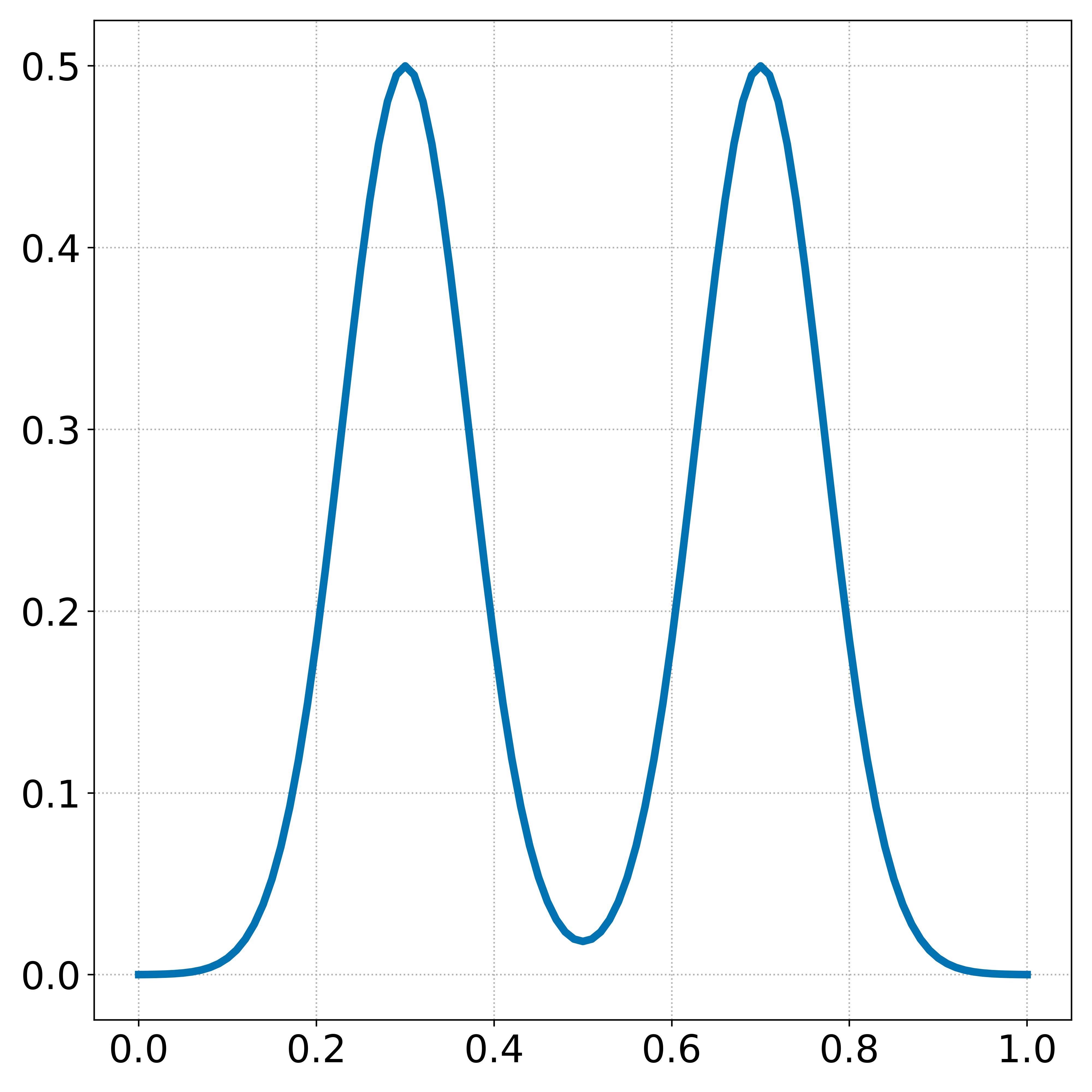}
          \caption{$\phi_T$ at $t = 0$}\label{fig:phiTt0}
          \end{subfigure}
        \begin{subfigure}{0.22\textwidth}
        \includegraphics[width=\textwidth]{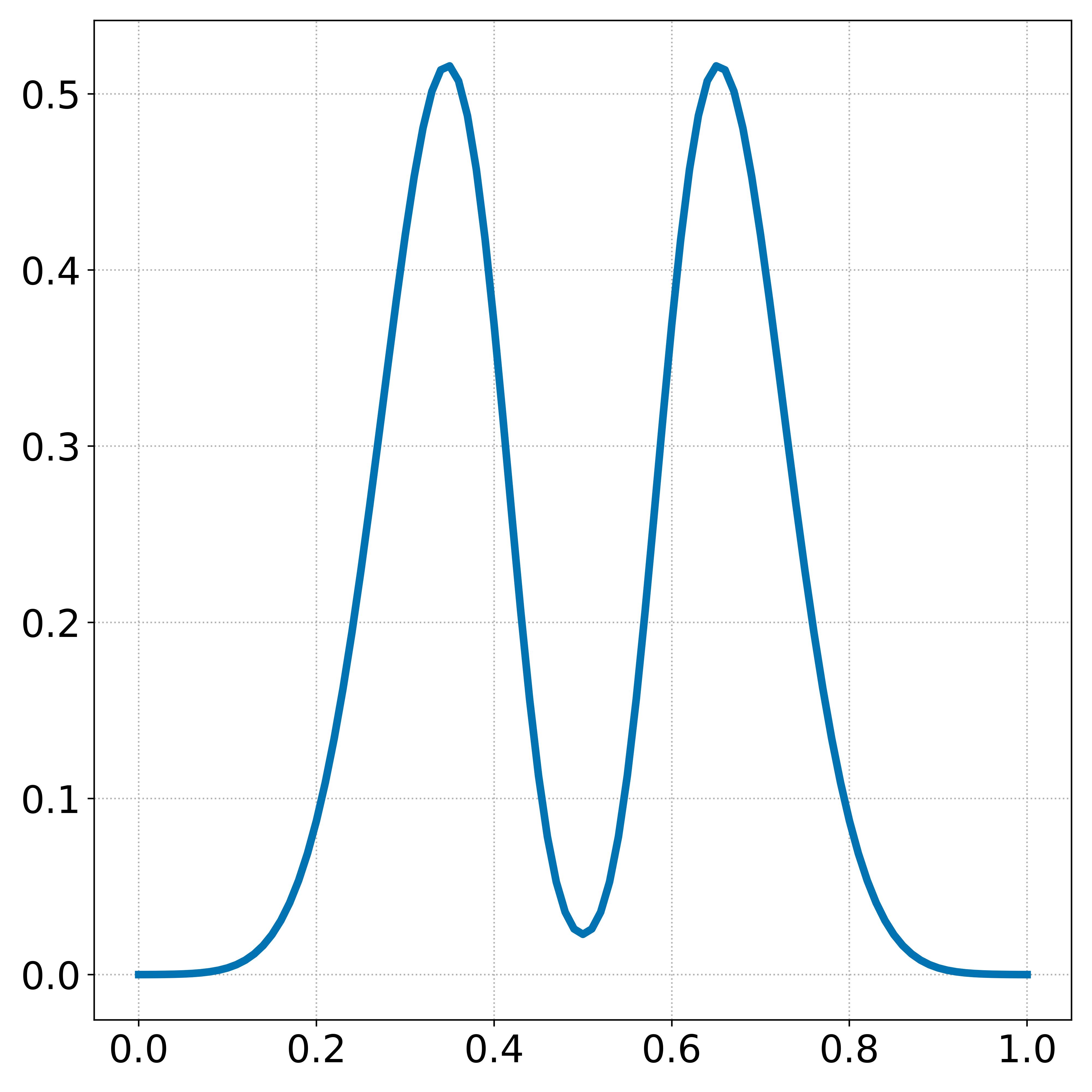}
          \caption{$\phi_T$ at $t = 5$}\label{fig:phiTt5}
          \end{subfigure}
        \begin{subfigure}{0.22\textwidth}
        \includegraphics[width=\textwidth]{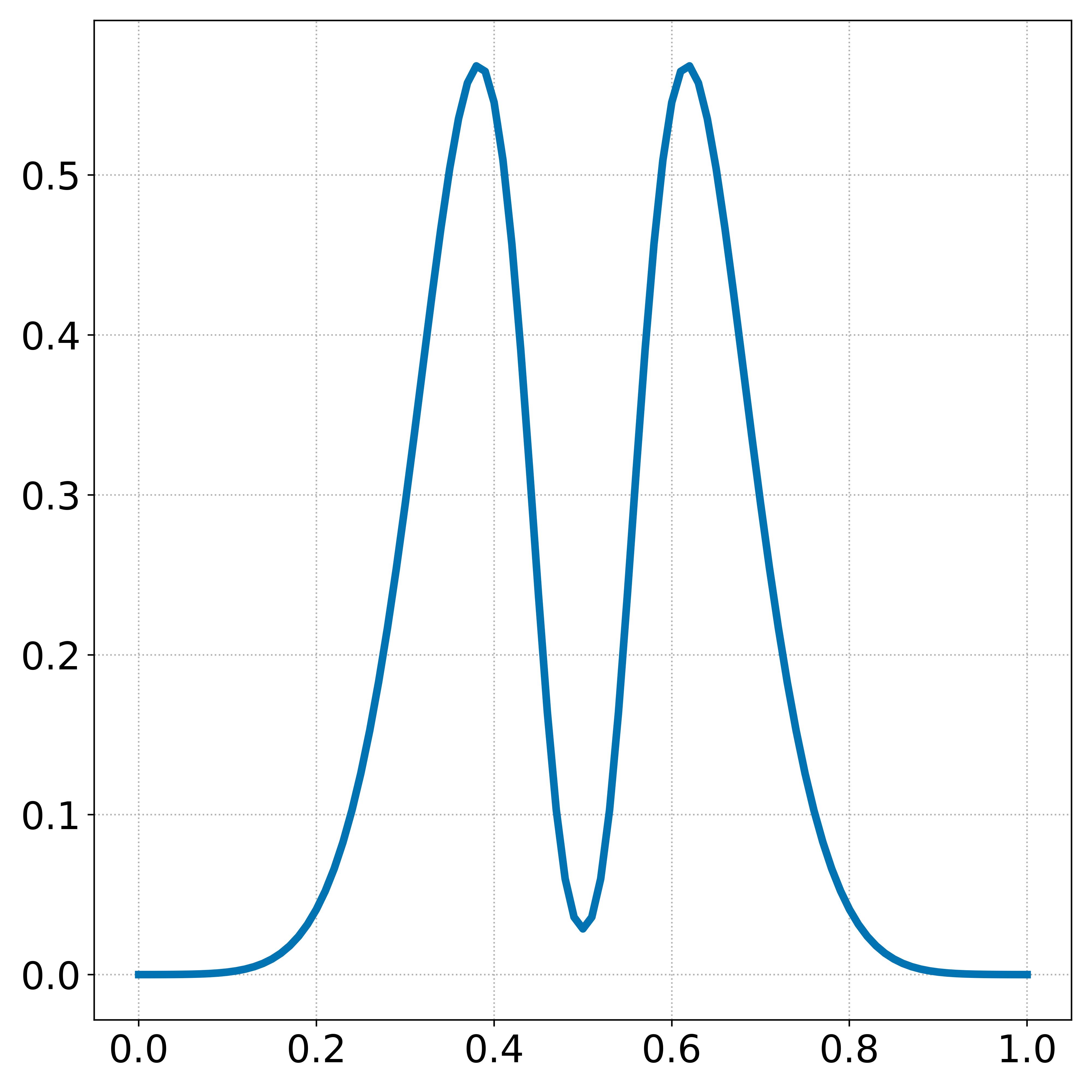}
        \caption{$\phi_T$ at $t = 10$}\label{fig:phiTt10}
          \end{subfigure}
          \begin{subfigure}{0.22\textwidth}
        \includegraphics[width=\textwidth]{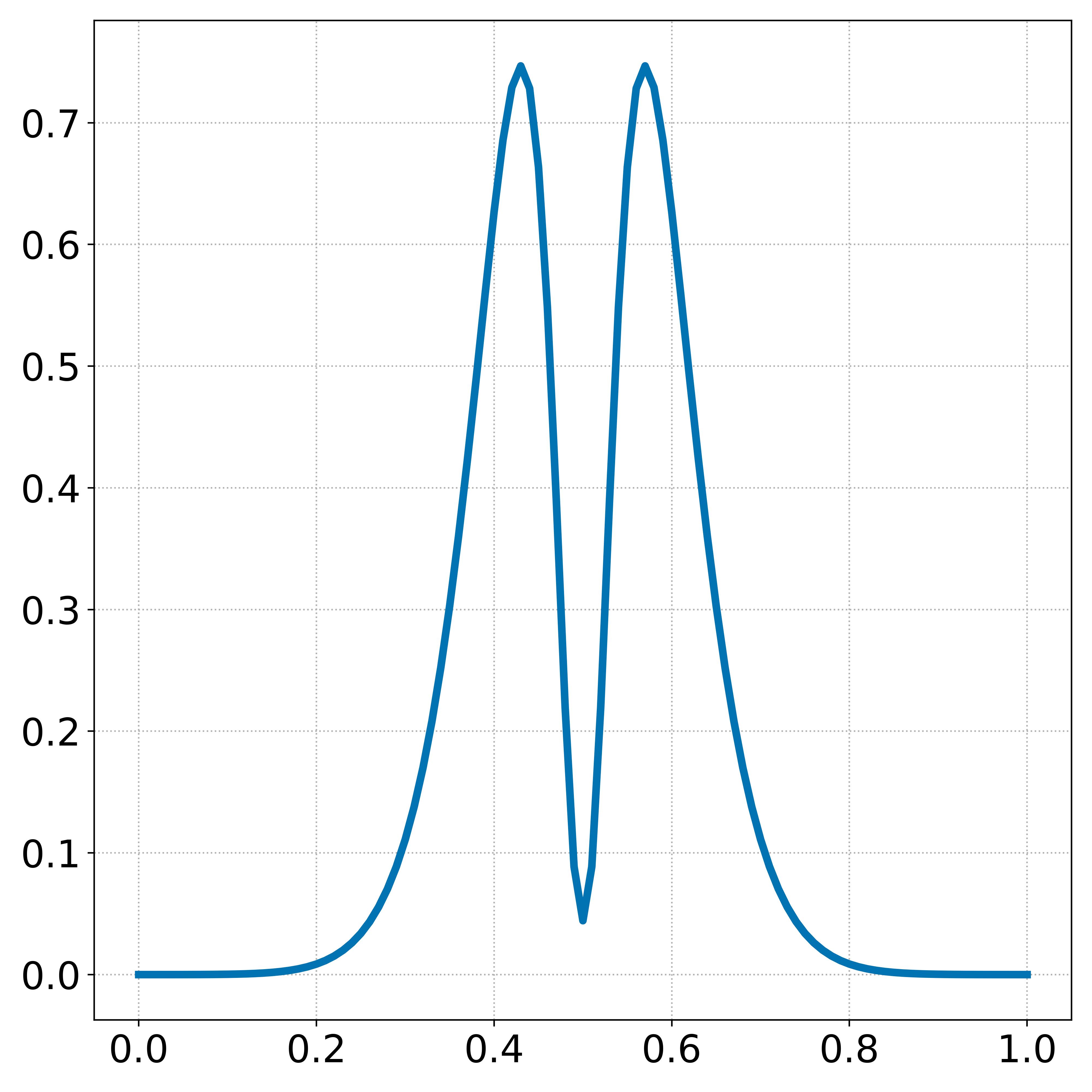}
          \caption{$\phi_T$ at $t = 20$}\label{fig:phiTt20}
          \end{subfigure}\\
\caption{Evolution of density $f_1$ for different transition functions $\phi$ (defined in Table \ref{tab:phitab}) and $\gamma = 0.4$; each row corresponds to one of the four representative families of functions. Initial conditions are the same for $\phi_L$, $\phi_R$, and $\phi_A$ to allow for direct comparison of agent movement patterns; the example for $\phi_T$ is otherwise chosen to show a case starting from a bimodal distribution. All simulations are run using MPCM with $100$ collocation points over the time interval $[0,20]$.}
\label{fig:simonepart}
\end{figure}

\begin{figure}[H]
    \centering
    \subcaptionbox{Log-log plot of the self-conv. metric  \label {fig:err}}{\includegraphics[width=.45\textwidth]{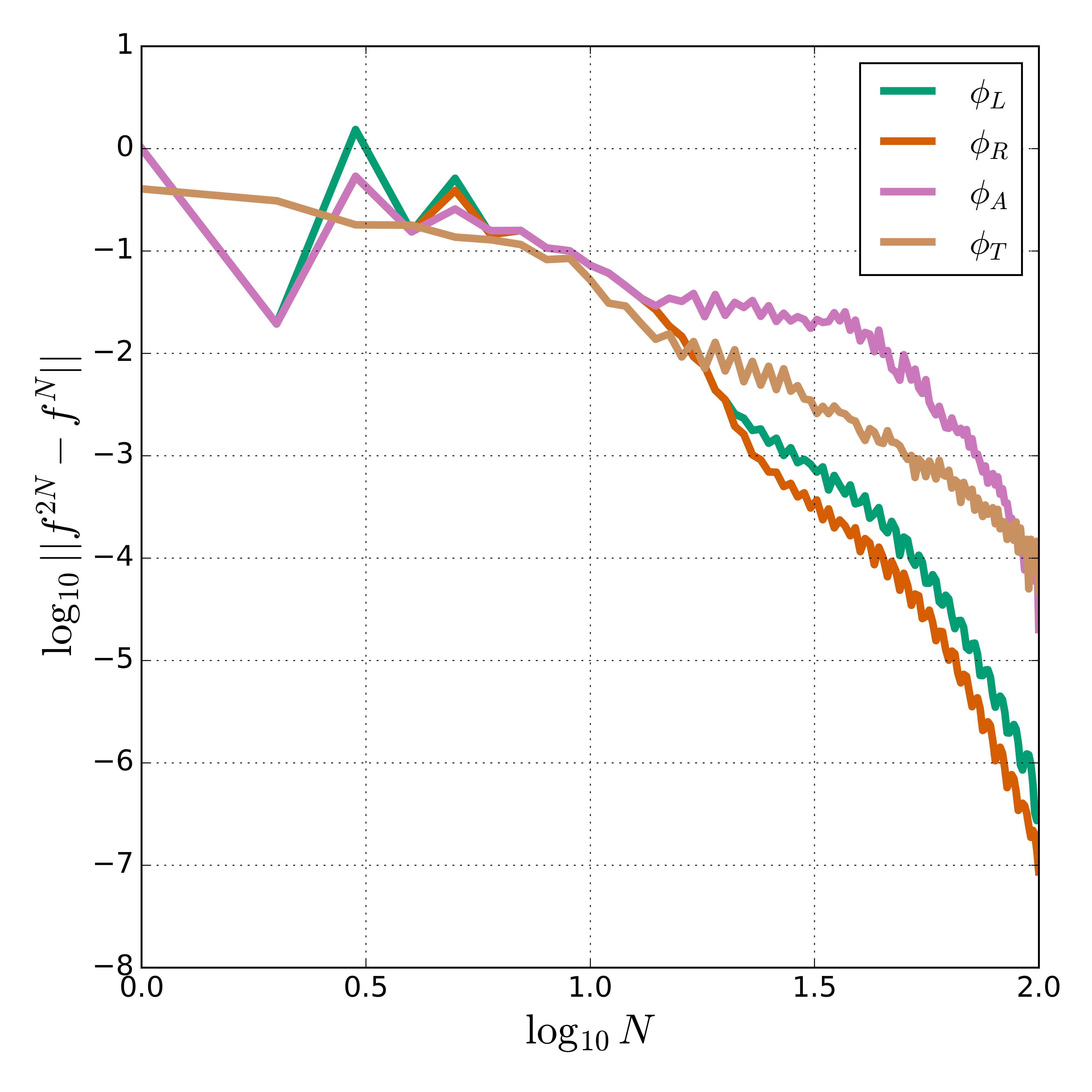}}\hspace{1em}%
    \subcaptionbox{Log-log plot of the runtime of MPCM\label{fig:runtime1}}{\includegraphics[width=.45\textwidth]{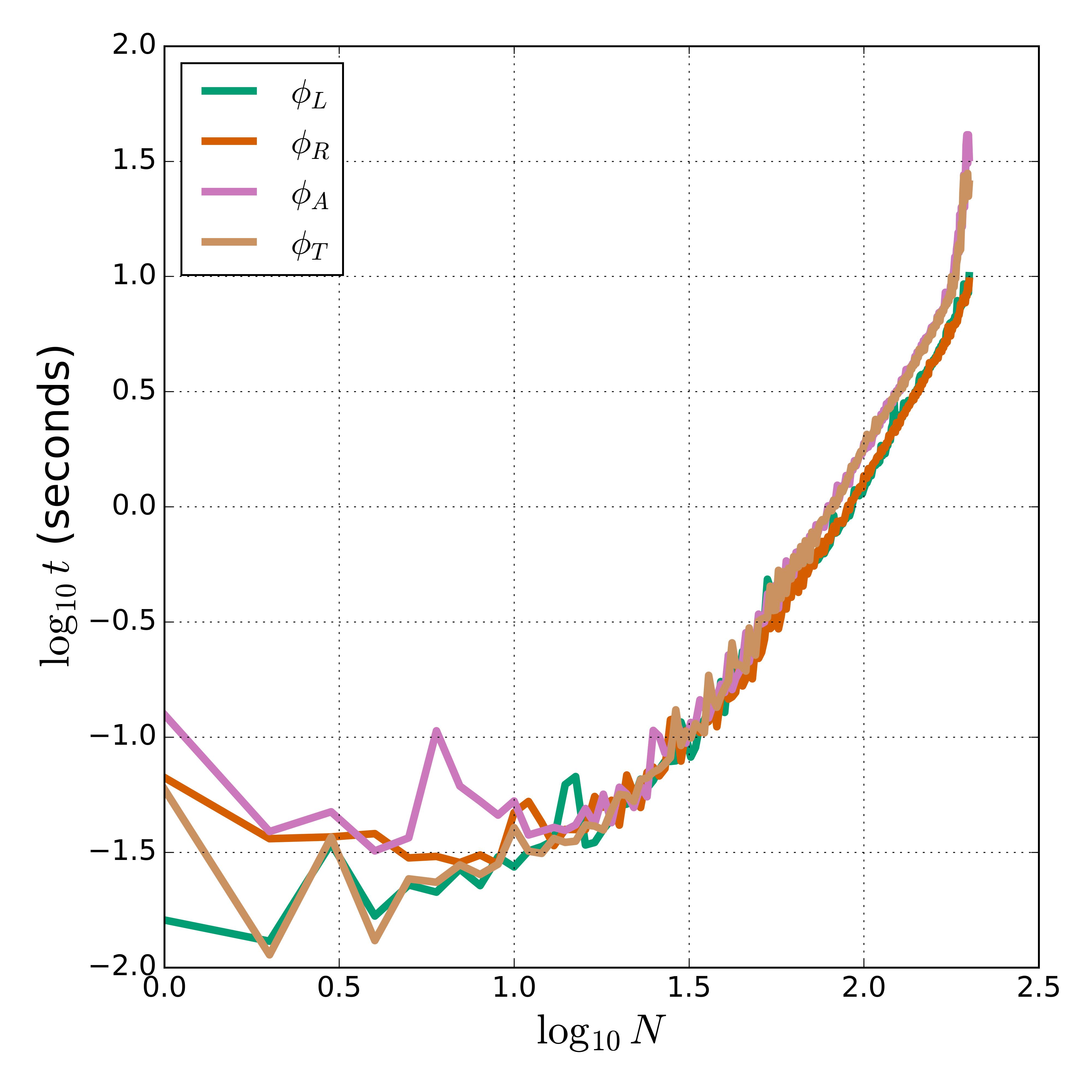}}\hspace{1em}%
    \caption{Comparison of the self-convergence metric and time needed to run for system \eqref{EQ:oneparticles} with different transition rates from Table \ref{tab:phitab}. (a) Log-log plot of the self-convergence metric for different numbers of nodes $N$. This plot shows that the difference between simulations with $N$ and $2N$ approaches zero superalgebraically as $N$ grows.  (b) Log-log plot of the average runtime needed to run for different initial values and transition rates; runtime and computational cost are observed to be around $O(N^3)$.}
    \label{fig:errorruntime}
\end{figure} 

\paragraph{Tests with two dynamic subsystems} Next, we allow both subsystems to evolve in time; we require two different transition rates, $T_{12}^1$ and $T_{21}^2$, and set encounter rates to $\eta_{12} = \eta_{21} = 1$. The integro-differential system takes the following form
{\small
\begin{align}\label{EQ:twoparticles}
\left\{\begin{array}{l}
    \partial_t f_1(t,u) = \int_0^1\int_0^1\  \delta_{\phi_{12}^1(x,y) - u}f_1(t,x)f_2(t,y)dxdy- f_1(t,u)\int_0^1 f_2(t,y)dy, \vspace{6pt}\\
    \partial_t f_2(t,u) = \int_0^1 \int_0^1 \delta_{\phi_{21}^2(x,y) - u}f_1(t,x)f_2(t,y)dxdy- f_2(t,u)\int_0^1 f_1(t,x)dx,\end{array}\right.
\end{align}}
where $\phi_{12}^1(x,y)$ and $\phi_{21}^2(x,y)$ are chosen from the families in Table \ref{tab:phitab}. 

As a representative example, Figure \ref{fig:sim2phiLphiR} illustrates the dynamics of the solution to the kinetic system \eqref{EQ:twoparticles} with $\phi_{12}^1(x,y) = x - 0.4xy$ and $\phi_{21}^2(x,y) = x + 0.4(1-x)(1-y)$. We take two normal distributions as the initial distributions for $f_1$ and $f_2$, the first centered at $0.8$ and the second centered at $0.2$. We observe mirrored behavior: $f_1$ transitions to the left toward $0$, while $f_2$ exhibits symmetric behavior, shifting toward microstate $1$.

\begin{figure}[H]
    \centering
    \begin{subfigure}{0.22\textwidth}
        \includegraphics[width=\textwidth]{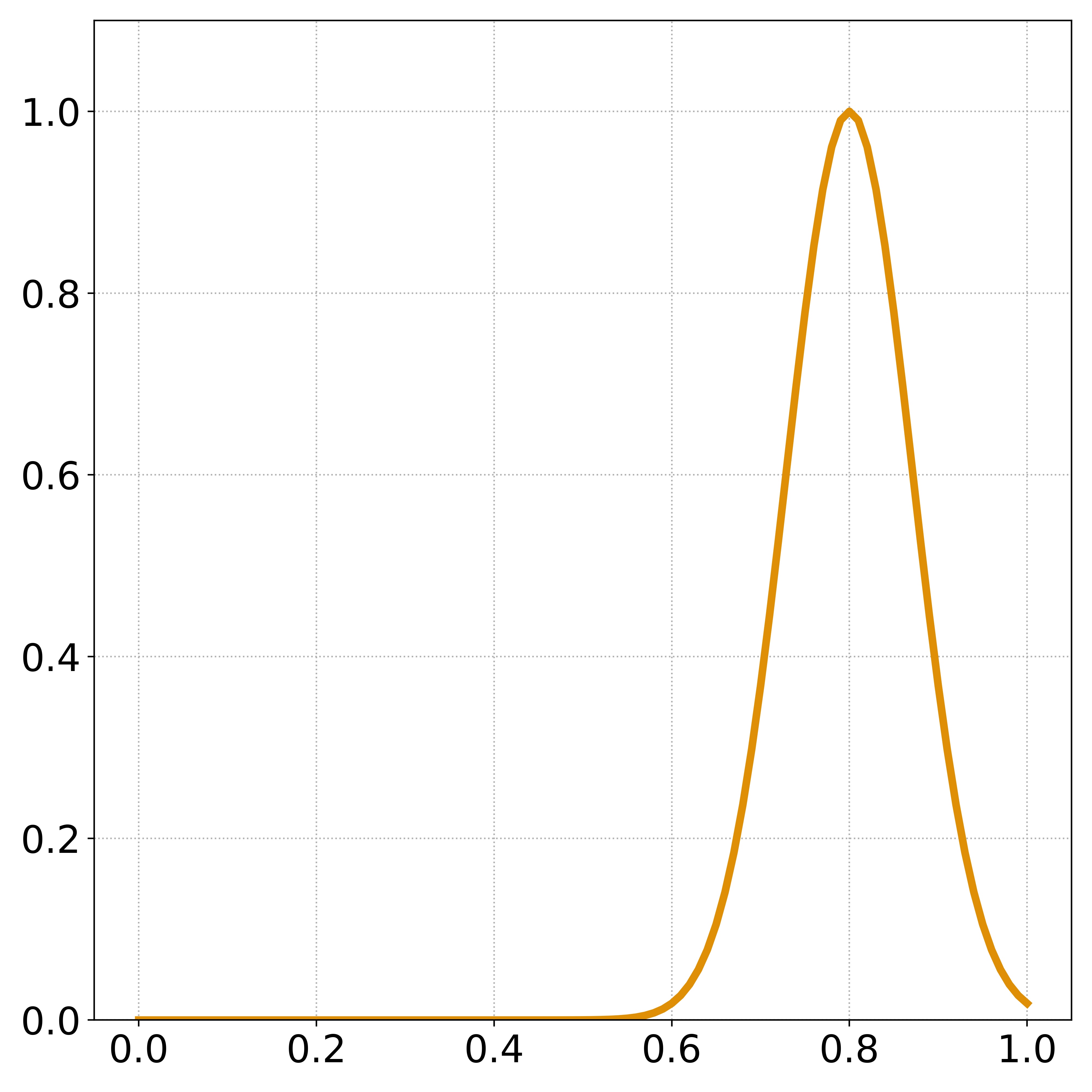}
          \caption{$f_1$ at $t = 0$}
          \end{subfigure}
        \begin{subfigure}{0.22\textwidth}
        \includegraphics[width=\textwidth]{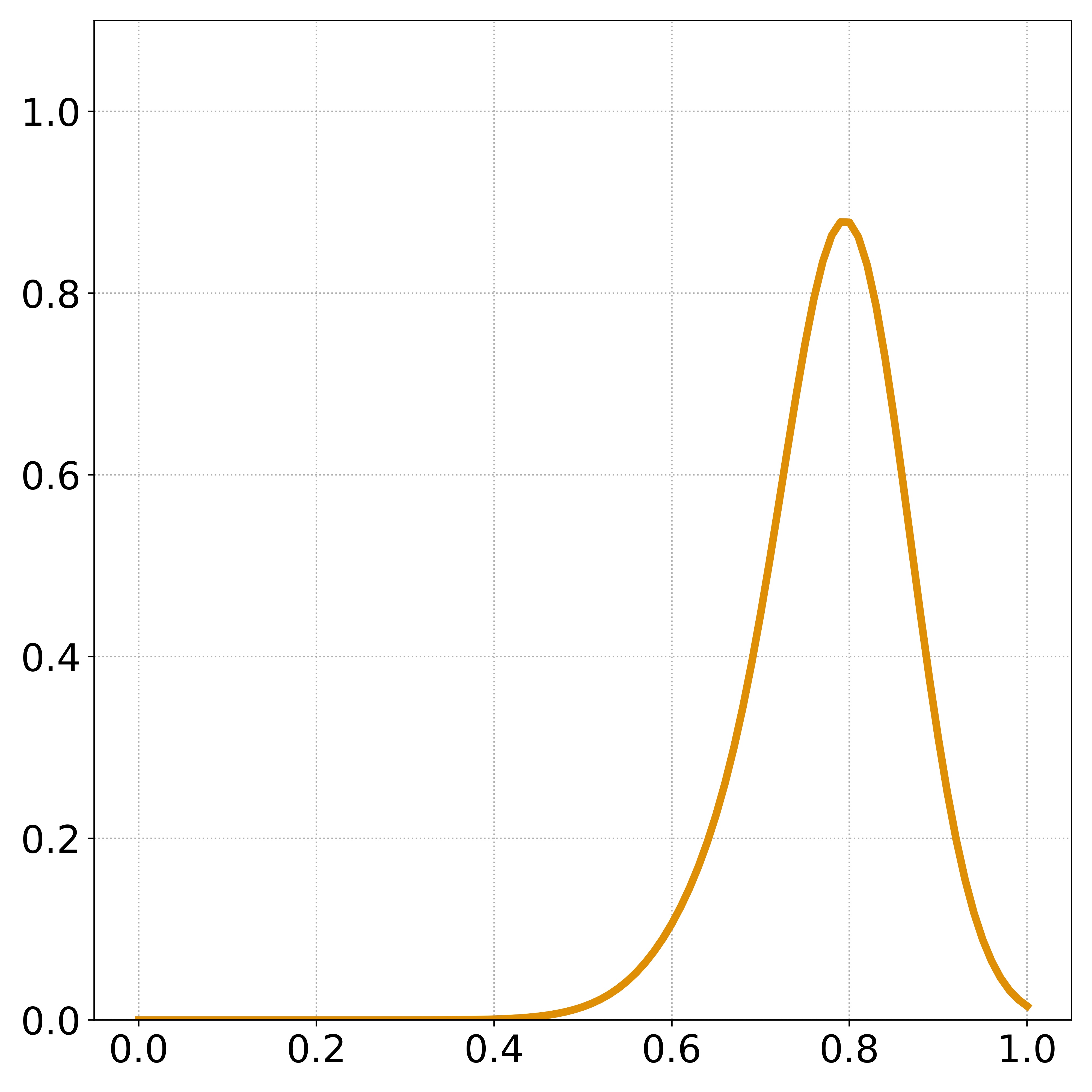}
          \caption{$f_1$ at $t = 1$}
          \end{subfigure}
        \begin{subfigure}{0.22\textwidth}
        \includegraphics[width=\textwidth]{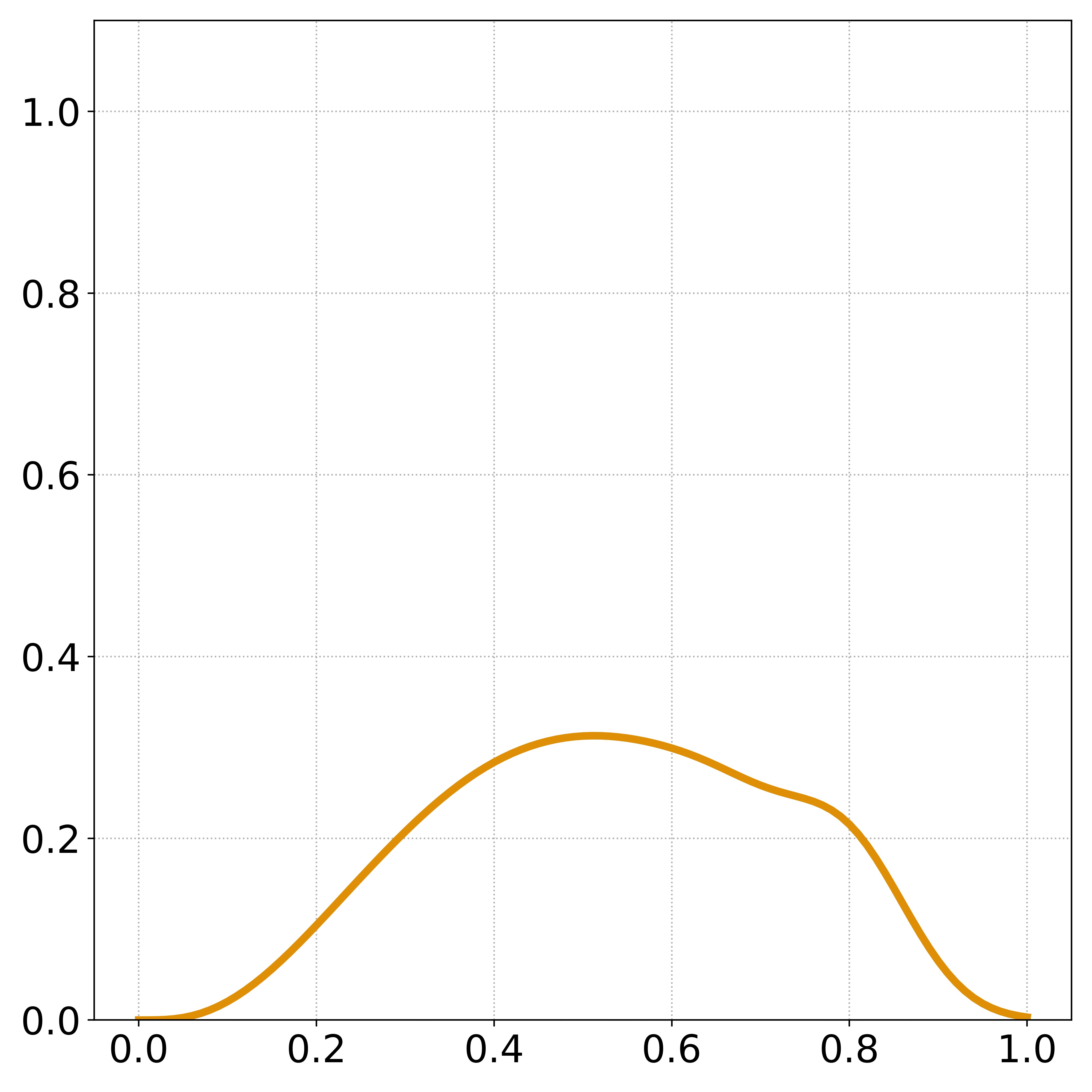}
          \caption{$f_1$ at $t = 10$}
          \end{subfigure}
          \begin{subfigure}{0.22\textwidth}
        \includegraphics[width=\textwidth]{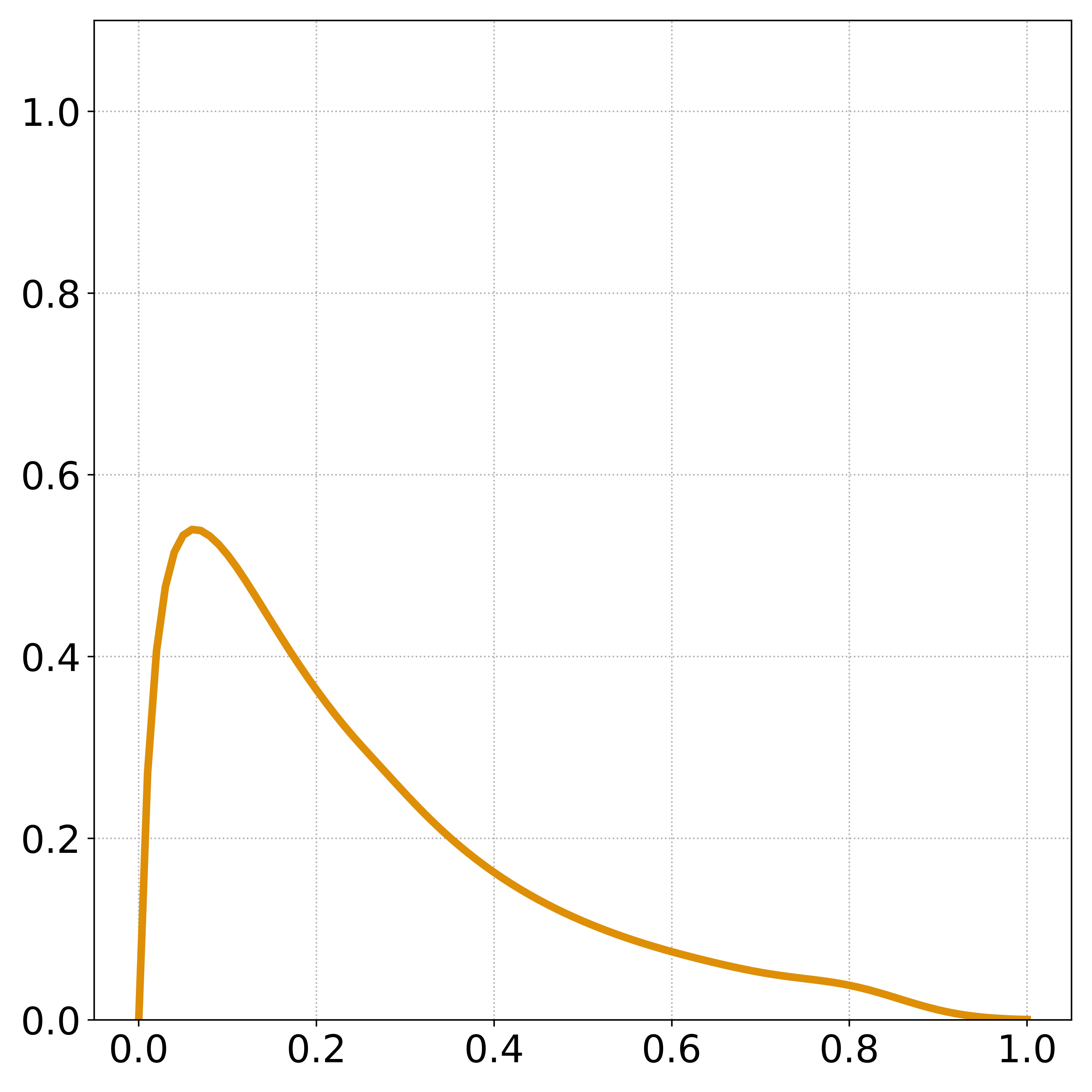}
          \caption{$f_1$ at $t = 20$}
          \end{subfigure}\\
        \begin{subfigure}{0.22\textwidth}
        \includegraphics[width=\textwidth]{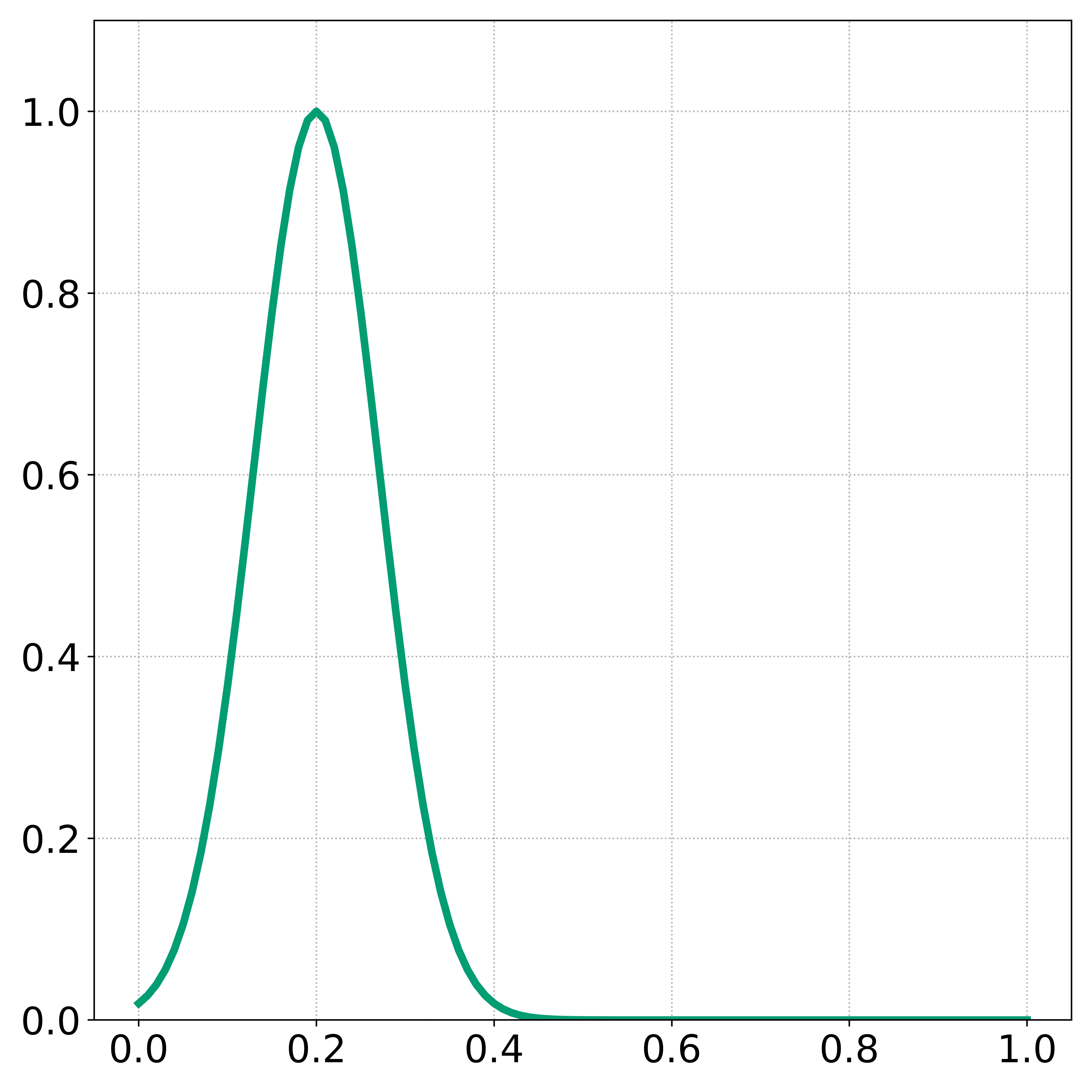}
          \caption{$f_2$ at $t = 0$}
          \end{subfigure}
        \begin{subfigure}{0.22\textwidth}
        \includegraphics[width=\textwidth]{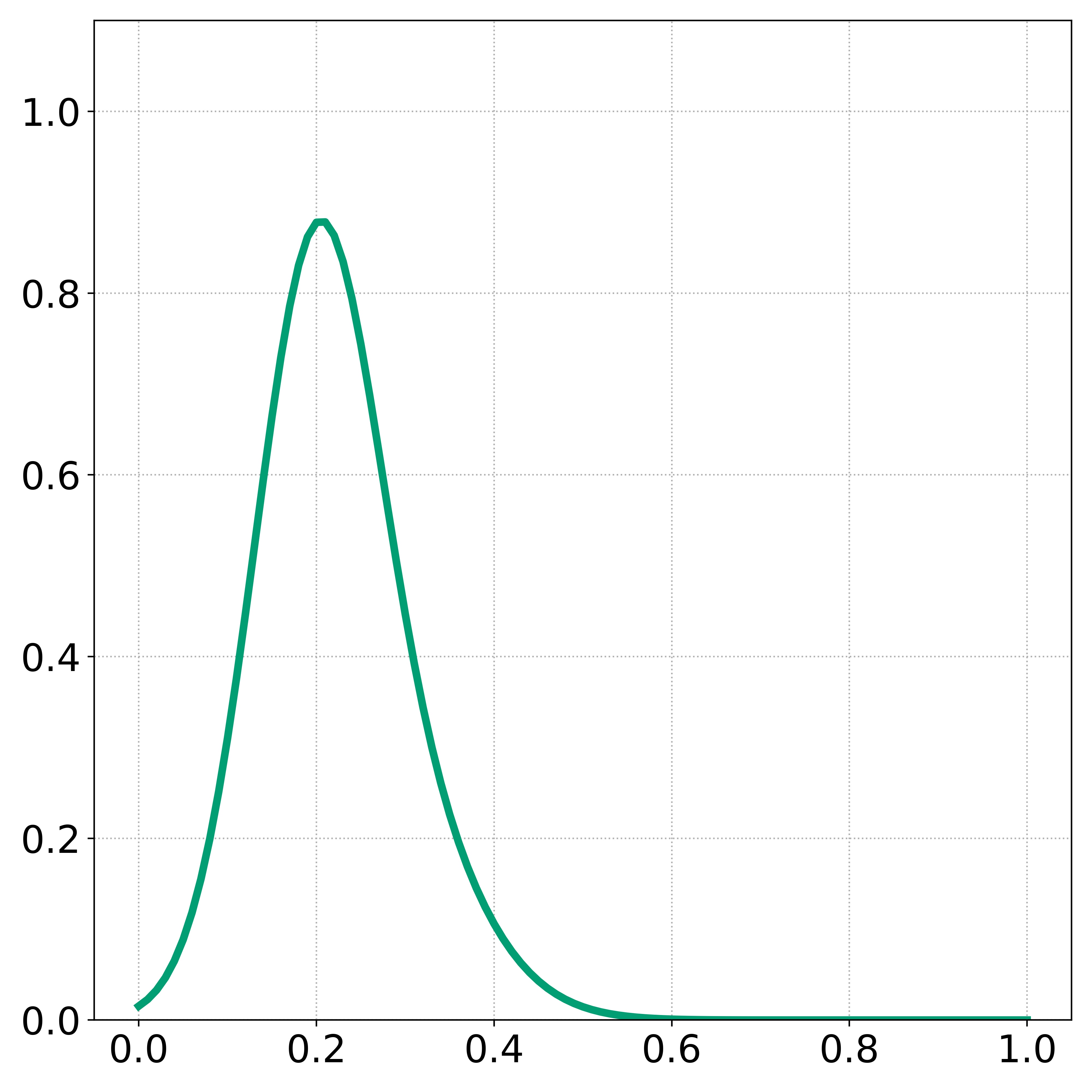}
          \caption{$f_2$ at $t = 1$}
          \end{subfigure}
        \begin{subfigure}{0.22\textwidth}
        \includegraphics[width=\textwidth]{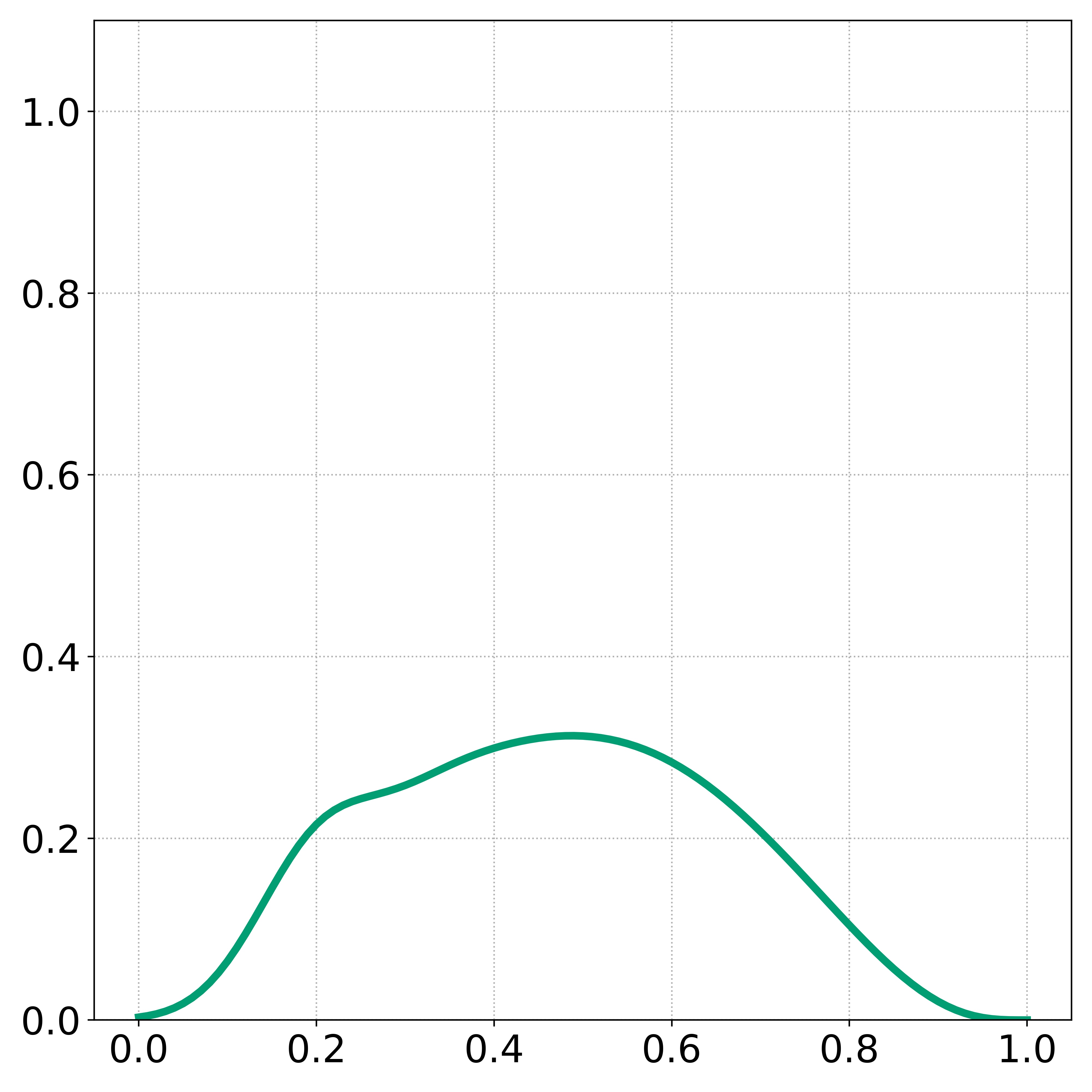}
          \caption{$f_2$ at $t = 10$}
          \end{subfigure}
          \begin{subfigure}{0.22\textwidth}
        \includegraphics[width=\textwidth]{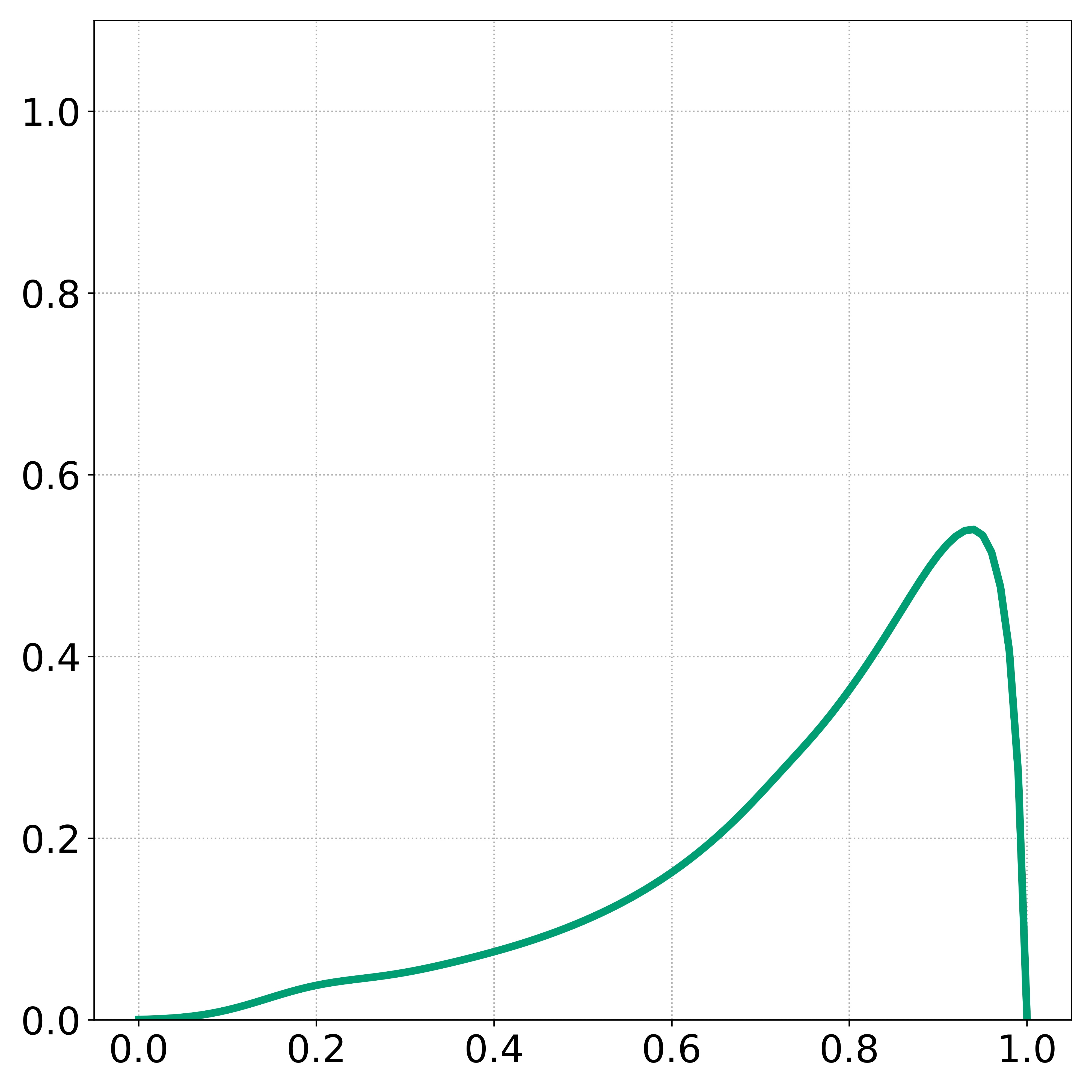}
          \caption{$f_2$ at $t = 20$}
          \end{subfigure}
  \caption{Temporal evolution of the system of kinetic equations \eqref{EQ:twoparticles} for the transition functions $\phi_{12}^1 = \phi_L$ and  $\phi_{21}^2 = \phi_R$ with $\gamma = 0.4$. The top row shows the evolution of $f_1$, while the bottom row illustrates the evolution of $f_2$. The simulation spans the time interval from $[0,20]$ using MPCM with $100$ collocation points.}
  \label{fig:sim2phiLphiR}
\end{figure}

We perform three separate sets of experiments corresponding to distinct choices of transition function pairs, tracking self-convergence metrics and average runtime. Results (Figure \ref{fig:errorruntime2}) are consistent with those in Figure \ref{fig:errorruntime}: the self-convergence metric decreases superalgebraically with $N$, and time complexity remains consistently $O(N^3)$, with precomputation dominating runtime.

\begin{figure}[H]
    \centering
    \subcaptionbox{Log-log plot of the self-conv. metric \label{fig:err2}}{\includegraphics[width=.45\textwidth]{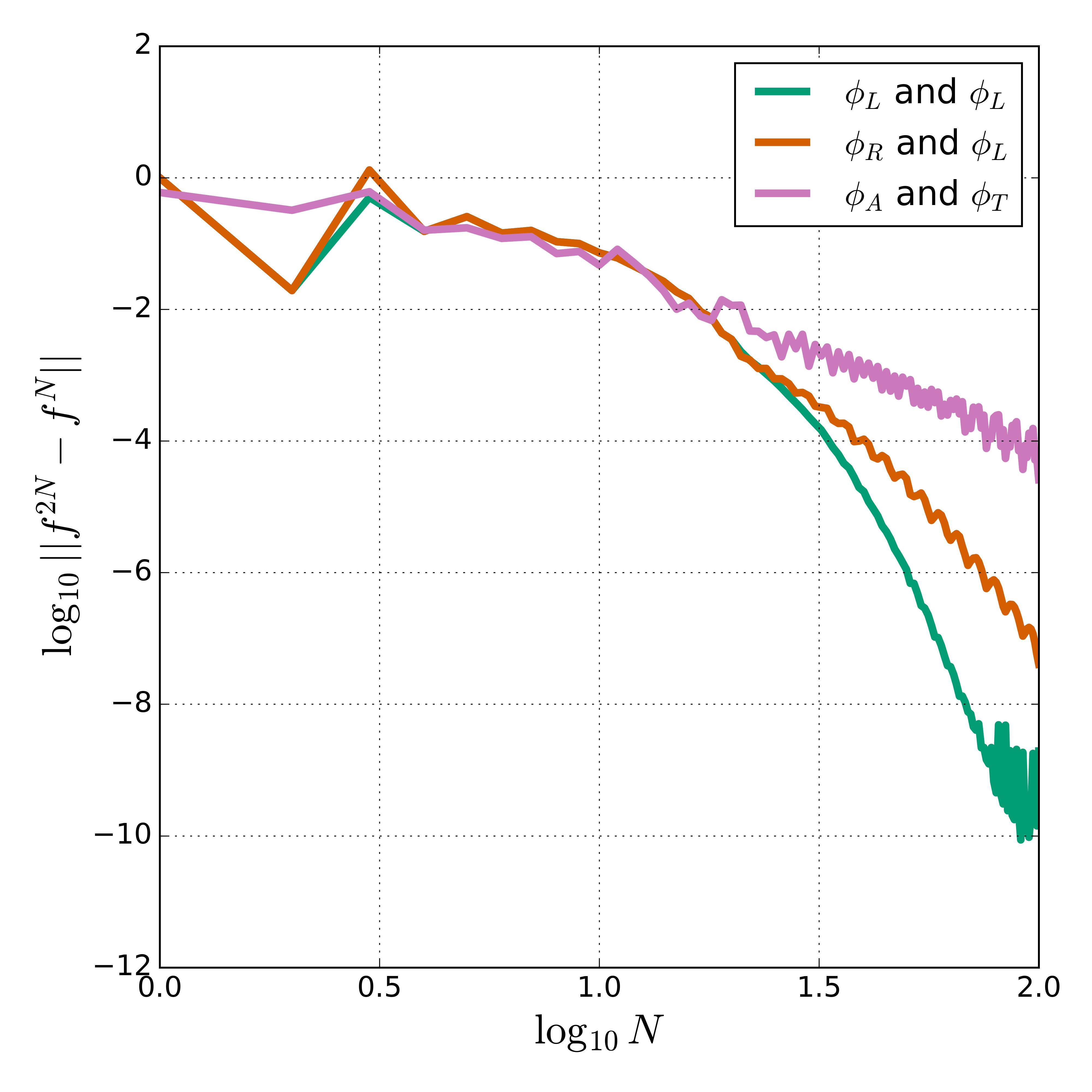}}\hspace{1em}%
    \subcaptionbox{ Log-log plot of the runtime of MPCM \label{fig:runtime}}{\includegraphics[width=.45\textwidth]{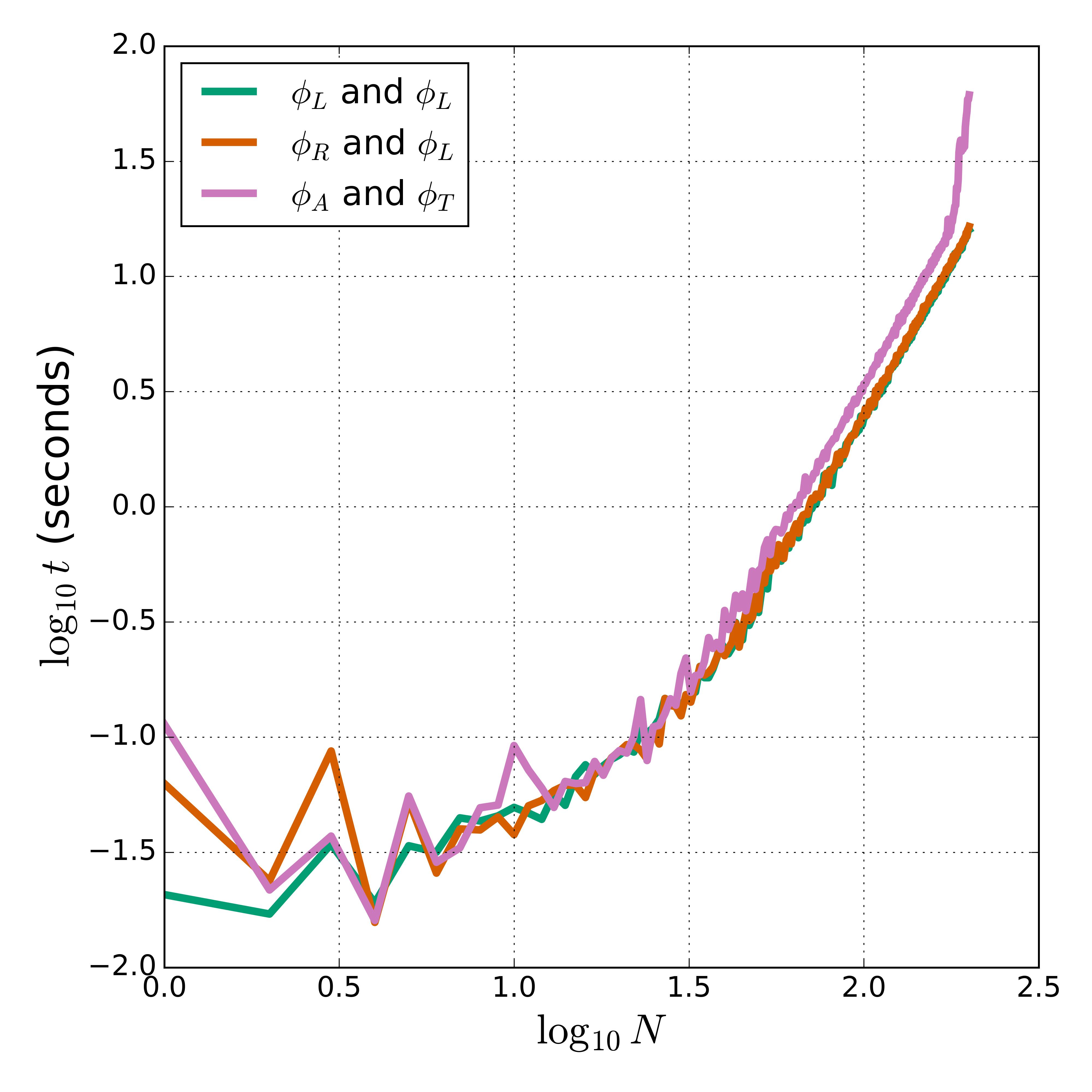}}\hspace{1em}%
    \caption{(a) Log-log plot of the MPCM self-convergence metric for system \eqref{EQ:twoparticles}, compared to $\log N$.  (b) Log-log plot of the average runtime needed to run for different initial values and transition rates for a kinetic system of two subsystems with two transition rates.}
    \label{fig:errorruntime2}
\end{figure}

\subsection{Comparison of MPCM with Tau-leaping and Hybrid Algorithms}

\paragraph{Experimental setup and metrics} In this section, we compare the performance of MPCM with two stochastic alternatives: the Tau-leaping Gillespie algorithm and a hybrid stochastic–deterministic approach. Both Tau-leaping and hybrid methods involve stochastic, agent-based simulation for which interacting agents are divided into $N_B = 100$ bins according to their microstate. The variability inherent to this approach can then be reduced by averaging results for $M$ independent simulations. We briefly describe each of these two methods, as well as additional setup and parameters. 

\emph{Tau-leaping method \cite{gillespie2007stochastic}:} it accelerates the standard Gillespie SSA \cite{gillespie}, by allowing multiple reactions to occur within a fixed time step $\tau$, modeled as Poisson random variables \cite{gillespie2001approximate,anderson2012multilevel}, the full algorithm of which is presented in \ref{app:taulp}. The choice of time step $\tau$ determines the trade-off between accuracy and computational efficiency. In this work, we use a fixed ($\tau$ = 0.001), which is small enough to maintain stability and capture particle interaction dynamics while still allowing efficient system evolution; this choice for $\tau$ is standard in kinetic systems modeling \cite{rathinam2007reversible}. 


\emph{Hybrid method:} as described in \ref{app:hybrid}, this approach aims to reduce computational cost by treating low-density regions stochastically (through Tau-leaping) while using deterministic updates in high-density regions where fluctuations are negligible; this requires the user to choose a deterministic-stochastic threshold parameter of $n_T$ particles. In our experience, finding an optimal $n_T$ can be problem-dependent; through empirical testing, we found $n_T = 1000$ maintained mass preservation within machine precision while maintaining accuracy comparable to that of Tau-leaping. 


\emph{MPCM setup and comparison:} Unlike both methods above, MPCM is a deterministic integrator, only requiring the user to choose the number of collocation nodes $N$. Unless otherwise stated, we employ $N=100$ first-kind Chebyshev nodes to ensure that the numerical resolution of the deterministic method matches the granularity of the stochastic methods. To compare discrepancies between density dynamics obtained by MPCM and either stochastic solver, we use the difference metric
\begin{equation}\label{eq:diffmetric}
    D = \max_{t \in [0, T_f]}\sum_i \sum_j \bigg\lvert f_i^{N}(t,u_j) - f_i^{M, N_B}(t,u_j)\bigg\lvert  du,
\end{equation}
where $f_i^{N}(t,u)$ is the MPCM result with $N$ collocation nodes and $f_i^{M,N_B}(t,u)$ is the stochastic solution with $N_B$ bins averaged over $M$ simulations. 

\paragraph{Tests with one dynamic subsystem} Figures \ref{fig:tauleaphybridleft1sim} show time evolution snapshots for system \eqref{EQ:oneparticles} setting transition function $\phi_L$ with $\gamma = 0.4$, comparing MPCM with the two stochastic solvers. In all cases, microstate densities produced by MPCM and both stochastic methods exhibit similar qualitative behavior. However, single realizations of Tau-leaping and hybrid simulations show noticeable fluctuations around the deterministic MPCM solution, as shown in the top row of Figure \ref{fig:tauleaphybridleft1sim}. Averaging over (M = 1000) realizations reduces this variability, producing close agreement between the methods; we display corresponding snapshots in the bottom row. A log-log plot of the difference metric (Figures \ref{fig:tauerrsims} -\ref{fig:hyberrsims}) shows convergence of both stochastic methods at the expected rate of $M^{-\frac{1}{2}}$.
\begin{figure}[H]
    \centering
    \begin{subfigure}{0.32\textwidth} \label{fig:tauphiL0}
        \includegraphics[width=\textwidth]{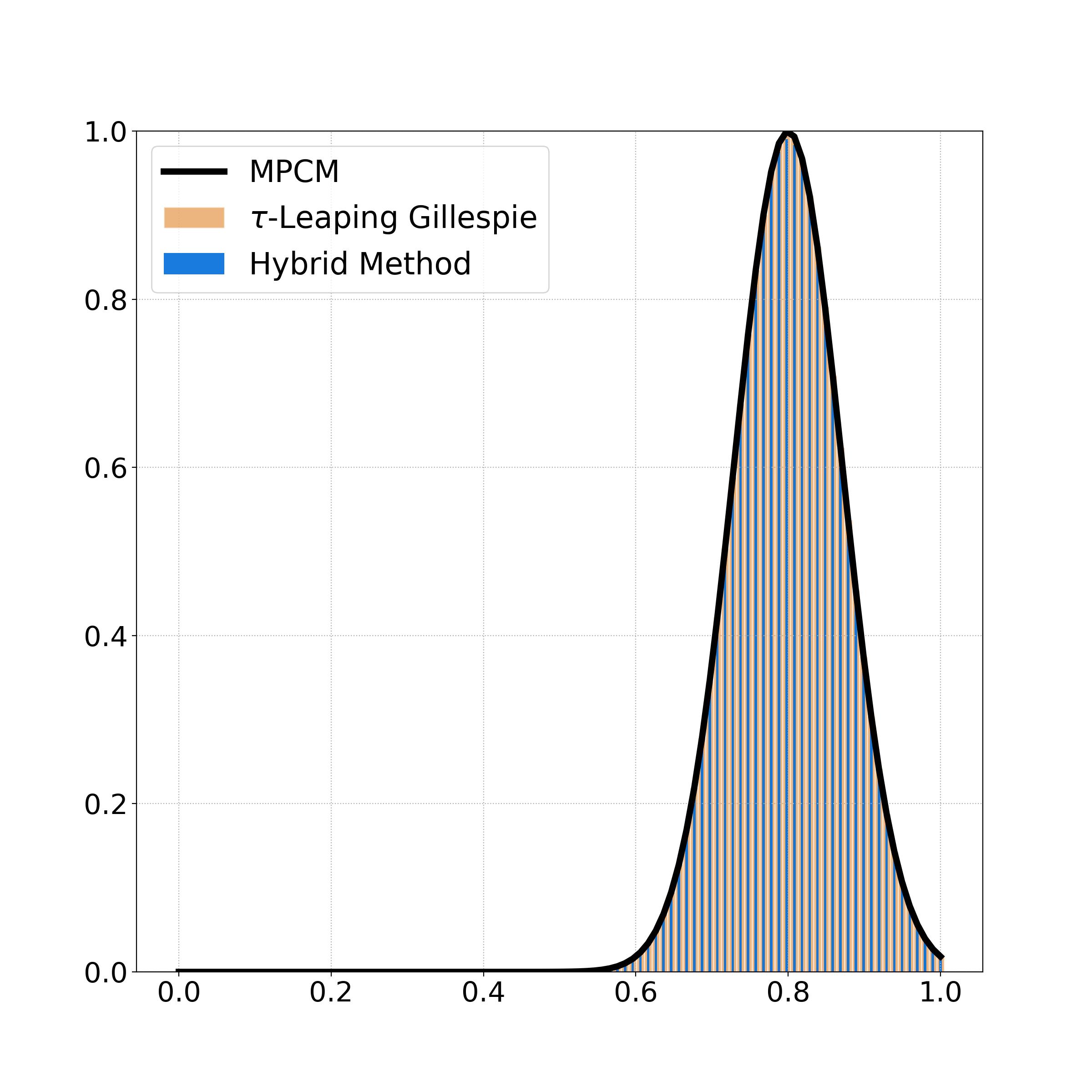}
          \caption{$M= 1$, $t = 0$}
          \end{subfigure}
        \begin{subfigure}{0.32\textwidth}
        \includegraphics[width=\textwidth]{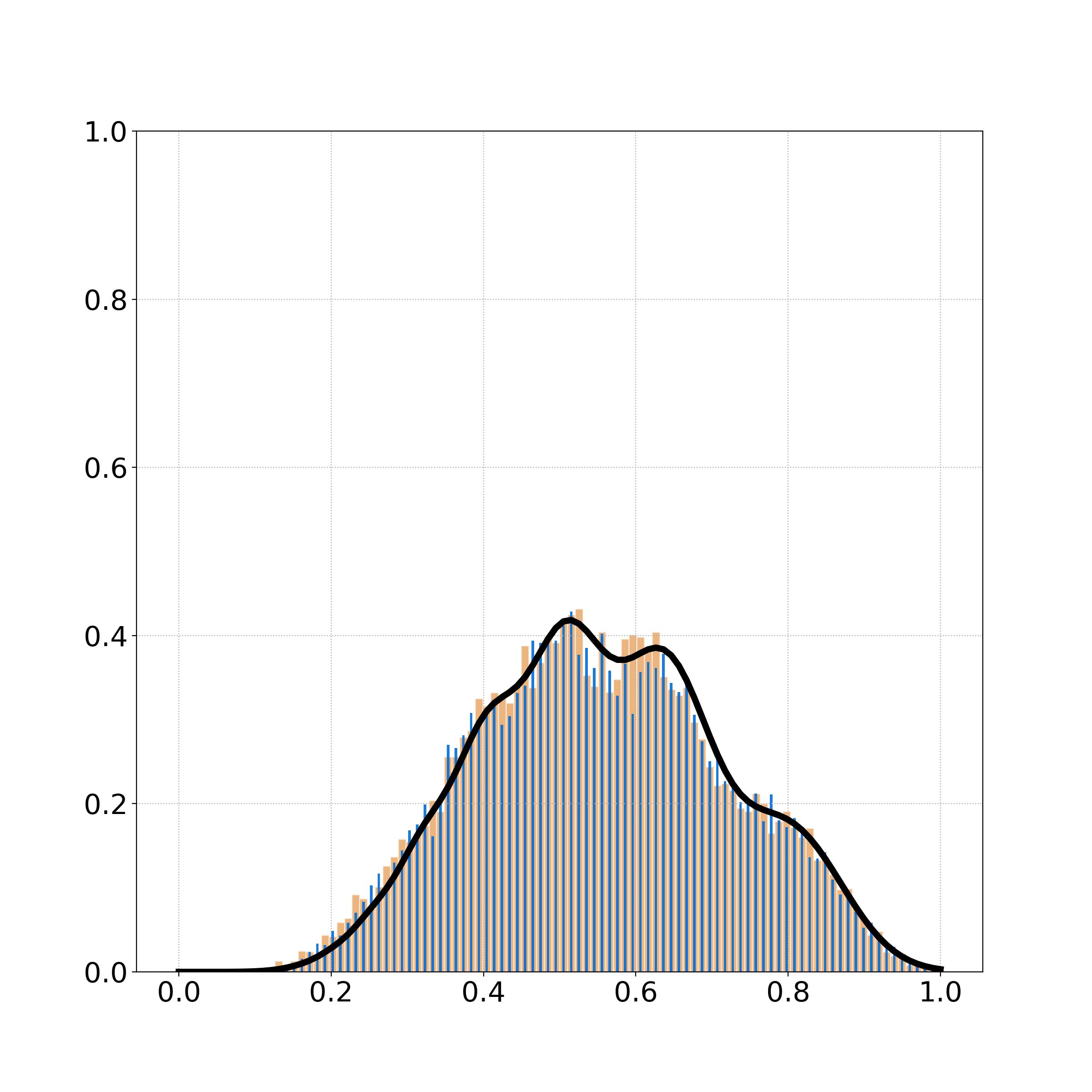}
          \caption{$M= 1$, $t = 10$}\label{fig:tauphiL10}
          \end{subfigure}
          \begin{subfigure}{0.32\textwidth}
        \includegraphics[width=\textwidth]{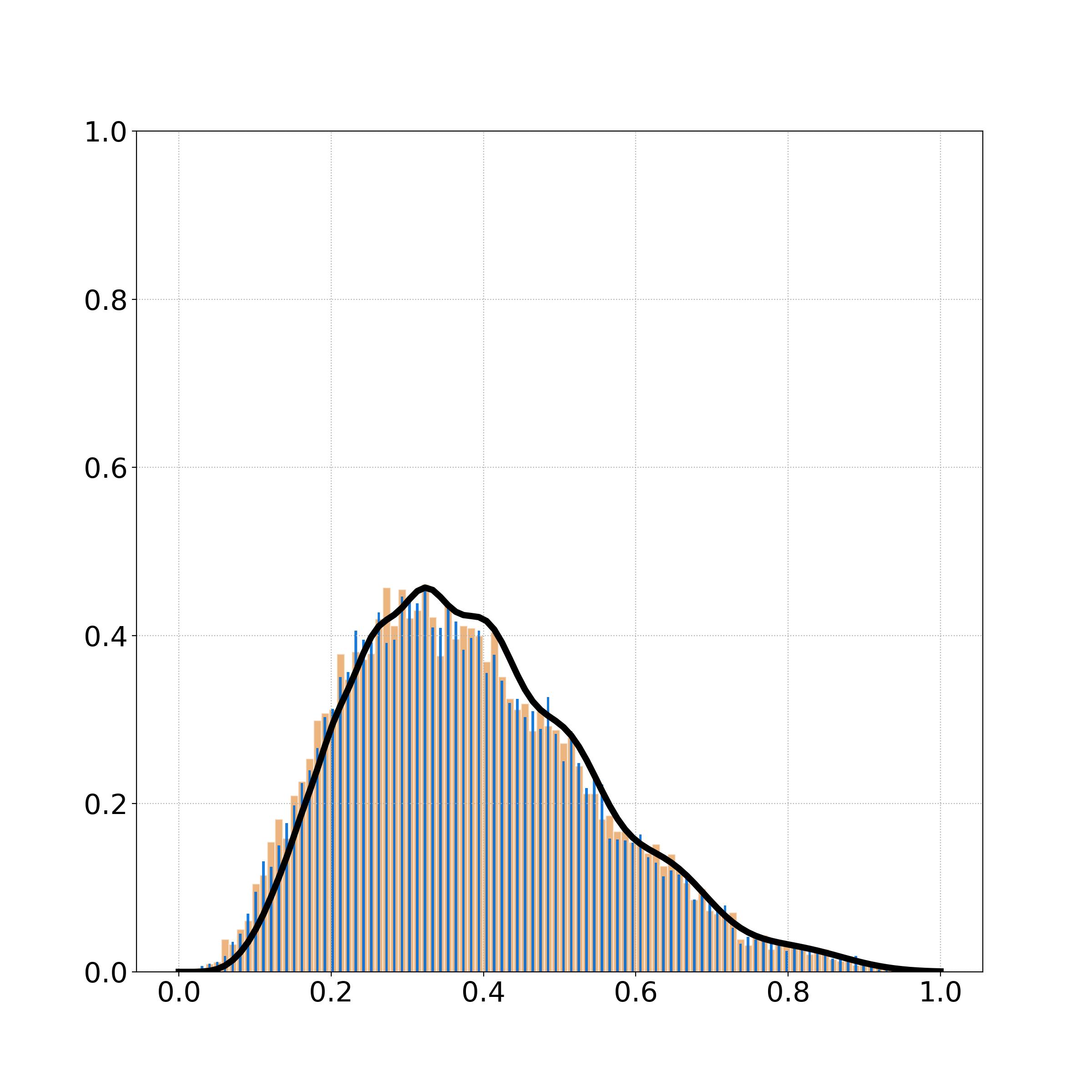}
          \caption{$M= 1$, $t = 20$}\label{fig:tauphiL0}
          \end{subfigure}\\
              \begin{subfigure}{0.32\textwidth} \label{fig:tauphiL0}
        \includegraphics[width=\textwidth]{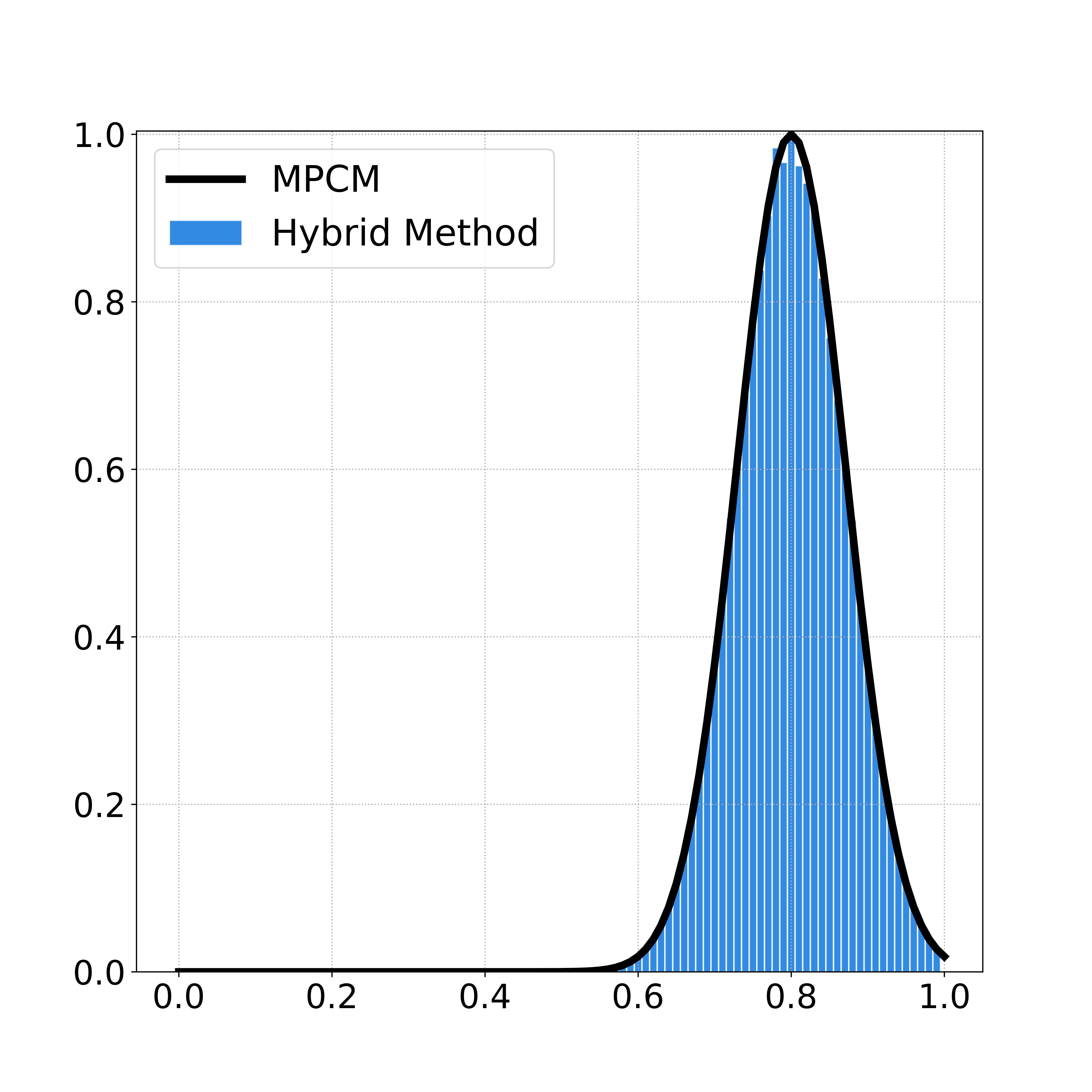}
          \caption{$M= 1000$, $t = 0$}
          \end{subfigure}
        \begin{subfigure}{0.32\textwidth}
        \includegraphics[width=\textwidth]{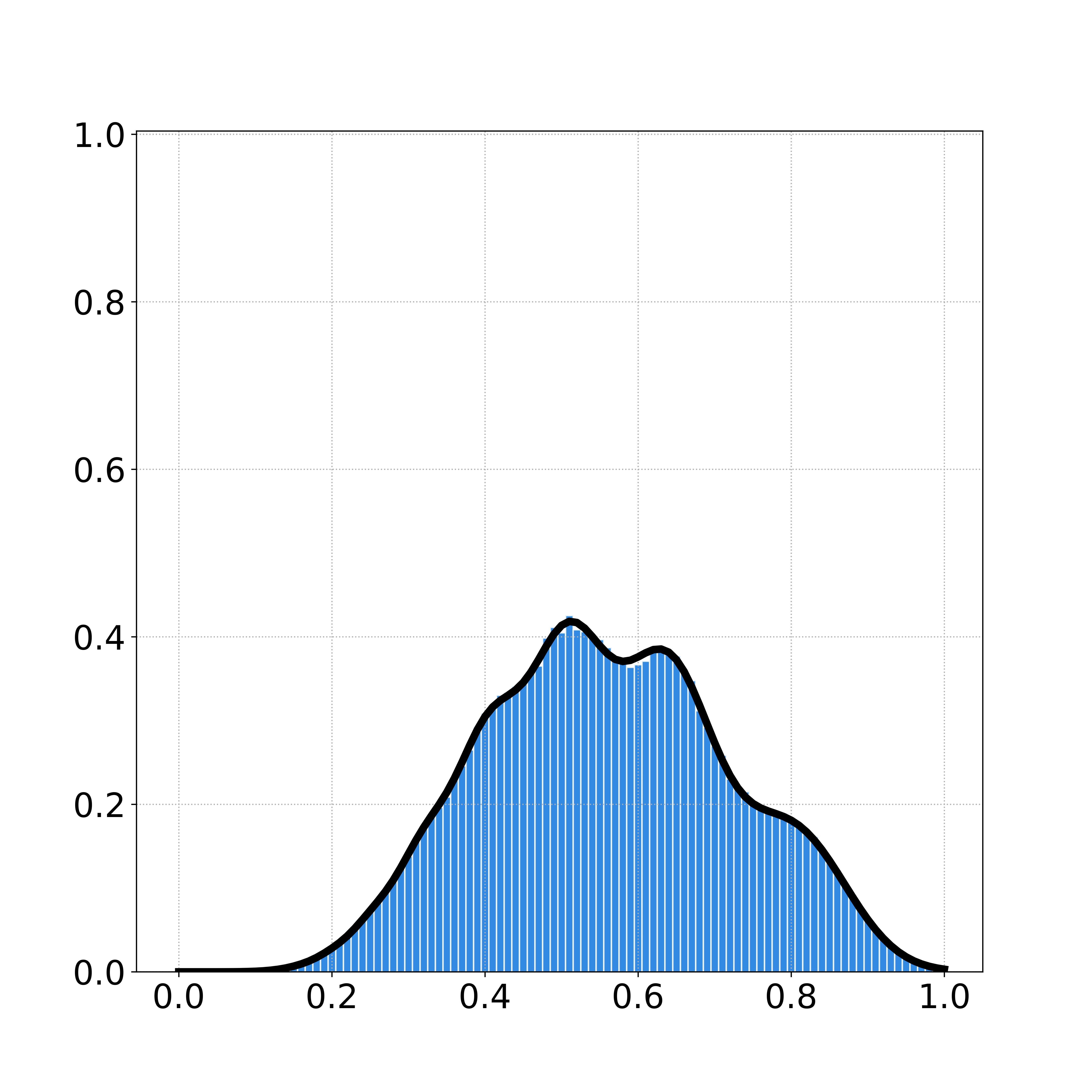}
          \caption{$M= 1000$, $t = 10$}\label{fig:tauphiL10}
          \end{subfigure}
          \begin{subfigure}{0.32\textwidth}
        \includegraphics[width=\textwidth]{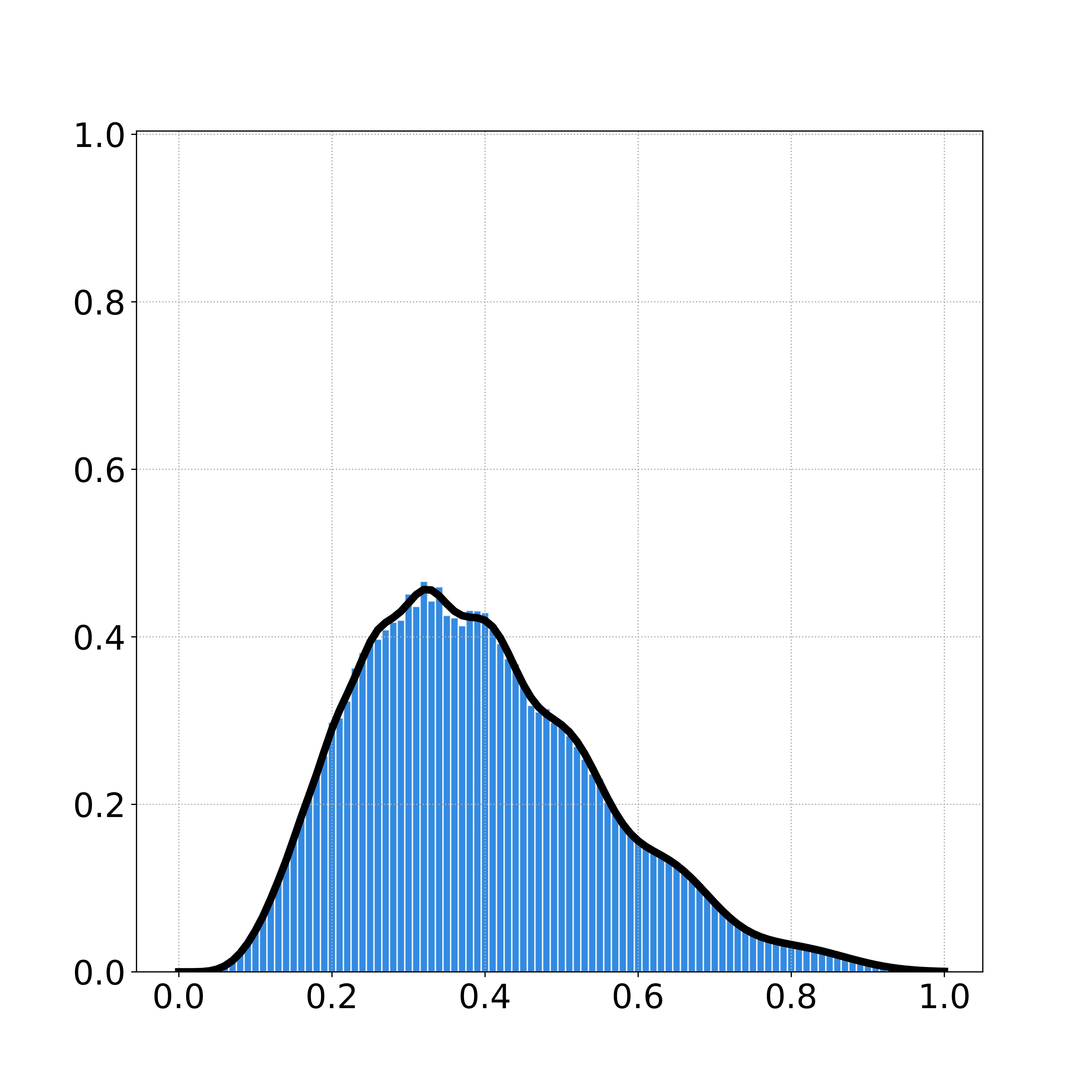}
          \caption{$M= 1000$, $t = 20$}\label{fig:tauphiL0}
          \end{subfigure}\\

\caption{Time evolution snapshots of system \eqref{EQ:oneparticles} for transition function $\phi_L$ with $\gamma = 0.4$. The black curve represents the solution obtained using MPCM, while bar plots present histograms on $N_B = 100$ bins for Tau-leaping (yellow) and the hybrid method (blue). The system is integrated for $t \in [0, 20]$. The top row shows results from a single stochastic simulation, while the bottom row shows results averaged over $M=1000$ simulations. Only the hybrid method is shown in the bottom row, as the Tau-leaping method produced identical results.}
\label{fig:tauleaphybridleft1sim}
\end{figure}

\begin{figure}[H]
    \centering
    \begin{subfigure}{0.45\textwidth}
        \includegraphics[width=\textwidth]{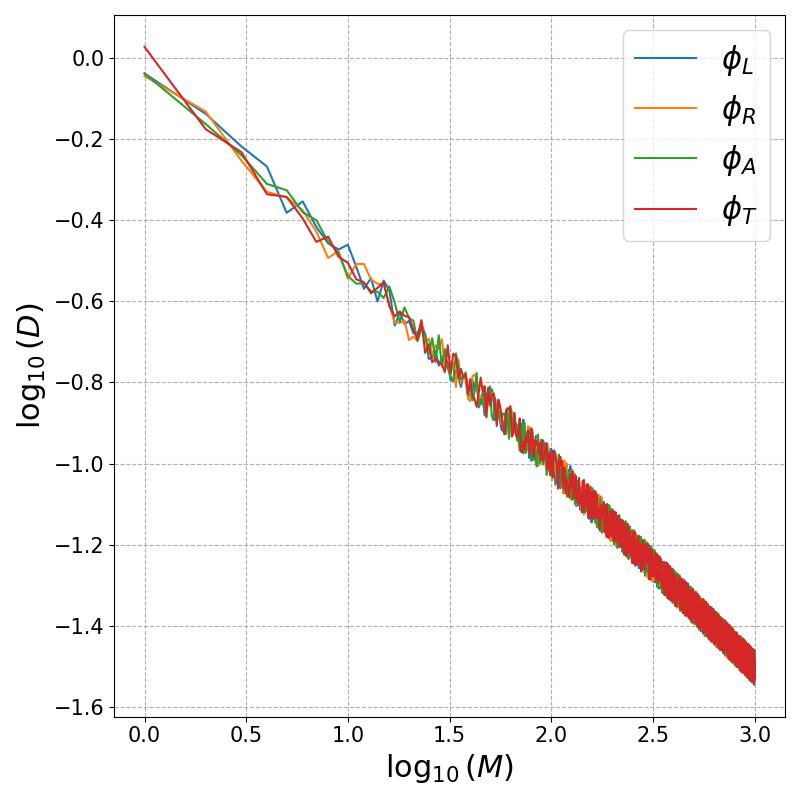}
         \caption{ Tau-leaping}\label{fig:tauerrsims}
          \end{subfigure}
   \begin{subfigure}{0.45\textwidth}
        \includegraphics[width=\textwidth]{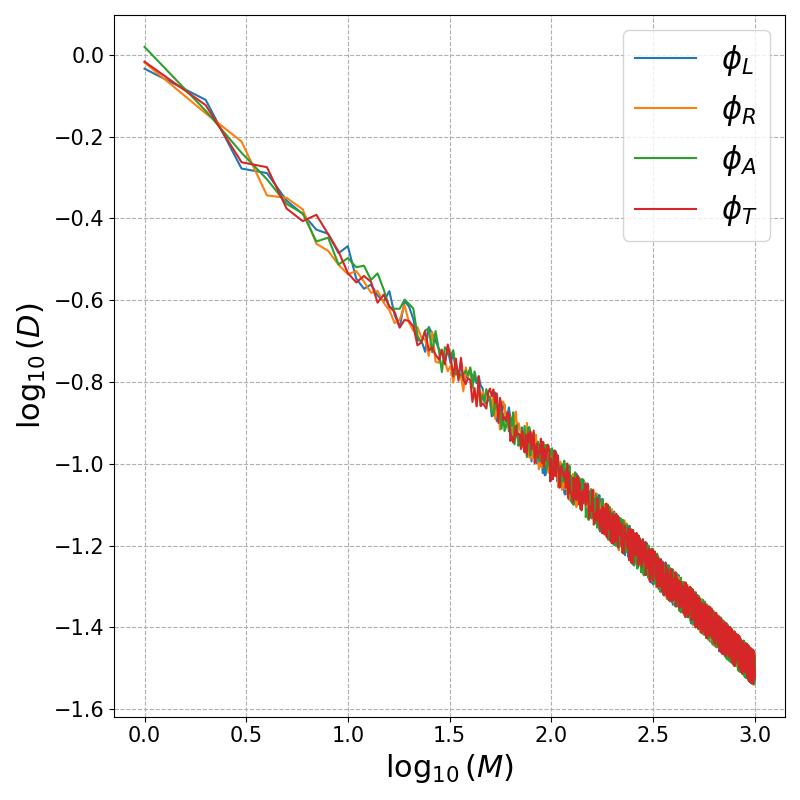}
        \caption{Hybrid Method}\label{fig:hyberrsims}
          \end{subfigure}
\caption{Log–log plot showing the difference between (a) the Tau-leaping Gillespie method and MPCM, and (b) the hybrid method and MPCM, using 100 collocation nodes and $N_B = 100$ compartments, for varying numbers of simulations $(M)$.}
\label{fig:errgilone}
\end{figure}

In Table \ref{tab:solvertimetoruntauhyb}, we compare the total runtimes for the MPCM and for single simulations ($M=1$) of both stochastic methods; this allows us to roughly estimate how cost-efficient each method is at delivering density dynamics of similar quality. We observe a clear progression in runtime efficiency. The Tau-leaping method is the slowest of the three methods. The hybrid method is more efficient, running approximately ten times faster than the Tau-Leaping method while maintaining nearly identical accuracy. MPCM runs the continuous density dynamics on a single computer in just a few seconds, making it about thirty times faster than the hybrid method. It also avoids randomness and the need to tune parameters. Since these results are shown for $M=1$, mitigating the stochasticity in the Tau-leaping and hybrid methods would require running more simulations, which increases the computation time. This extra cost can be reduced by running the simulations on $P$ processors in parallel, which ideally reduces the total time by a factor of $M/P$.

\paragraph{Tests with two dynamic subsystems} We repeated the comparison for the two-particle kinetic system \eqref{EQ:twoparticles} with $\phi_{12}^1(x,y) = x - 0.4xy$ and $\phi_{21}^2(x,y) = x + 0.4(1-x)(1-y)$. As seen in Figure \ref{fig:tauleaphybrid2part}, the qualitative behavior mirrors the single-particle case: MPCM, Tau-leaping, and the hybrid method capture the same macroscopic dynamics, with stochastic fluctuations again diminishing under ensemble averaging. The difference metric \eqref{eq:diffmetric} exhibits the same $M^{-1/2}$ convergence for stochastic solvers. Computational cost roughly doubles for all three methods due to the higher dimensionality of the state space. Tau-leaping remains the slowest approach; the hybrid method is about ten times faster, and MPCM is the most efficient, running about thirty times faster than the hybrid method. This shows that although higher-dimensional systems require more resources, the relative efficiency of MPCM remains consistent.

Taken together, these results indicate that while stochastic solvers offer a natural probabilistic description of the dynamics, their computational overhead and variability render them less practical for large-scale kinetic systems. MPCM offers a deterministic alternative that preserves accuracy, ensures stability, and achieves substantial efficiency gains.

\newcolumntype{P}[1]{>{\centering\arraybackslash}p{#1}}
\begin{table}[!htbp]
    \centering
            \begin{tabular}{|P{5cm}|P{2cm}|P{2cm}|P{2cm}|}
         \toprule
        \multirow{2}{*}{Model} & \multicolumn{3}{c|}{Time (Seconds)} \\
        \cline{2-4}
          & Tau-leaping Gillespie & Hybrid Method & MPCM with 100 nodes \\
         \midrule
         System \eqref{EQ:oneparticles} with $\delta_{\phi_L(x,y) - u}$ & $364$ & $31$ & $1.4$\\
         \midrule
         System \eqref{EQ:oneparticles} with $\delta_{\phi_R(x,y) - u}$ & $352$ & $29$ & $1.4$\\
         \midrule
         System \eqref{EQ:oneparticles} with $\delta_{\phi_A(x,y) - u}$ & $357$ & $37$ & $1.8$\\
         \midrule
         System \eqref{EQ:oneparticles} with $ \delta_{\phi_T(x,y) - u}$ & $361$  & $35$ & $1.8$\\
         \midrule
         System \eqref{EQ:twoparticles} & $624$ & $62$ & $1.8$\\
         \bottomrule
    \end{tabular}
    \caption{Time to simulate kinetic systems \eqref{EQ:oneparticles} and \eqref{EQ:twoparticles} with different transition rates. The Tau-leaping and hybrid method run for $1$ simulations across $100$ compartments. MPCM is run $100$ times, and the average simulation time is measured to compare performance.}
    \label{tab:solvertimetoruntauhyb}
\end{table}

\subsection{Kinetic Model with Five Subsystems and Multiple Transition Rates}\label{sec:five}

In sociological models, the number of subsystems is usually greater than two. One example is the model presented in \cite{bellomo2013modeling}. In this section, we apply MPCM to a system with five subsystems, denoted as $\{S_k\}_{k=1}^5$, and six transition rates; a subsystem graph connecting all interacting pairs and including directed edges (notated with a corresponding transition kernel) is shown in Figure \ref{fig:sys5}. All transition functions are described in Table \ref{tab:5partTrans}. Our objective in this section is to demonstrate the accuracy and efficiency of MPCM on larger systems. 

\begin{figure}[H]
    \centering
    \begin{subfigure}{0.32\textwidth} \label{fig:tauphiL0}
        \includegraphics[width=\textwidth]{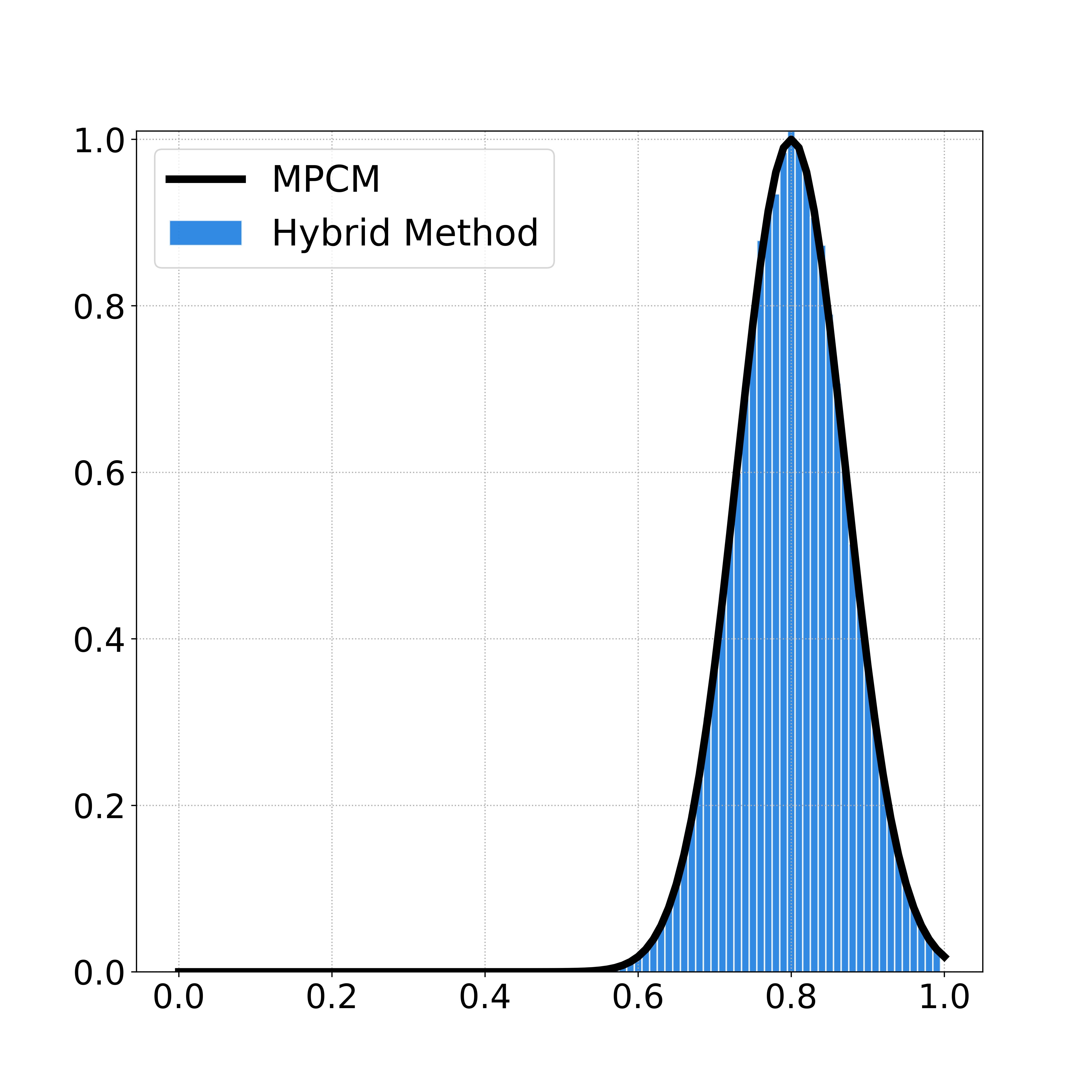}
          \caption{$f_1$ at $t = 0$}
          \end{subfigure}
        \begin{subfigure}{0.32\textwidth}
        \includegraphics[width=\textwidth]{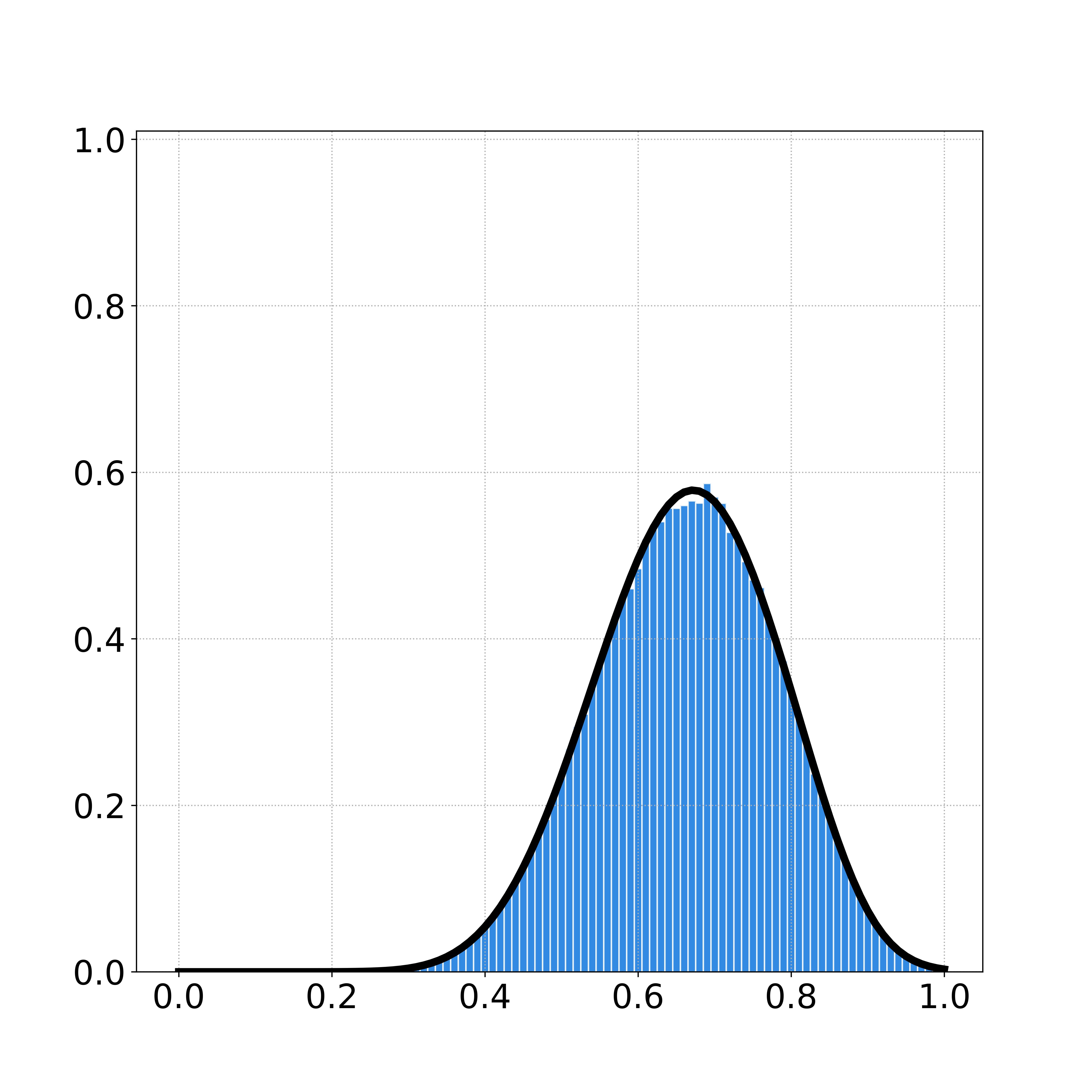}
          \caption{$f_1$ at $t = 10$}\label{fig:tauphiL10}
          \end{subfigure}
          \begin{subfigure}{0.32\textwidth}
        \includegraphics[width=\textwidth]{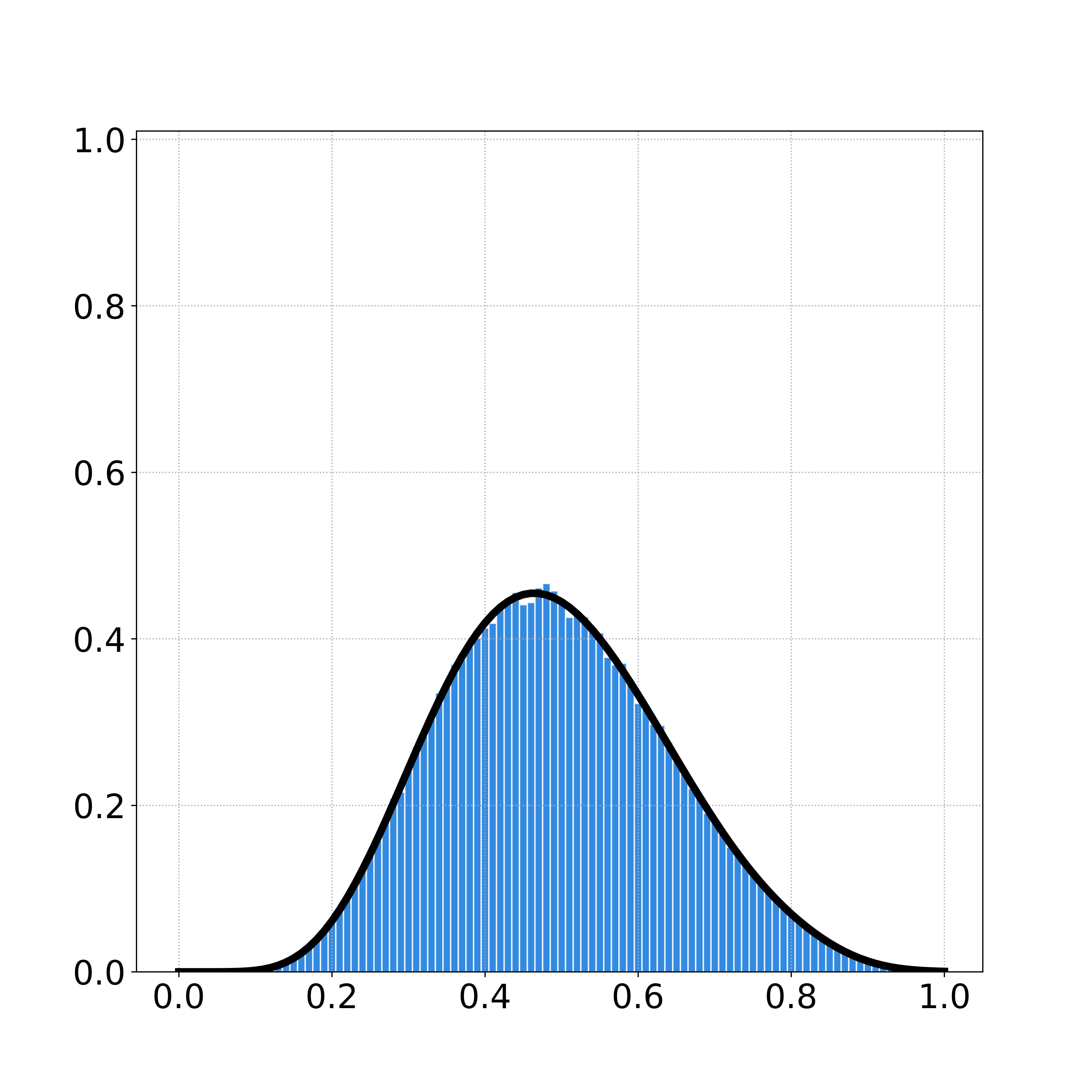}
          \caption{$f_1$ at $t = 20$}\label{fig:tauphiL0}
          \end{subfigure}\\
        \begin{subfigure}{0.32\textwidth} \label{fig:tauphiL0}
        \includegraphics[width=\textwidth]{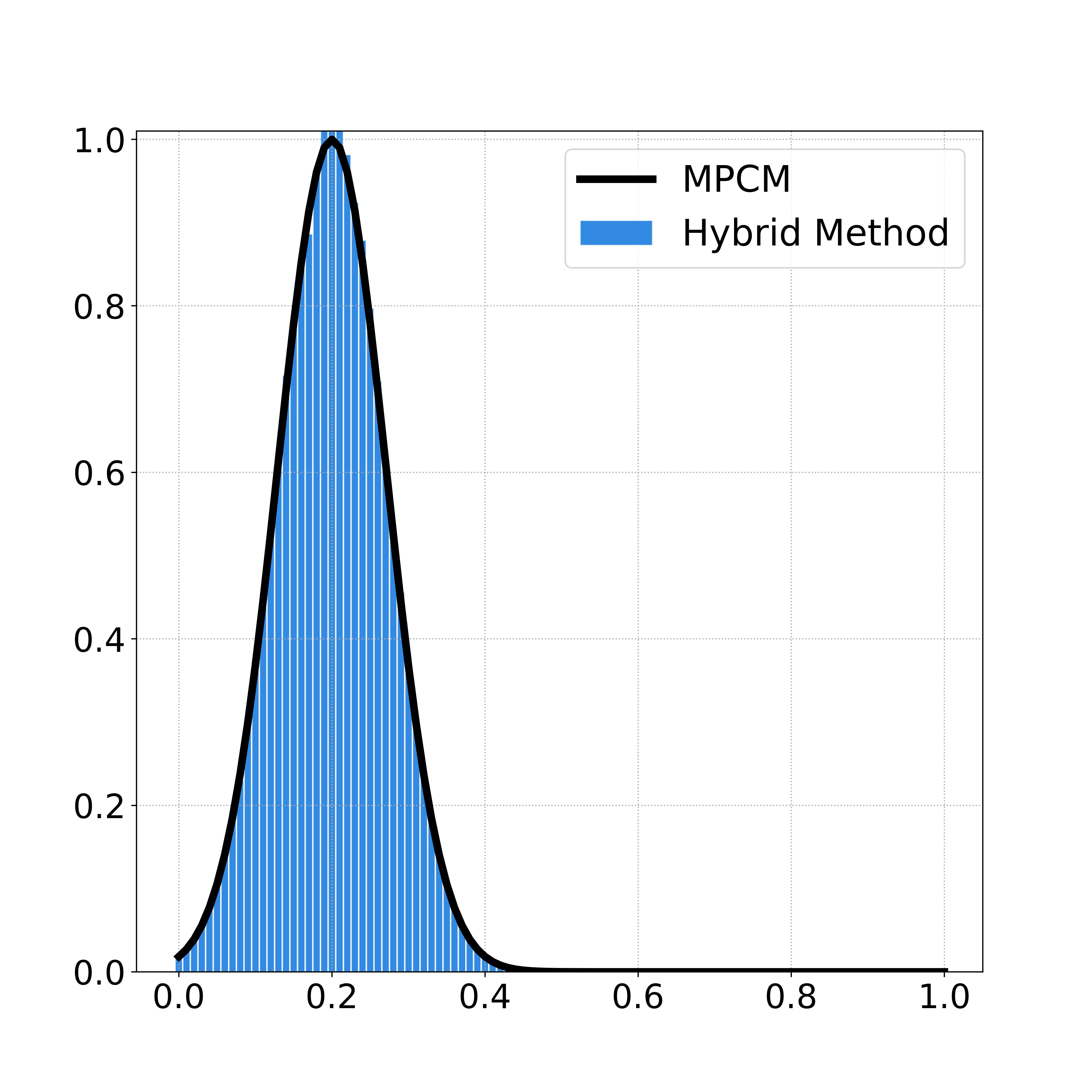}
          \caption{$f_2$ at $t = 0$}
          \end{subfigure}
        \begin{subfigure}{0.32\textwidth}
        \includegraphics[width=\textwidth]{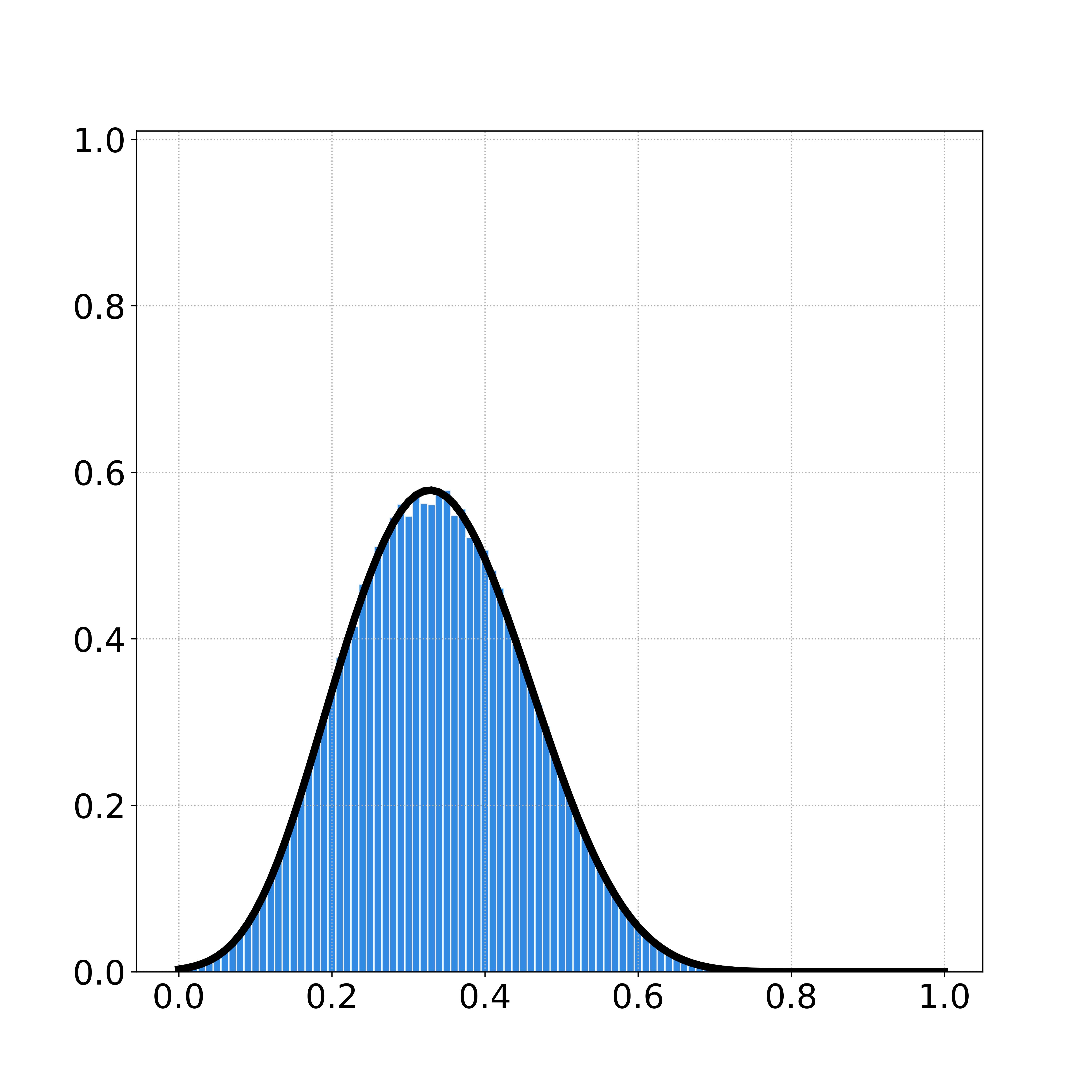}
          \caption{$f_2$ at $t = 10$}\label{fig:tauphiL10}
          \end{subfigure}
          \begin{subfigure}{0.32\textwidth}
        \includegraphics[width=\textwidth]{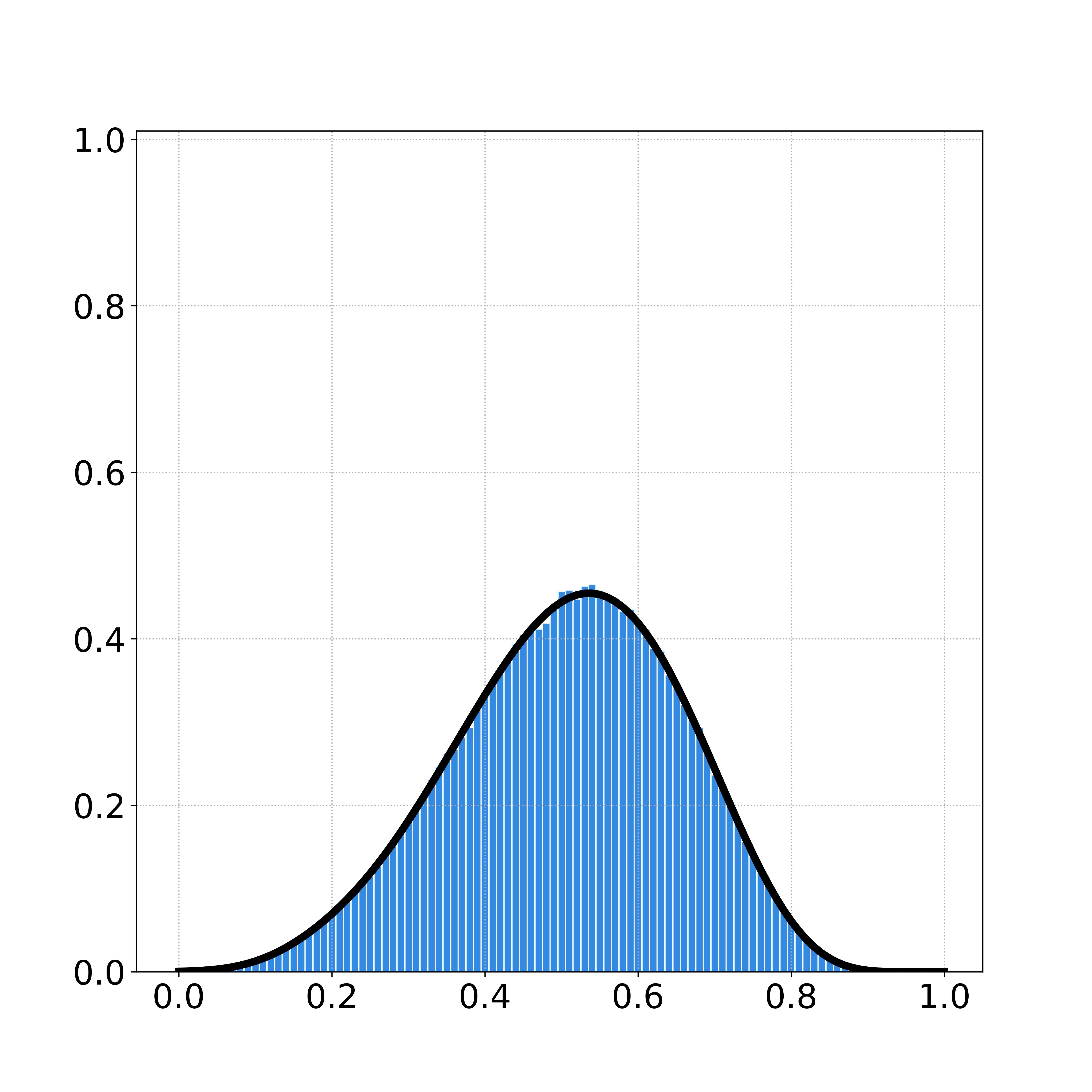}
          \caption{$f_2$ at $t = 20$}\label{fig:tauphiL0}
          \end{subfigure}
\caption{Temporal evolution of the kinetic system \eqref{EQ:twoparticles} with $\phi_{12}^1(x,y) = x - 0.4xy$ and $\phi_{21}^2(x,y) = x + 0.4(1-x)(1-y)$, simulated using the hybrid method and MPCM. The black curve represents the MPCM solution, and the blue bars display the results from the hybrid method. Simulations cover the time interval $[0, 20]$. The Tau-leaping results are not shown, as they are identical to those of the hybrid method. Each stochastic result represents an average over 1000 runs.}
\label{fig:tauleaphybrid2part}
\end{figure}

\paragraph{System dynamics} From the schematic in Figure \ref{fig:sys5}, we can observe that $S_2$ outflows into $S_3$, which in turn outflows into $S_1$; this suggests the mass in $S_2$, and then subsequently that in $S_3$, will both go to zero as time advances. In contrast, the mass in $S_1$ should increase, and the mass in the remaining two (which do not have an inflow or an outflow) should remain constant.  In Figures \ref{fig:sim4f1t0} - \ref{fig:sim4f1t20}, we can observe a few snapshots from the evolution of microstate densities for the populations in each subsystem. These confirm our predictions above, and they allow us to see how the microstate densities change in detail. For example, an additional observation that can be made from the dynamics is that $S_1$ quickly reaches a steady state (as microstate changes are due to interaction with a vanishing population). At the same time, that of $S_4$ and $S_5$ will continue to evolve for longer (as microstate changes are due to interactions with $S_1$ and itself, respectively). 

\begin{figure}[H]
    \centering
    \includegraphics[width=0.8\linewidth]{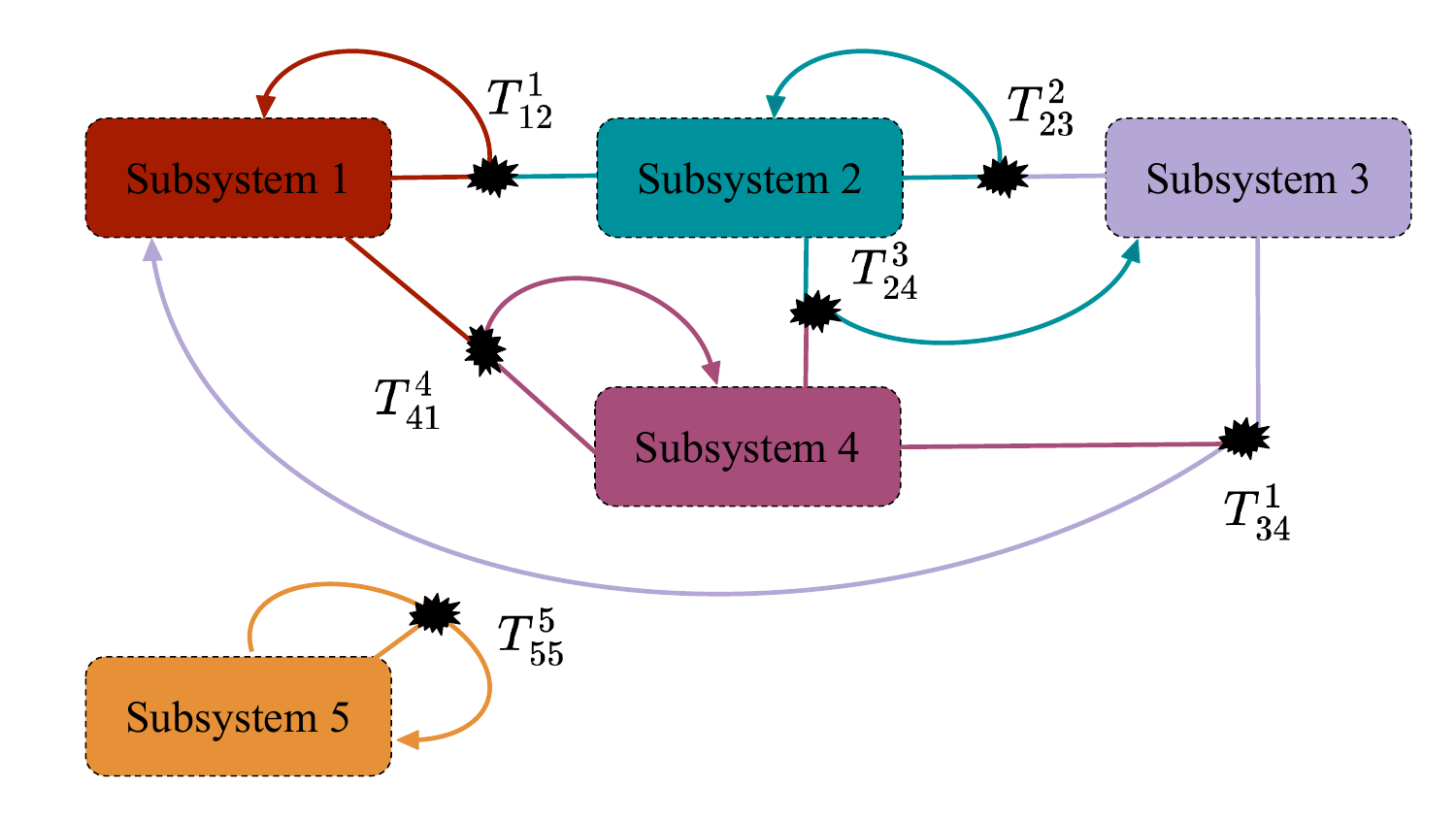}
    \caption{Schematic of our test kinetic system: each node represents a subsystem, with undirected edges between interacting pairs and directed edges reflecting transitions notated by kernel $T_{ij}^k$ (transition functions can be found in Table \ref{tab:5partTrans}).}
    \label{fig:sys5}
\end{figure}

\begin{table}[H]
    \centering
            \begin{tabular}{|p{6.5cm}|p{5.5cm}|}
         \toprule
         Transition function & Description\\
         \midrule
         {\scriptsize $\phi_{12}^1(x,y) = x - 0.4 xy$} & $S_1$ agents decrease microstate upon interaction with $S_2$. \\
         \midrule
        {\scriptsize $\phi_{34}^1(x,y)=x + 0.5 (1-x)(1-y) - u$ }& $S3$ agents transition to $S_1$ and increase microstate upon interaction with $S4$. \\
         \midrule
        {\scriptsize $\phi_{23}^2(x,y) = \begin{cases} x - 0.5 xy, & \text{if } x \leq 0.5,\\ x + 0.5 (1-x)(1-y), & \text{if } x > 0.5.\end{cases}$ }&  $S_2$ moves microstates away from $0.5$ upon interaction with $S3$. \\
         \midrule
        {\scriptsize $\phi_{24}^3(x,y) =x - 0.2xy $ }&$S2$ agents transition to $S_3$ and decrease microstate upon interaction with $S_4$.\\
         \midrule
         {\scriptsize $\phi_{41}^4(x,y) =\begin{cases} x - 0.3 (x-0.5)y, & \text{if } x > 0.5,\\ x + 0.3 (0.5-x)(1-y), & \text{if } x \leq 0.5. \end{cases} $} & $S_4$ moves microstates away from $0.5$ upon interaction with $S_1$.\\
         \midrule
         {\scriptsize $\phi_{55}^5(x,y) =x - 0.4 xy$ }& Agents in $S_5$ decrease microstates upon self-interaction.\\
         \bottomrule
    \end{tabular}
    \caption{   
    List of transition function formulas and physical interpretation of transition functions for the kinetic system described in Figure \ref{fig:sys5}. Subsystems are notated here as $S_k$, $k=1,\dots,5$.}
    \label{tab:5partTrans}
\end{table}

\paragraph{MPCM performance} In Figure \ref{fig:errorruntime5}, we observe self-convergence and runtime scaling results for MPCM consistent with those for smaller kinetic systems. Precomputation time is once again the most expensive stage of the algorithm, taking 4.41 seconds to compute the necessary coefficient tensors. Overall, integrating our system for $N=100$ takes 0.71 seconds to complete on a single workstation. This exercise confirms the practicality of this approach for systems with a moderate to large number of subsystems and reinforces our observations on MPCM performance in Section \ref{sec:KMresults}. 


\begin{figure}[H]
    \centering
    \subcaptionbox{Log-log plot of the self-conv. metric \label{fig:err5}}{\includegraphics[width=.45\textwidth]{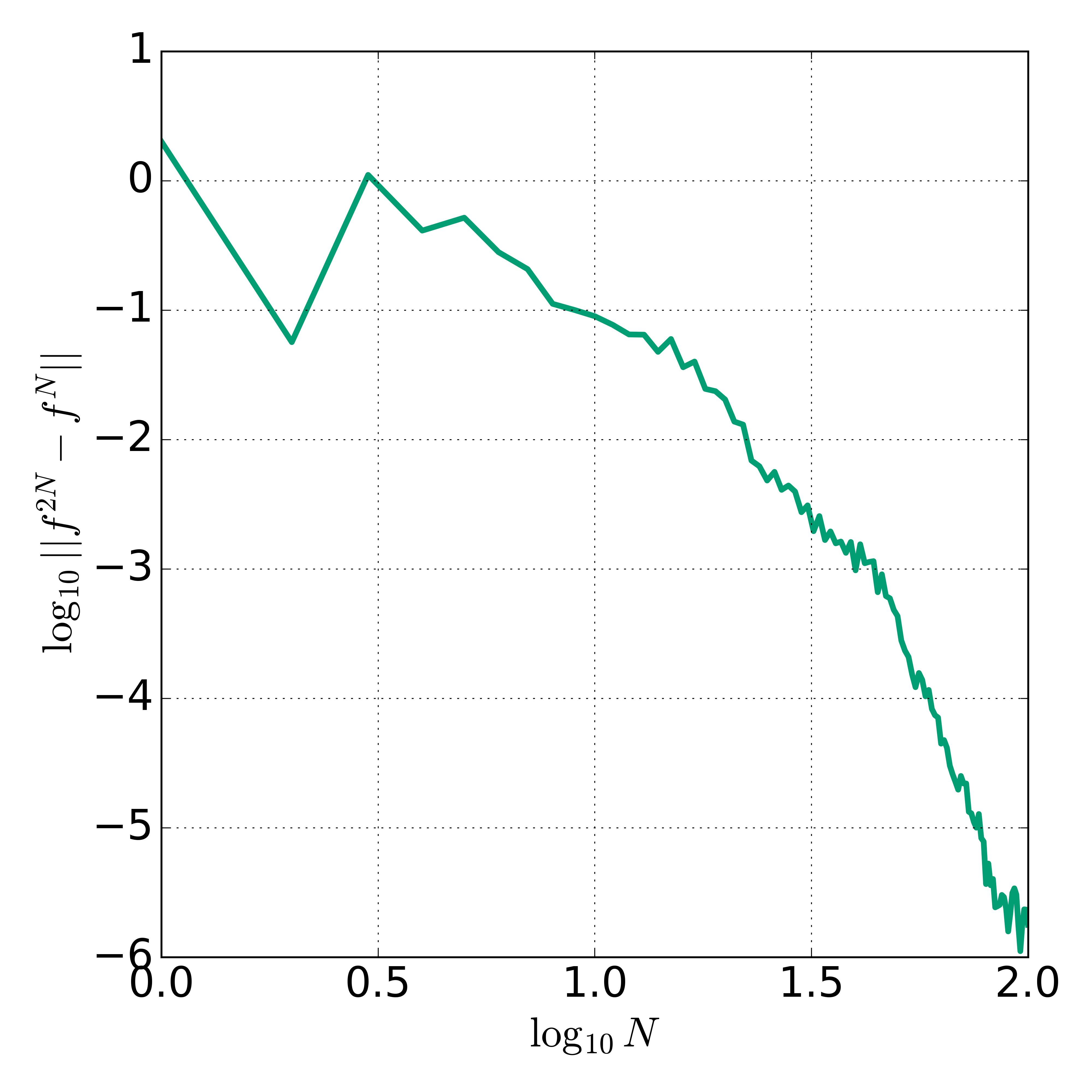}}\hspace{1em}%
    \subcaptionbox{Average time needed to simulate system from Figure \ref{fig:sys5} \label{fig:runtime5}}{\includegraphics[width=.45\textwidth]{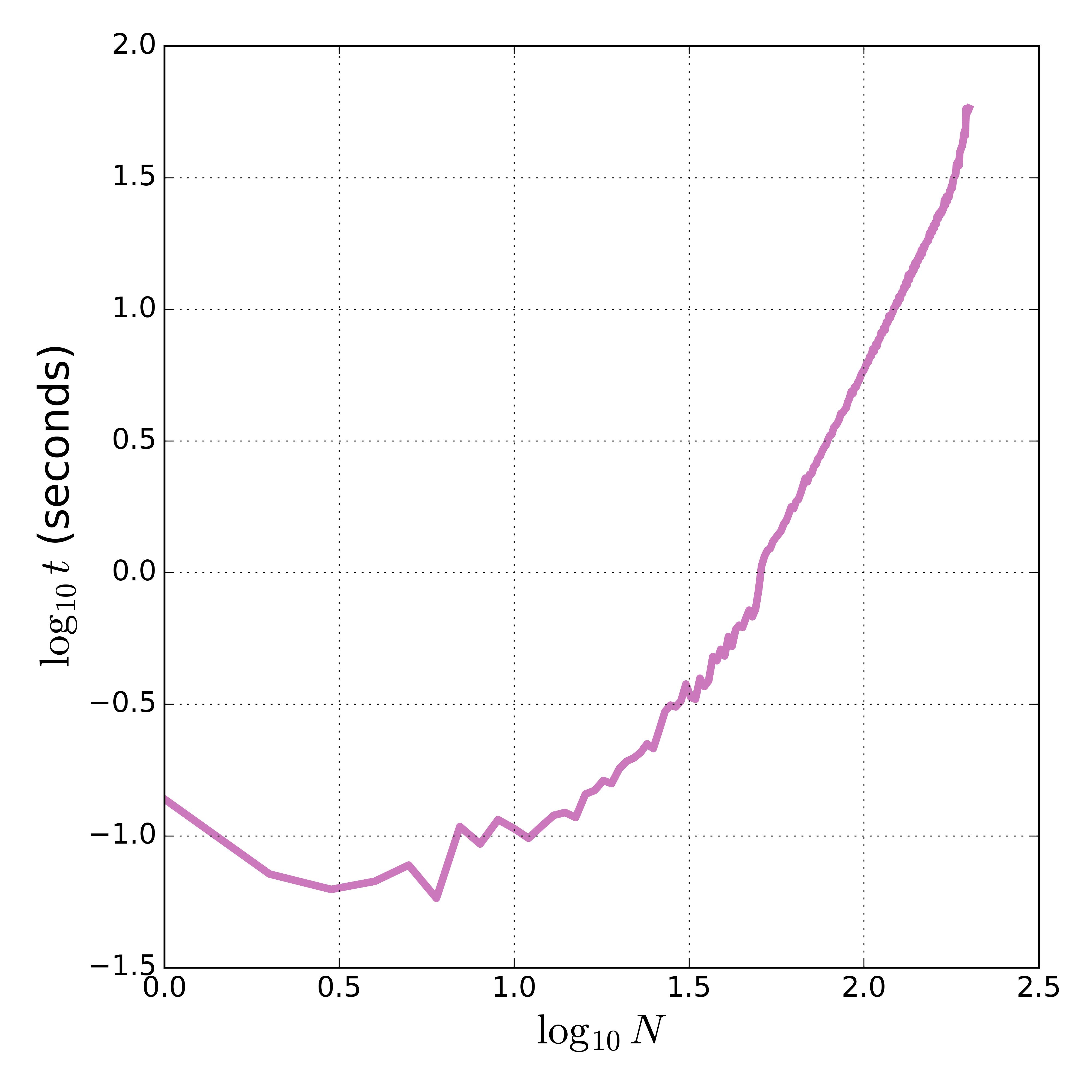}}\hspace{1em}%
    \caption{(a) Log-Log plot of the self-convergence metric of the simulation for system \ref{fig:sys5} compared to the log of the number of nodes $N$.  (b) Log-Log plot of the average time needed to simulate the system from Figure \ref{fig:sys5} as a function of the number of interpolating nodes.}
    \label{fig:errorruntime5}
\end{figure}

\section{Conclusion and Discussion}

We have introduced the mass-preserving collocation method (MPCM) for kinetic models of social dynamics with Dirac delta transition rates. The scheme ensures exact mass conservation and offers significant gains in computational efficiency compared to traditional stochastic methods; even when implemented serially at a high-level language like Python, it is 30 times faster than a shared-memory parallel hybrid method. MPCM's construction is particularly well-suited to models where transition dynamics are concentrated on lower-dimensional manifolds, which arise frequently in sociological modeling \cite{pareschi2013interacting}. As shown in Section \ref{sec:five}, MPCM handles large systems well and can easily model kinetic systems with many subsystems. While a formal convergence proof is left for future work, our experiments show consistent high-order convergence and match both expected qualitative behavior and dynamics obtained by two state-of-the-art stochastic agent-based methods (Tau-leaping and hybrid). In comparison with these stochastic methods, not only is MPCM consistently faster, but it also avoids variability as well as  problem-dependent parameter and threshhold tuning present in these methods, as it only requires the user to set the number of collocation nodes.


\paragraph{Future work} One of the main drivers for the method contributed in this work was the need for efficient solvers to integrate kinetic equations in social systems, the applications of which were of great interest to the authors and collaborators. A natural next step in our ongoing work is thus utilizing MPCM to simulate specific sociological models, especially kinetic models of crime dynamics. These models often describe how people transition between non-criminal and criminal states through interactions with police \cite{short2008statistical, bertozzimathematics}. Because they use Dirac delta transition functions to represent discrete events, like arrests or crime escalations, MPCM is well-suited to track population changes accurately. We are currently developing detailed mesoscale models to investigate the impact of policing on crime and violence in society. Another future research direction involves the use of kinetic models to study online-to-offline spillovers; an example is the interaction of gangs in online environments that escalates into offline violence. This project will build on the work of Leverso and collaborators in \cite{leverso2025measuring}. A natural extension would be to build on the agent-based model proposed in that paper by introducing a suitable kinetic framework, in which the escalation of violent rhetoric online represents a microstate, its increase driving the emergence of offline violence. Other possible applications could be building on Toscani's opinion dynamics \cite{toscani2006kinetic}, building simulations for biological kinetic systems \cite{resat2009kinetic}, simulating the market economy \cite{cordier2005kinetic}, or even simulating pedestrian dynamics \cite{cristiani2014multiscale}.

While MPCM broadens the range of kinetic equation models that we can solve efficiently, there are several useful extensions worthy of consideration. As was discussed in Section 4, a similar methodology can be used to develop a number of Galerkin schemes. In addition, we envision that a number of computational challenges will arise in developing MPCM for models with time-dependent or stochastic transition rates, which could capture complexity and randomness in social systems. This scheme could also potentially be extended to models in which parameters and kernels vary as a result of spatial heterogeneity. These extensions would further broaden the applicability of our methodology and strengthen its potential as a tool for the quantitative analysis of social dynamics.
\\\\
{\bf Acknowledgements:}  Rodr\'iguez and Bahid were partially funded by NSF-DMS-2042413 and AFOSR MURI FA9550-22-1-0380. Corona acknowledges support by NSF-DMS-2309661.

\newpage
\appendix
\section{Proof of Theorem \ref{Th:wellpos}}\label{app:prf}

\begin{proof}
We equip the space $\calC$ with the norm 
$$
\norm{f}= \sup_{t\in[0,T_f]}\sum_{i=1}^N \norm{f_i(t,\cdot)}_{L^1} = \sup_{t\in[0,T_f]}\sum_{i=1}^N \int_{D_i} |f_i(t,u)| du,
$$
making $(\calC, \norm{\cdot})$ a Banach Space. 
Consider the space:
$$\calE = \{f = (f_1,...,f_n) |\ f \in \calC \ ,\ \norm{f} \leq a \text{ and } f_i(t,u) \geq 0,\ \forall (t,u)\in [0,T_f]\times D_i \;\text{for}\; i \in \mathbb{S}\},$$
where $T_f = \frac{1}{C_{\eta}}\min\left(\ln{\frac{4a + 1}{4a}},\ln{\frac{2a^2+a}{2a^2+1}} \right)$ and $a > 1$. 
%
Define the operator $\calT: \calE \rightarrow \calC$ such that:
\begin{equation*}
\begin{split}
    (\calT f)_i &= e^{-C_\eta t}\Bigg[ f_i^0(u) +\int_0^t e^{C_\eta s}\mathlarger{\Bigg(} \sum_{h,k} \int_{D_h} \int_{D_k}\eta_{hk}(x,y)T_{hk}^i(x,y,u)f_h(s,x)f_k(s,y)dxdy \\
    &\phantom{\sum_{h,k} \int_0^1 \int_0^1\eta_{hk}}- f_i(s,u)\Bigg(\sum_h\int_{D_h}\eta_{ik}(u,x)f_h(s, x)dx - C_{\eta} \Bigg)\mathlarger{\Bigg)}  ds \Bigg]\quad \forall i \in \mathbb{S},
\end{split}
\end{equation*}
and note that a fixed point of $\;\calT$ is a solution to system \eqref{eq:integrodiffthrm} with initial conditions $f_0$.
We first show that $\ \calT(\calE) \subset \calE$. To accomplish this, note that for all $t \in [0, T_f]$

\begin{align*}
    \sum_i \int_{D_i} |\ (\calT f)_i(t,u)| du 
    &= \sum_i \int_{D_i} \Biggl\lvert e^{-C_\eta t}\Bigg[ f_i^0(u) du \\
    & + \int_0^t e^{C_\eta s}\Bigg( \sum_{h,k} \int_{D_h} \int_{D_k}\eta_{hk}(x,y)T_{hk}^i(x,y,u)f_h(s,x)f_k(s,y)dxdy \\
    &- f_i(s,u)\bigg(\sum_h\int_{D_k}\eta_{ik}(u,x)f_h(s, x)dx - C_{\eta}\bigg)\Bigg)ds  \Bigg]\Biggl\rvert du\\ 
    &\leq  e^{-C_\eta t}\bigg[ \sum_i \int_{D_i} |f_i^0(u)| du \\
    & +   \int_0^t e^{C_\eta s}\bigg( \sum_{h,k} \int_{D_h} \int_{D_k}\eta_{hk}(x,y)\\
    &\bigg( \sum_{i} \int_{D_i} T_{hk}^i(x,y,u) du\bigg)\lvert f_h(s,x)f_k(s,y)\rvert dxdy \\
    &+\sum_i \int_{D_i}\biggl\lvert f_i(s,u)\bigg(\sum_h\int_{D_k}\eta_{ik}(u,x)f_h(s, x)dx   - C_{\eta} \bigg) \biggl\rvert du \bigg) ds \bigg].
\end{align*} 
Using the assumption that the transition rates satisfy condition \eqref{eq:trans_condition}, we then obtain that
\begin{align*}
 \sum_i \int_{D_i} |\ (\calT f)_i(t,u)| du &\leq  e^{-C_\eta t}\bigg[ \sum_i \int_{D_i} |f_i^0(u)| du \\
    & \phantom{\sum_{h,k} \int_{D_h}}+   \int_0^t e^{C_\eta s}\bigg( C_{\eta}\bigg[\sum_{h} \int_{D_h} |f_h(s,x)| dx \bigg] \bigg[\sum_{k} \int_{D_k} |f_k(s,y)| dy\bigg]  \\
    &\phantom{\sum_{h,k} \int_{D_h}}+\sum_i \int_{D_i}\lvert f_i(s,u)\lvert\bigg(\sum_h\int_{D_k}C_{\eta}\lvert f_h(s, x)\lvert dx + C_{\eta} \bigg) du \bigg) ds \bigg]\\
    &\leq e^{-C_{\eta}t}\bigg[ \norm{f^0}_{L^1} + \int_0^t e^{C_{\eta} s}\bigg( C_{\eta}\norm{f}^2  +\norm{f}\bigg( C_{\eta}\norm{f} + C_{\eta} \bigg) \bigg) ds \bigg]\\ 
    &\leq  e^{-C_\eta t}\bigg[ 1 +  \int_0^t (2a^2+a)C_\eta e^{C_\eta s} ds \bigg] \\
    &\leq e^{-C_{\eta}t}(1 + (2a^2+a)(e^{C\eta t} - 1))\\
    &\leq   e^{-C_{\eta}t} + (2a^2+a)(1 -  e^{-C_{\eta}t}).
\end{align*}
Therefore, we have that
\begin{align*}
    \norm{\ \calT f} &= \sup_{t\in [0,T_f]} \sum_i \int_{D_i} |\ (\calT f)_i(t,u)|  du\\
    &\leq \sup_{t\in [0,T_f]}   e^{-C_{\eta}t} + (2a^2+a)(1 -  e^{-C_{\eta}t})\\
    &\leq 1 + (2a^2+a)(1 -  e^{-C_{\eta}T_f}).
\end{align*}
Given that $T_f \leq \frac{1}{C_{\eta}}\ln{\frac{2a^2+a}{2a^2+1}}$, we observe that $\ \calT\ $   is an operator from $\calE$ to  $\calE$.
To show that $\ \calT\ $ is a contraction, take two functions $f,g \in \calE$. We have for $t \in [0, T_f]$ that 
\begin{align*}
    \norm{\ \calT f - \calT g}
    &=  \sup_{t\in [0,T_f]} \sum_i \int_0^1 |\ \calT f - \calT g|(t,u) du\\
    &=  \sup_{t\in [0,T_f]} \sum_i \int_{D_i}\Biggl\lvert  e^{- C_\eta t}\Bigg[\int_{D_i} f_i^0(u) - f_i^0(u)du \\
    &\phantom{\sum_{h,k} \int_0^1 \int_0^1\eta_{hk}} + \int_0^t e^{C_\eta s}\Bigg( \sum_{h,k} \int_{D_h} \int_{D_k}\eta_{hk}(x,y)T_{hk}^i(x,y,u)[f_h(s,x)f_k(s,y)\\
    &\phantom{\sum_{h,k} \int_0^1 \int_0^1\eta_{hk}}-g_h(s,x)g_k(s,y)]dxdy\\
    &\phantom{\sum_{h,k} \int_0^1 \int_0^1\eta_{hk}}- f_i(s,u)\bigg( \sum_h\int_{D_h}\eta_{ik}(u,x)f_h(s, x) - C_\eta \bigg) \\
    &\phantom{\sum_{h,k} \int_0^1 \int_0^1\eta_{hk}}+g_i(s,u)\bigg(\sum_h\int_{D_h}\eta_{ik}(u,x)g_h(s,x)dx -C_{\eta} \bigg)\Bigg)ds \Bigg] \Biggl\lvert du.
    \end{align*}
    Therefore, we get
\begin{align*}
    \norm{\ \calT f - \calT g}
    &\leq \sup_{t\in [0,T_f]}  e^{-C_\eta t}\bigg[ \int_0^t e^{C_\eta s}\bigg( \sum_{h,k} \int_{D_h} \int_{D_k} \eta_{hk}(x,y)\bigg(  \sum_i \int_{D_i} T_{hk}^i(x,y,u) du\bigg)\\
    &\phantom{\sum_{h,k} \int_0^1 \int_0^1\eta_{hk}}\biggl\lvert f_h(s,x)f_k(s,y) -g_h(s,x)g_k(s,y) \biggl\rvert dxdy \\
    &\phantom{\sum_{i,k} \int_0^1 \int_0^1\eta_{hk}} + \sum_{i,k} \int_{D_i} \int_{D_k} C_\eta \lvert f_i(s,x)f_k(s,y) - g_i(s,x)g_k(s,y) \rvert dxdy\\
    &\phantom{\sum_{i,k} \int_0^1 \int_0^1\eta_{hk}} + C_{\eta}\sum_i \int_{D_i} |f_i(t,u) - g_i(t,u)| du \bigg)ds \bigg]. 
\end{align*}
Hence, it holds that
\begin{align*}
    \norm{\ \calT f - \calT g}
    &\leq \sup_{t\in [0,T_f]}  e^{-C_\eta t}\bigg[ \int_0^t e^{C_\eta s}\bigg( \sum_{h,k} \int_{D_h} \int_{D_k} 2C_{\eta} \biggl\lvert \bigg(f_h(s,x)f_k(s,y)\\
    &\phantom{\sum_{h,k} \int_0^1 \int_0^1\eta_{hk}}-g_h(s,x)g_k(s,y) + f_h(s,x)g_k(s,y) - f_h(s,x)g_k(s,y)\bigg)\biggl\rvert dxdy  \\
    &\phantom{\sum_{h,k} \int_0^1 \int_0^1\eta_{hk}aaaaa} + C_{\eta}\sum_i \int_{D_i} |f_i(t,u) - g_i(t,u)| du \bigg)ds \bigg]. 
    \end{align*}
Therefore,
\begin{align*}
    \norm{\ \calT f - \calT g}
    &\leq \sup_{t\in [0,T_f]}  e^{-C_\eta t}\bigg[ \int_0^t e^{C_\eta s}\bigg(\sum_{h,k} \int_{D_h} \int_{D_k} 2C_\eta \biggl\lvert f_h(s,x)(f_k(s,y) - g_k(s,y)) \\
    &\phantom{\sum_{h,k} \int_0^1 \int_0^1\eta_{hk}aaaaa} - g_k(s,y)(f_h(s,x) - g_h(s,x))\biggl\rvert dxdy \\
    &\phantom{\sum_{h,k} \int_0^1 \int_0^1\eta_{hk}aaaaa} + C_{\eta}\sum_i \int_{D_i} |f_i(t,u) - g_i(t,u)| du\bigg)ds \bigg] \\
    &\leq \sup_{t\in [0,T_f]}  e^{-C_\eta t}\bigg[ \int_0^t e^{C_\eta s}\bigg( C_\eta \bigg[ 4a \norm{f-g} + \norm{f-g} \bigg] \bigg)ds \bigg]\\
    &\leq  (4a+1)[1 - e^{-C_\eta T_f}]\norm{f-g} .
\end{align*}

Given that $T_f < \frac{1}{C\eta}\ln{\frac{4a + 1}{4a}}$, we see that $(4a+1)[1 - e^{-C_\eta T_f}] < 1$.
Therefore, by the contraction mapping theorem, system \eqref{eq:integrodiffthrm} accepts a solution in $\calE$ with $T_f = \frac{1}{C_{\eta}}\min\bigg(\ln{\frac{4a + 1}{4a}},\ln{\frac{2a^2+a}{2a^2+1}} \bigg)$.
If $T<T_f$, then there exists a solution in $\calC$. Otherwise, we know that $T_f$ is independent of the initial state $f_0$. Thus, by setting $f_0(u) = f(T_f,u)$ and having a new final time of $2T_f$, we can reapply the contraction mapping theorem to find a solution on $[T_f, 2T_f]$. Stitching these solutions together up to $T$ yields a unique solution over $[0, T]$.
$$\phantom{adfsadfasgasfgsfgsagf}$$
\end{proof}

\section{Asymptotic Behavior of Solutions for Different $\phi$}\label{app:asymp}
In this section, we aim to gain insight into the long-term behavior under different choices of the function $\phi$. To build intuition, we consider a single equation with one subsystem. We will analyze the first and second moments to determine how the distribution evolves. The kinetic equation is given by
\begin{equation}\label{eq:onesubsystem}
  \partial_t f(t,u) = \int_0^1 \int_0^1 \delta_{\phi(x,y) - u} f(t,x)f(t,y) dx dy - f(t,u)\int_0^1 f(t,y) dy .
\end{equation}
We define the first (mean) and the second moments as
\begin{align*}
    m_1(t) = \int_0^1 u f(t,u)du,\; m_2(t) = \int_0^1 u^2 f(t,u)du,\;\text{and}\;\mathrm{Var}(t) =  m_2(t) - m_1(t)^2.
\end{align*}

Note that taking the time derivative of the two moments and using equation \eqref{eq:onesubsystem} will yield a different equation for each function $\phi$ listed in Table \ref{tab:phitab}. For $\phi_L$, we get the following system of ODEs
\begin{equation}\label{eq:momphiL}
        \begin{cases}
        \frac{d}{dt} m_1(t) &= -\gamma m_1^2(t),\vspace{3pt}\\
        \frac{d}{dt} \mathrm{Var}(t) &= -2\gamma m_1(t)\mathrm{Var}(t).
\end{cases}\end{equation}

From system \eqref{eq:momphiL}, we conclude that both $m_1(t) \to 0$ and $\mathrm{Var}(t) \to 0$ as $t \to \infty$. Therefore, the mass accumulates at the $u = 0$ boundary as time progresses. For the function $\phi_R$, the system of ODEs takes the form
\begin{equation}\label{eq:momphiR}
\begin{cases}
    \frac{d}{dt} m_1(t) &= \gamma (1-m_1(t))^2,\vspace{3pt}\\
    \frac{d}{dt} \mathrm{Var}(t) &= -2\gamma m_1(t)\mathrm{Var}(t).
\end{cases}\end{equation}

From system \eqref{eq:momphiR}, we conclude that $m_1(t) \to 1$ and $\mathrm{Var}(t) \to 0$, so the mass is shifting towards the $u = 1$ boundary in the long term.
For {\it $\phi_T$}, the moments satisfy
     \begin{equation}\label{eq:momphiT}\begin{cases}
         \frac{d}{dt} m_1(t) &= -\gamma (m_1-a)^2(t),\\
         \frac{d}{dt} \mathrm{Var}(t) &= -2\gamma |m_1(t)-a|\mathrm{Var}(t).
    \end{cases}\end{equation}

From system \eqref{eq:momphiT}, we find that as $t \to \infty$, $m_1(t) \to a$ and $\mathrm{Var}(t) \to 0$. This implies that the mass concentrates at $u = a$ over time.
Finally, for $\phi_A$, we decompose the moments into left and right components
$$ \begin{cases}
    m_i^L(t) &= \int_0^a u^i f(t,u)du,\\
    m_i^R(t) &= \int_a^1 u^i f(t,u)du,
\end{cases}$$

which yields the system of ODEs
\begin{equation}\label{eq:momphiA}
\begin{cases}
    \frac{d}{dt} m_1^L(t) &=  -\gamma (m_1^L(t))^2\\
     \frac{d}{dt} m_1^R(t) &=  \gamma (1-m_1^R(t))^2.\\
\end{cases}\end{equation}
System \eqref{eq:momphiA} yields two subpopulations. The mass of the first shifts towards  $0$, while the second's mass shifts towards $1$. This provides insight into the dynamics expected from the solutions, depending on the choice of $\phi$.

\section{Failure of the Standard Collocation Scheme and Convergence of $C\{a_{ij}\}$}\label{app:naivefail}

We seek to compare the two schemes \eqref{ODESYS} and \eqref{NumIntDiff} through the modeling of a kinetic system of two subsystems with one transition rate. The integro-differential equation that governs this system takes the form \eqref{EQ:oneparticles}. For the purposes of this case study, we set $T_{12}^1(x,y,u) = \delta_{x - 0.9xy - u}$. We plot the mass of the entire system of both schemes for different numbers of collocation nodes. Figure \ref{fig:errmassnaive} shows that as time progresses, the mass of the system increases regardless of the number of nodes, while the MPCM preserves the mass as shown in Figure \ref{fig:errmassmpcs}. This is even clearer in Figure \ref{fig:loglogmass}, where we can see the log-log plot of the maximum difference between the mass and $1$ by number of collocation nodes. We observe that the MPCM significantly decreases the difference. This is due to the term $C(\{a_{ij}\})$, which acts as a correction to the mass and that gets closer to 1 as the number of collocation points increases, as seen in Figure \ref{fig:errmassmpcs}

\begin{figure}[H]
    \centering
    \subcaptionbox{\label{fig:errmassnaive}}{\includegraphics[width=.45\textwidth]{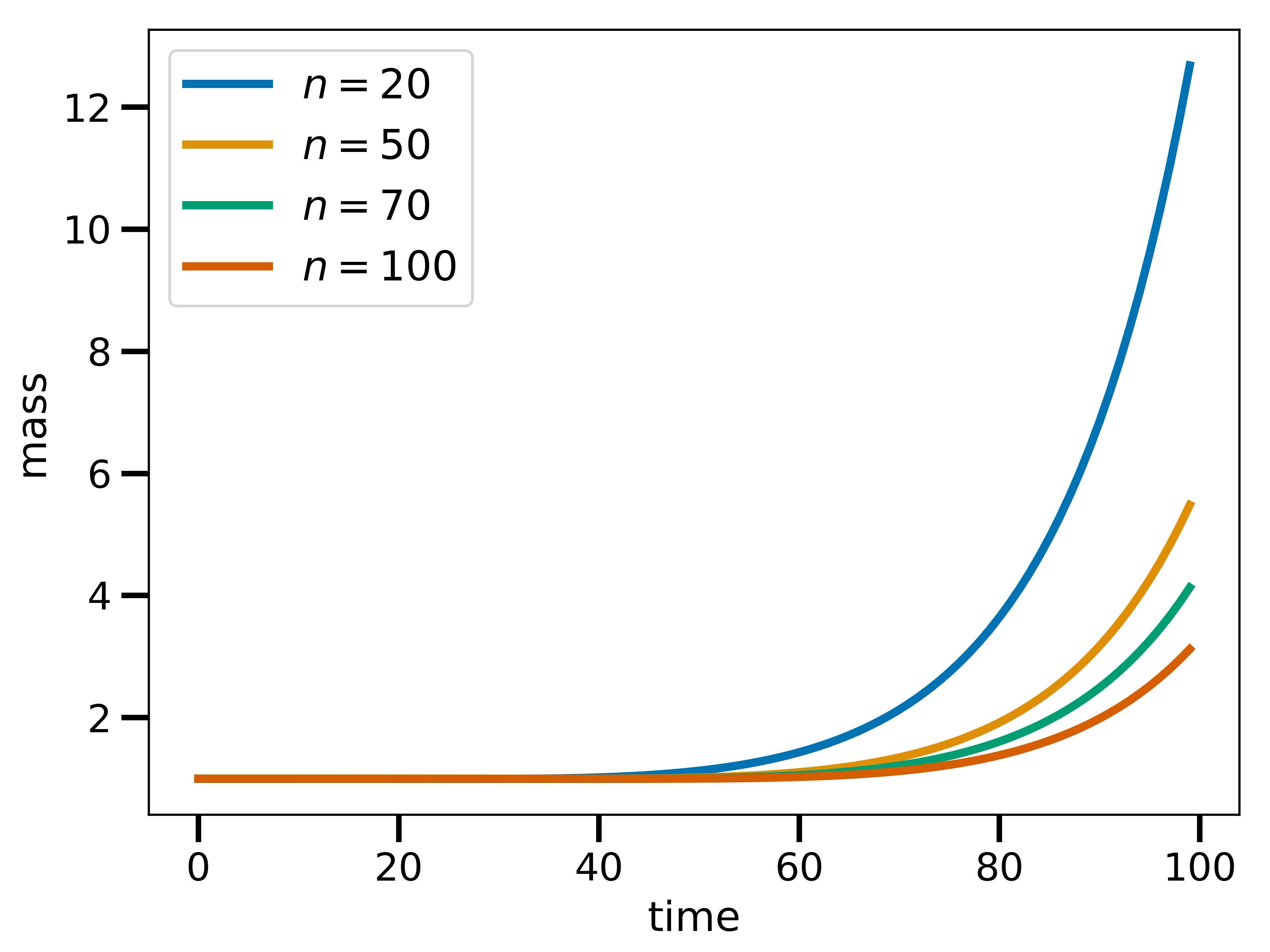}}\hspace{1em}%
    \subcaptionbox{\label{fig:errmassmpcs}}{\includegraphics[width=.45\textwidth]{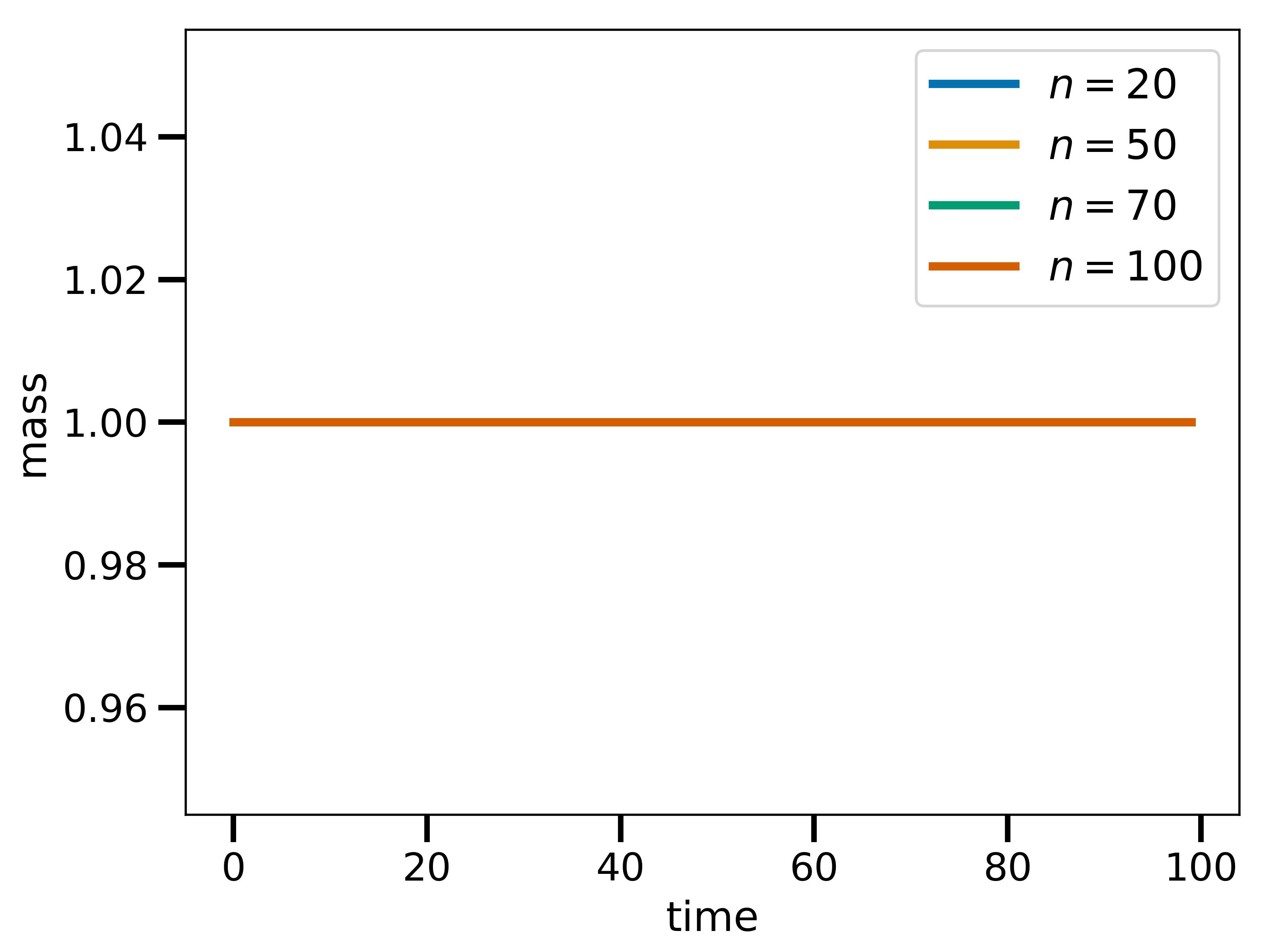}}\hspace{1em}%
    \caption{Evolution of mass over time for system \eqref{EQ:oneparticles} with $T_{12}^1(x,y,u) = \delta_{x - 0.9xy - u}$ simulated using   (a) - scheme \eqref{NumIntDiff} and (b) -  scheme \eqref{ODESYS}.}
    \label{fig:errormass}
\end{figure}

\begin{figure}[H]
    \centering
    \subcaptionbox{\label{fig:loglogmass}}{\includegraphics[width=.45\textwidth]{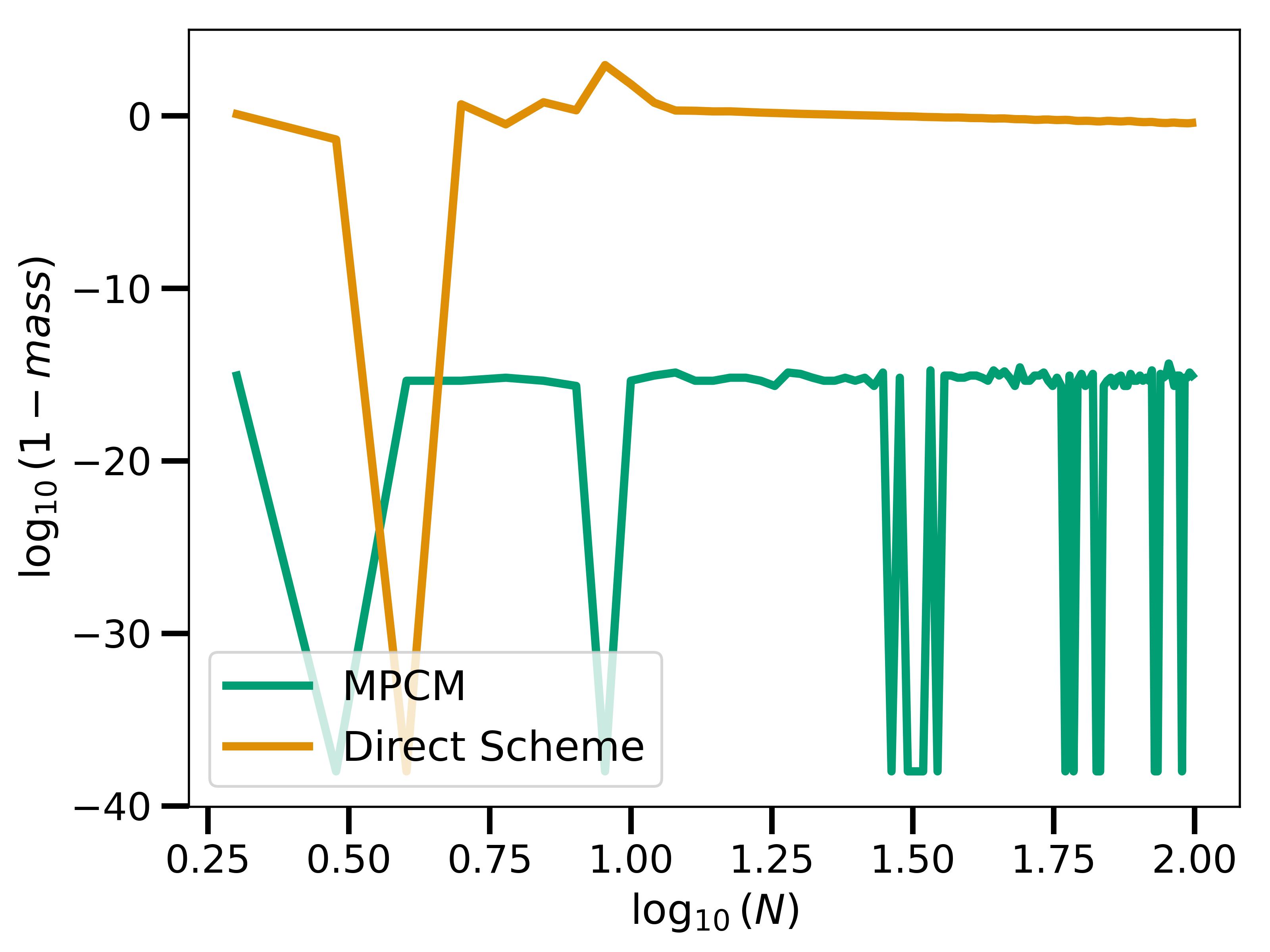}}\hspace{1em}%
    \subcaptionbox{\label{fig:errmassmpcs}}{\includegraphics[width=.45\textwidth]{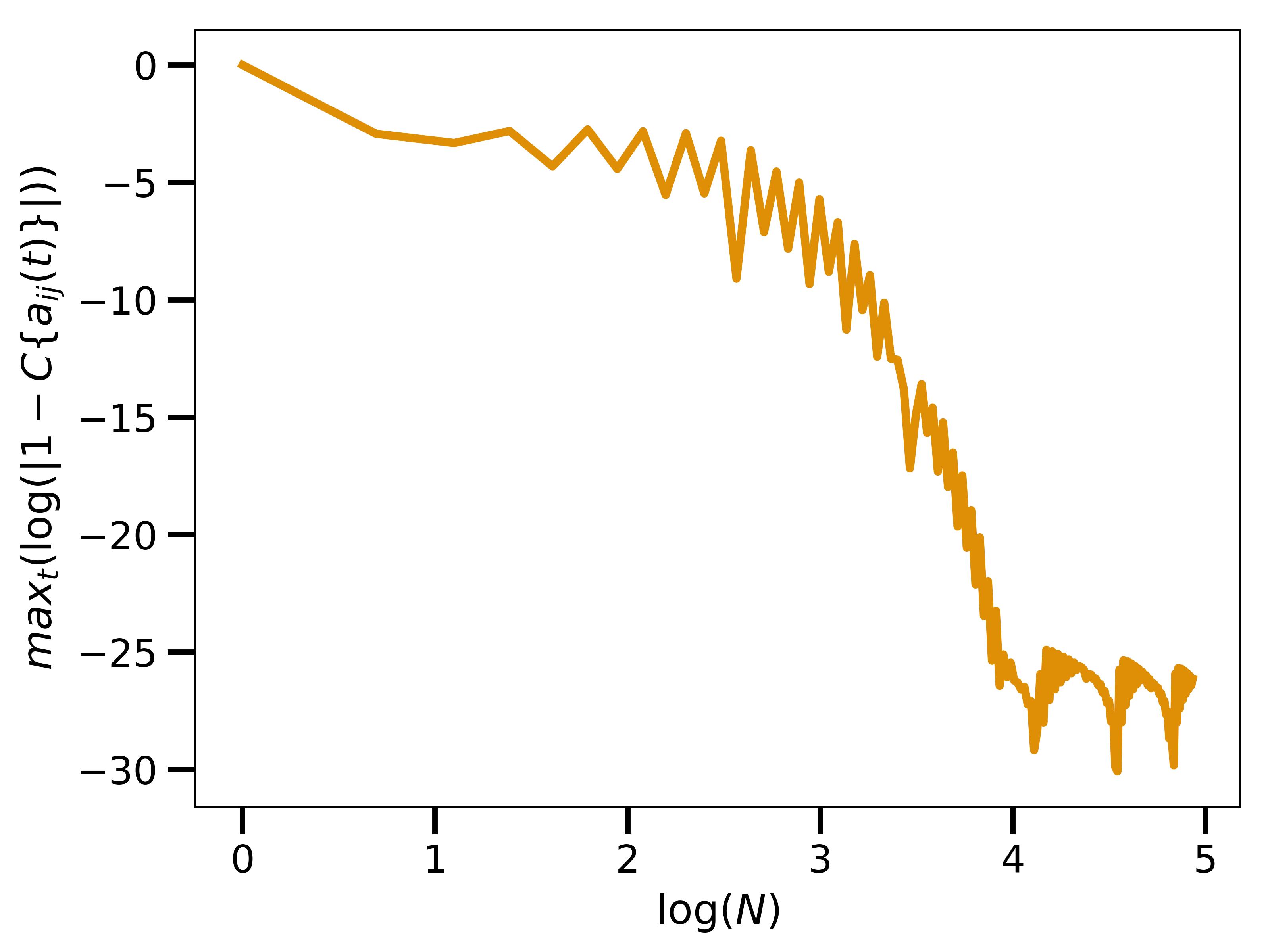}}\hspace{1em}%
    \caption{Log-log plot of the difference in (a) mass and (b) - the coefficient $C\{a_{ij}\}$ for system \eqref{EQ:oneparticles} with $\phi_{12}^1 = \phi_L$ for different numbers of collocation points $N$.}
    \label{fig:errorlogmassc}
\end{figure}

\newpage
\section{Kinetic System with One Transition Rate and Different $\gamma$'s } \label{app:diffgam}
We model the kinetic system \eqref{EQ:oneparticles} with different transition functions $\phi$ from Table \ref{tab:phitab} and different values of $\gamma$.
\begin{figure}[H]
    \centering
    \begin{subfigure}{0.2\textwidth}
        \includegraphics[width=\textwidth]{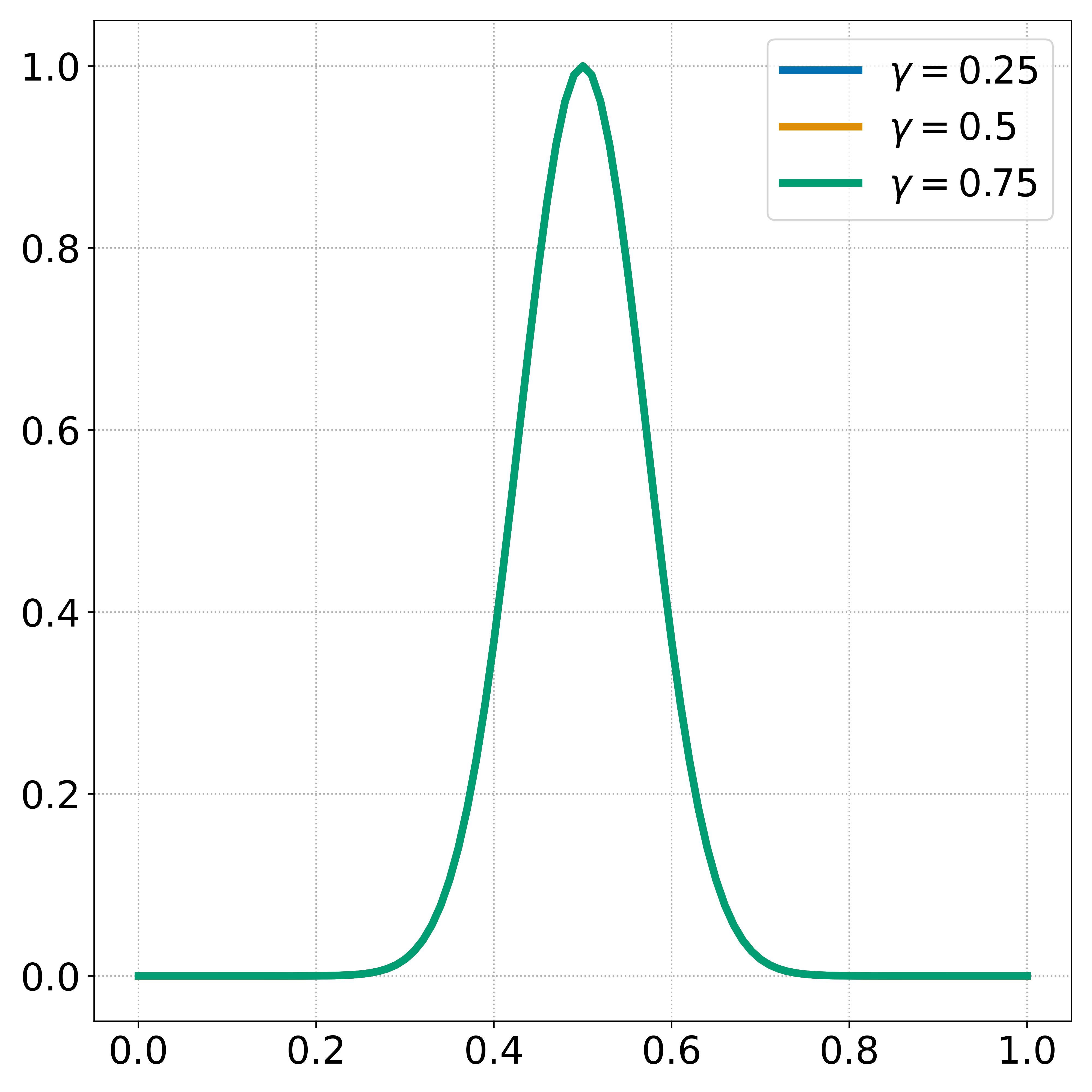}
          \caption{$\phi_L$ at $t = 0$} \label{fig:phiLt0}
          \end{subfigure}
        \begin{subfigure}{0.2\textwidth}
        \includegraphics[width=\textwidth]{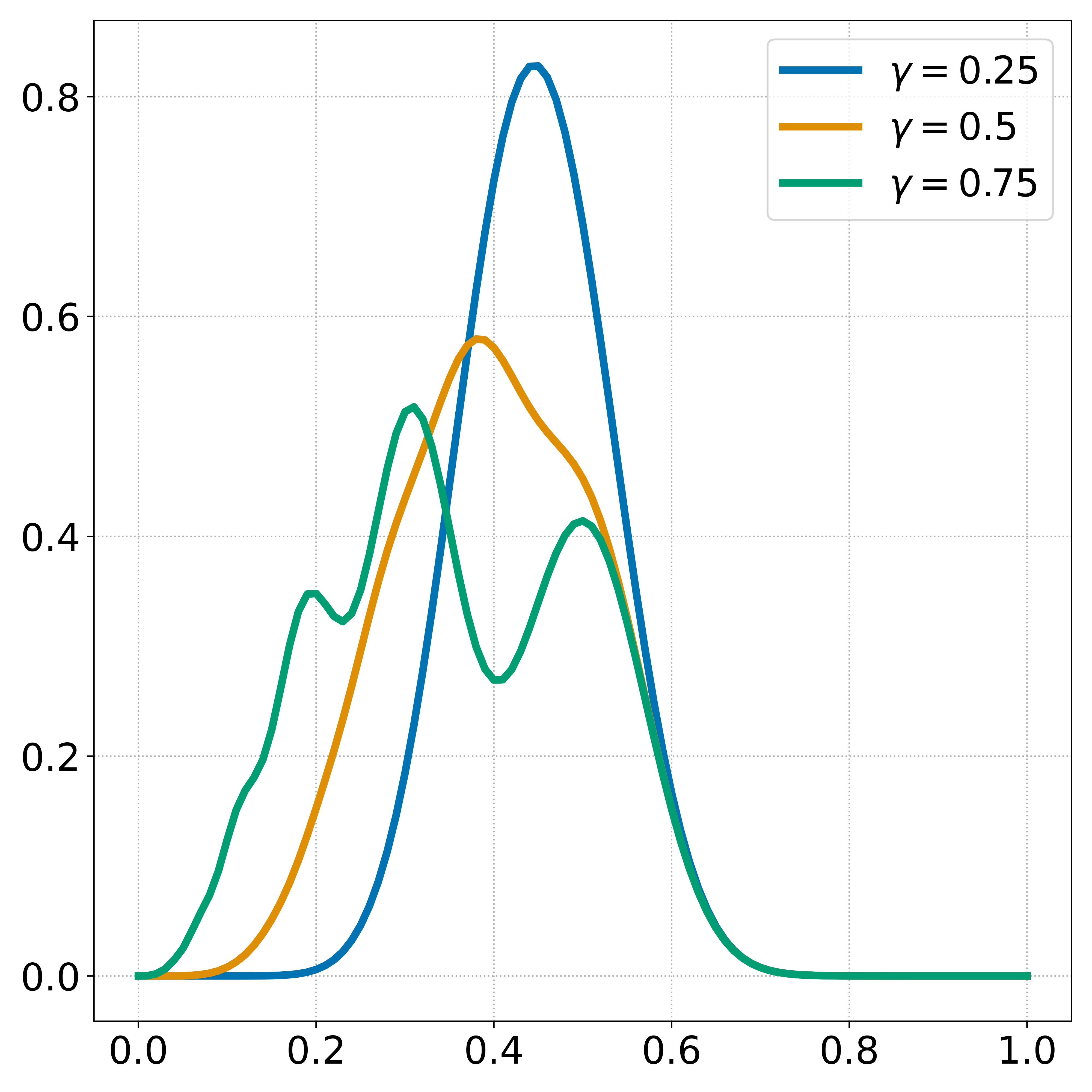}
          \caption{$\phi_L$ at $t = 5$}\label{fig:phiLt5}
          \end{subfigure}
        \begin{subfigure}{0.2\textwidth}
        \includegraphics[width=\textwidth]{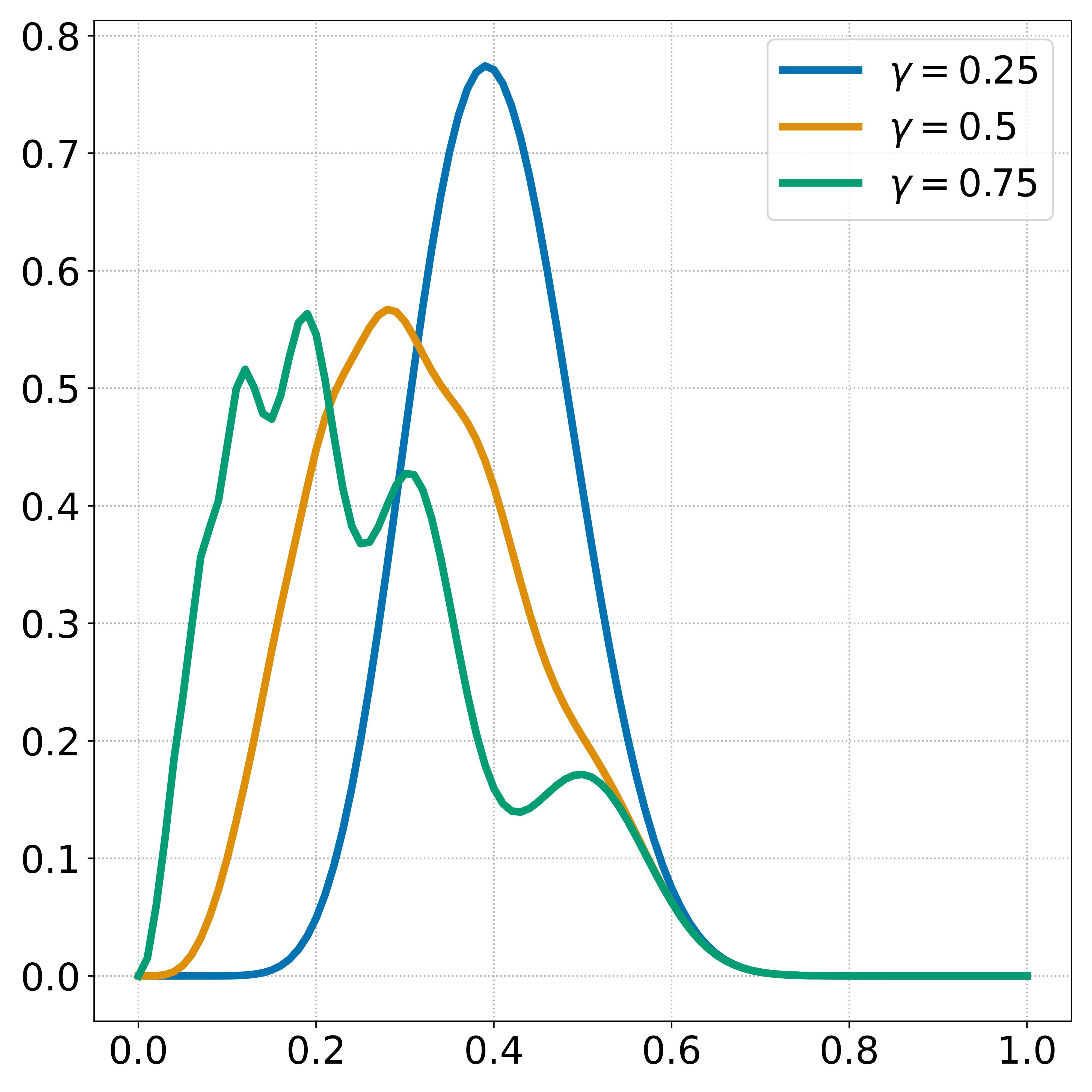}
          \caption{$\phi_L$ at $t = 10$}\label{fig:phiLt10}
          \end{subfigure}
          \begin{subfigure}{0.2\textwidth}
        \includegraphics[width=\textwidth]{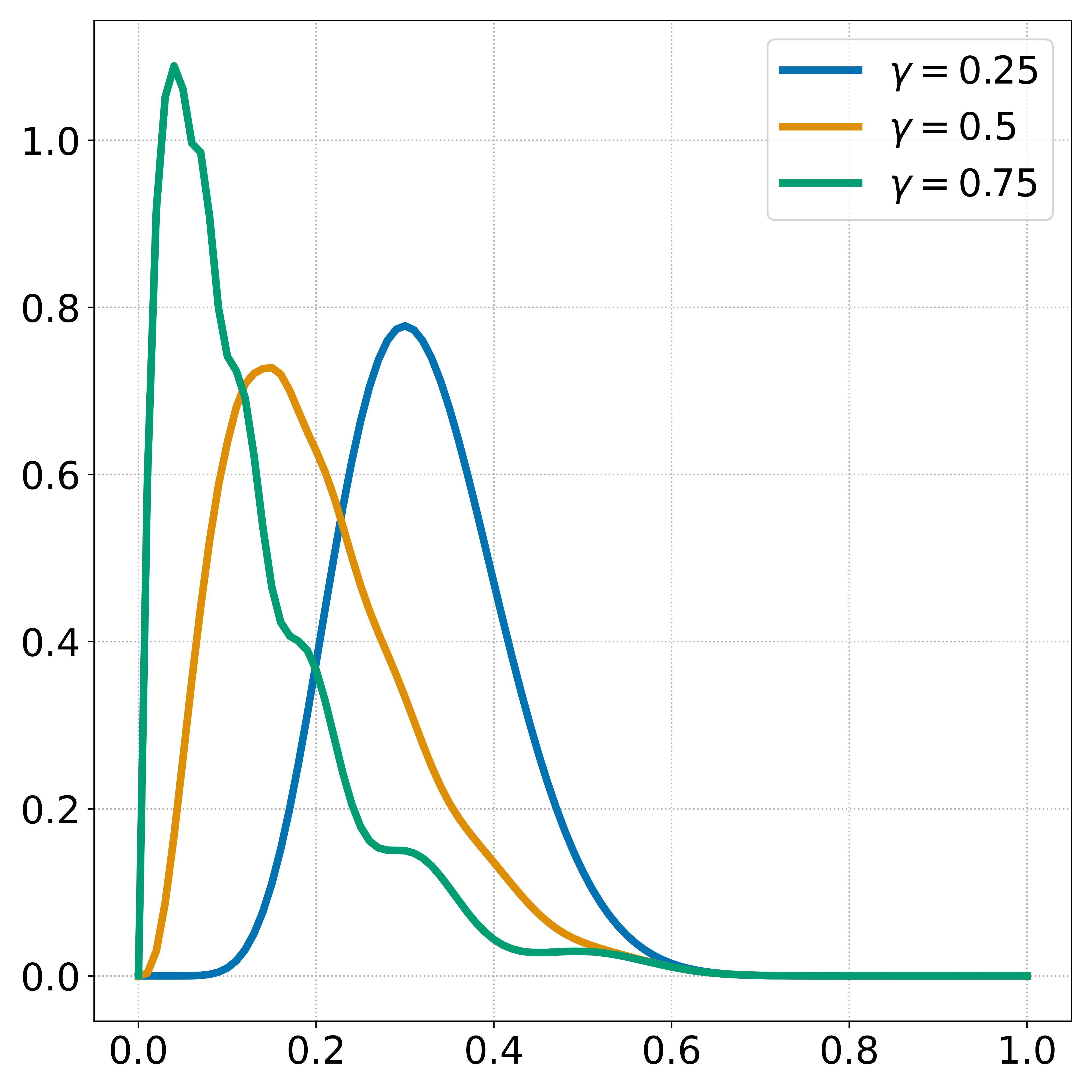}
          \caption{$\phi_L$ at $t = 20$}\label{fig:sphiLt20}
          \end{subfigure}\\
           \begin{subfigure}{0.2\textwidth}
        \includegraphics[width=\textwidth]{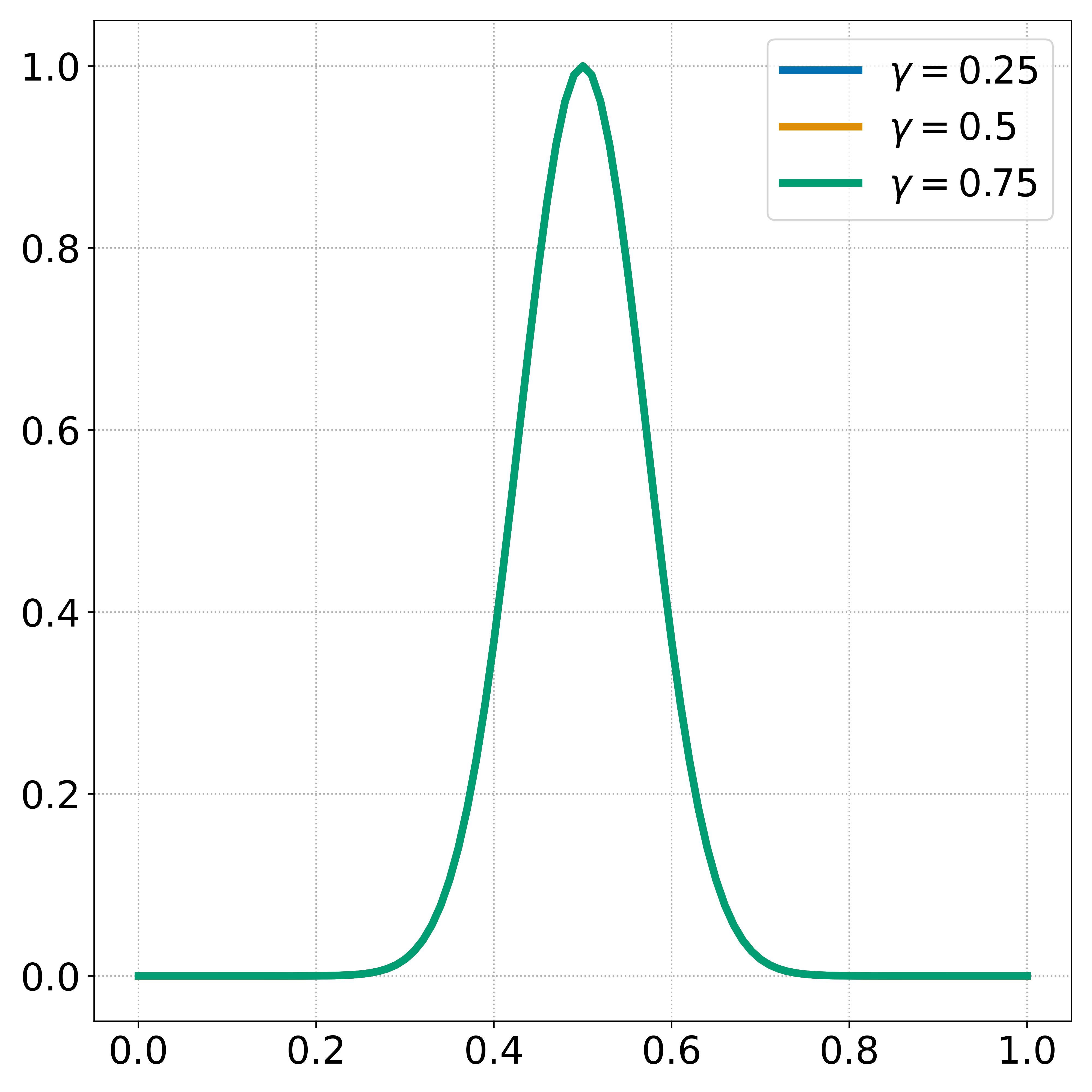}
        \caption{$\phi_R$ at $t = 0$}
        \label{fig:phiRt0}
          \end{subfigure}
        \begin{subfigure}{0.2\textwidth}
        \includegraphics[width=\textwidth]{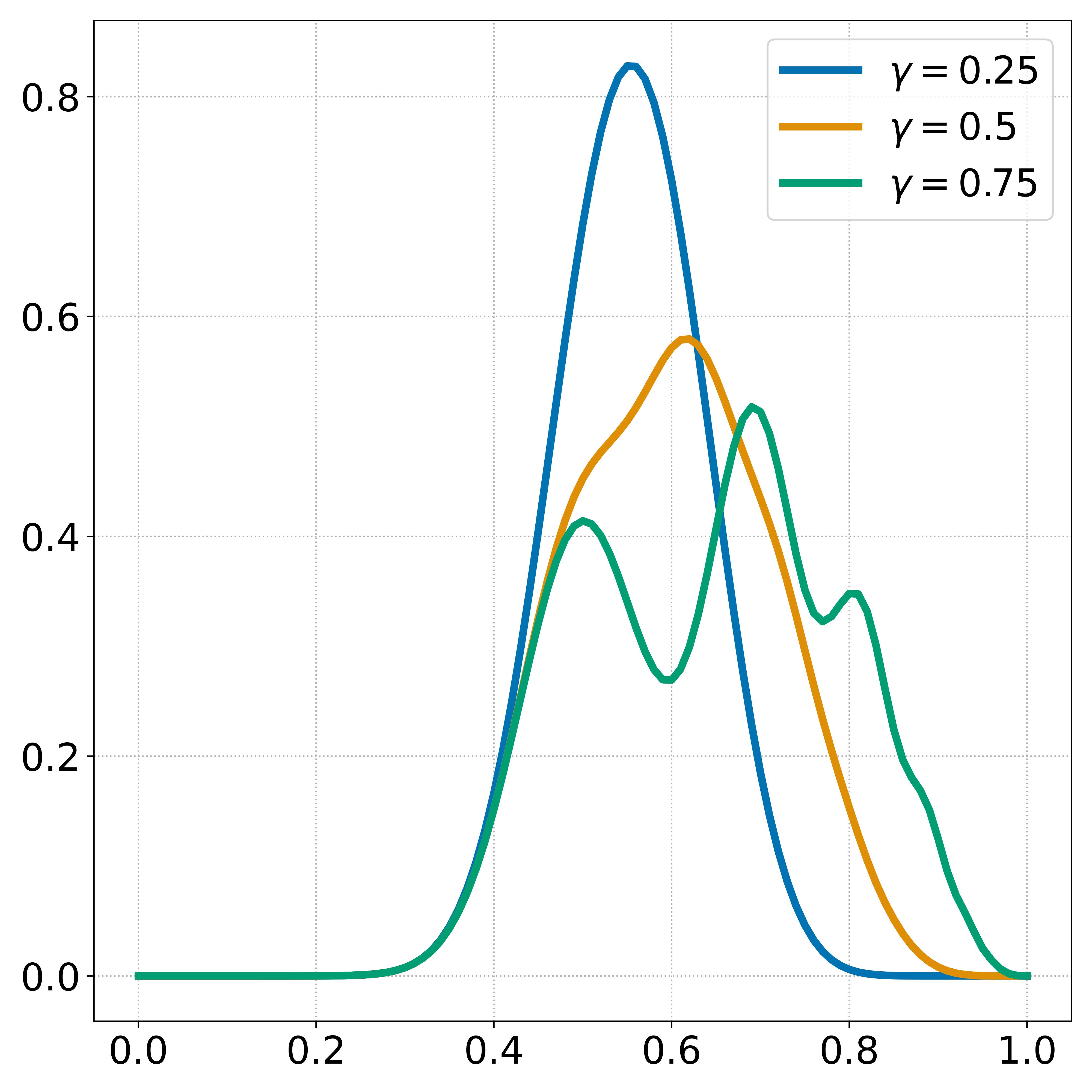}
          \caption{$\phi_R$ at $t = 5$}\label{fig:phiRt5}
          \end{subfigure}
        \begin{subfigure}{0.2\textwidth}
        \includegraphics[width=\textwidth]{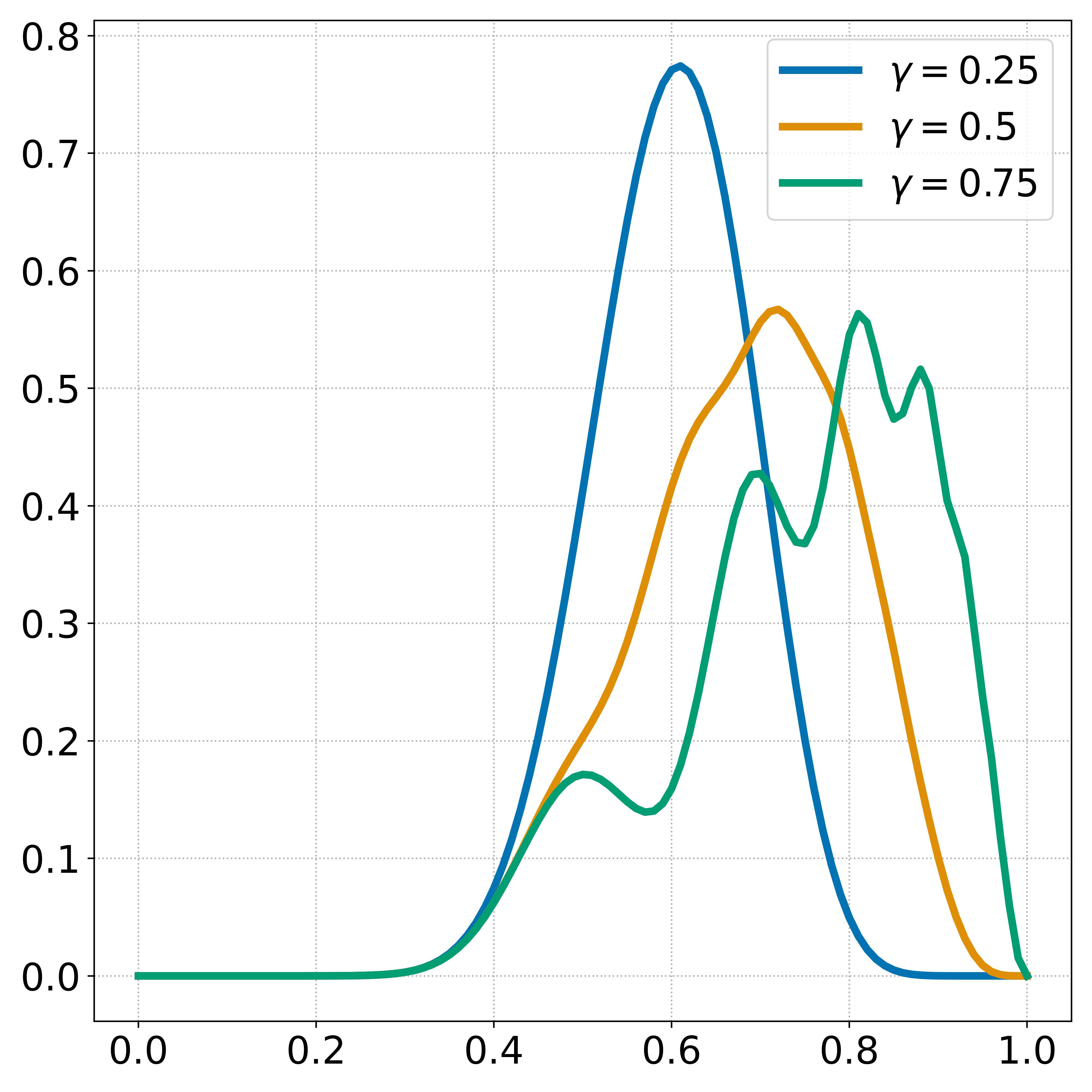}
          \caption{$\phi_R$ at $t = 10$}\label{fig:phiRt10}
          \end{subfigure}
          \begin{subfigure}{0.2\textwidth}
        \includegraphics[width=\textwidth]{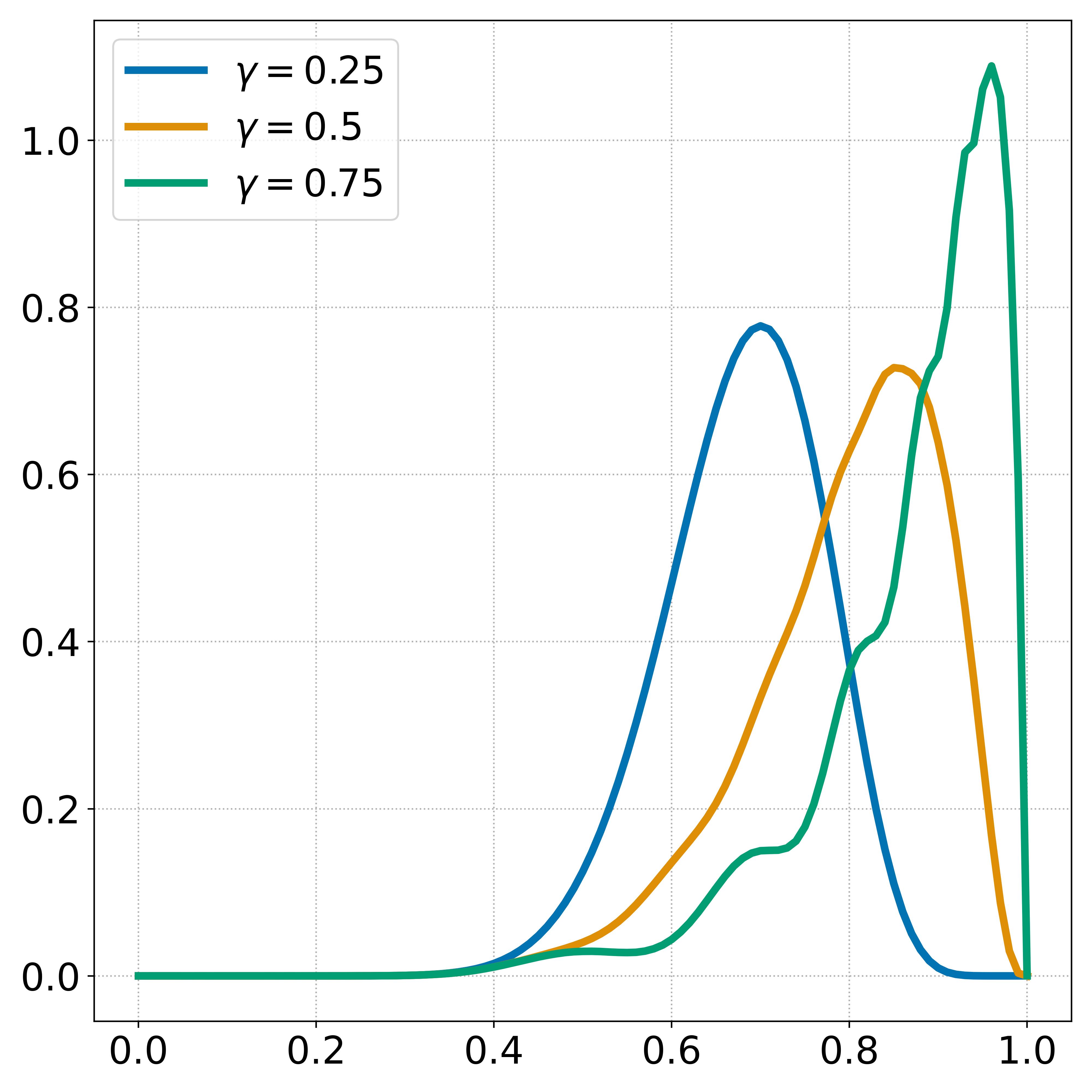}
          \caption{$\phi_R$ at $t = 20$}\label{fig:phiRt20}
          \end{subfigure}\\
           \begin{subfigure}{0.2\textwidth}
        \includegraphics[width=\textwidth]{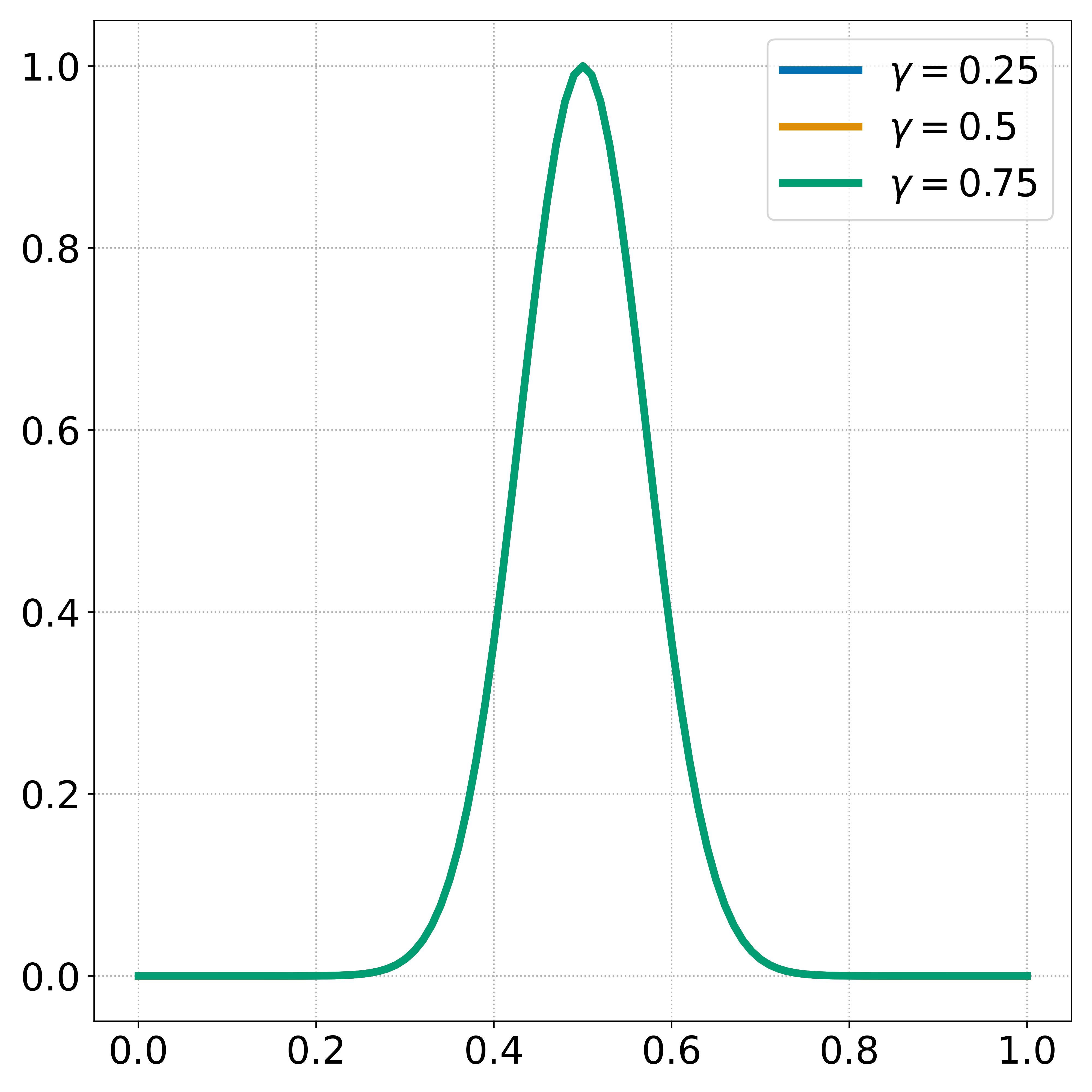}
          \caption{$\phi_A$ at $t = 0$}\label{fig:phiAt0}
          \end{subfigure}
        \begin{subfigure}{0.2\textwidth}
        \includegraphics[width=\textwidth]{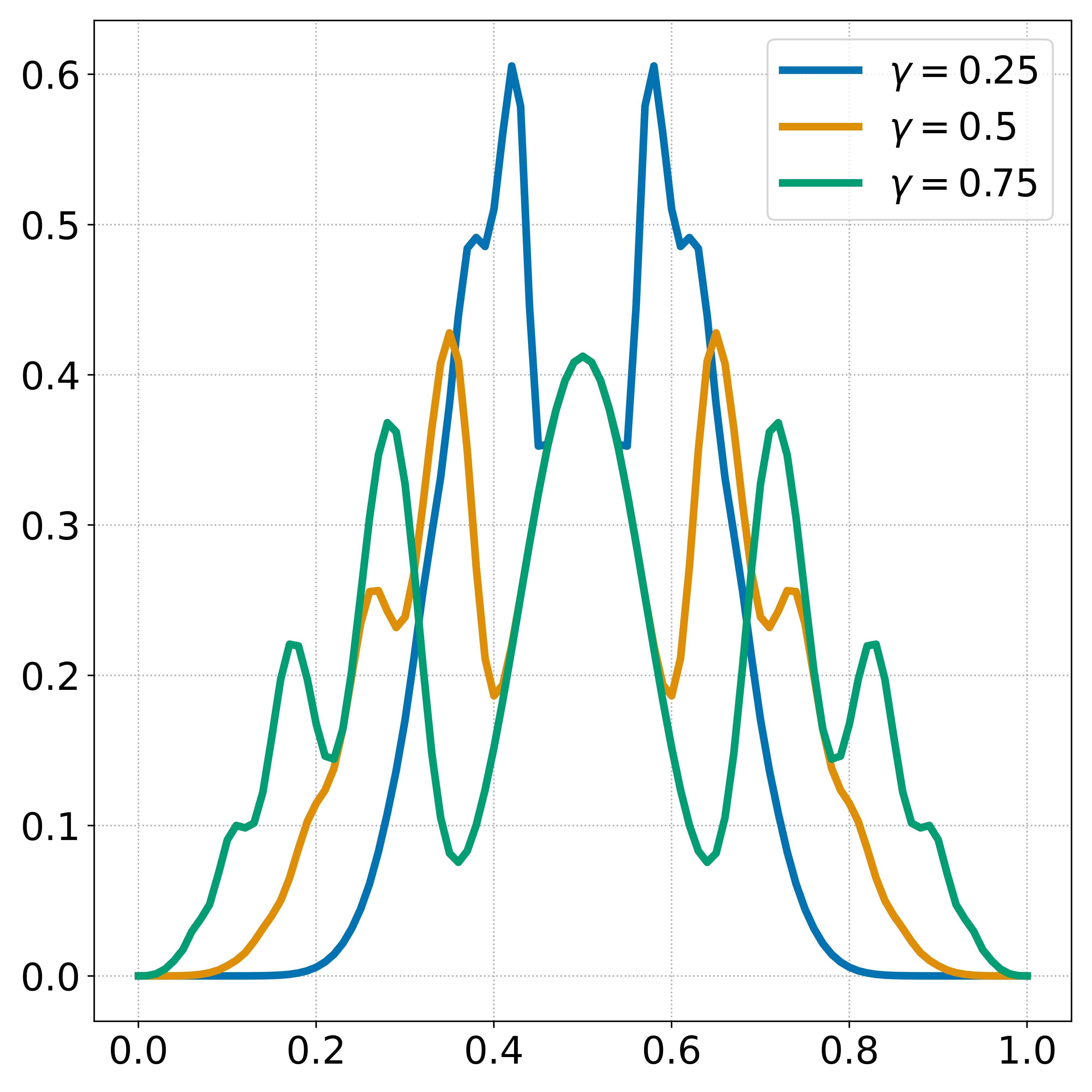}
          \caption{$\phi_A$ at $t = 5$}\label{fig:phiAt5}
          \end{subfigure}
        \begin{subfigure}{0.2\textwidth}
        \includegraphics[width=\textwidth]{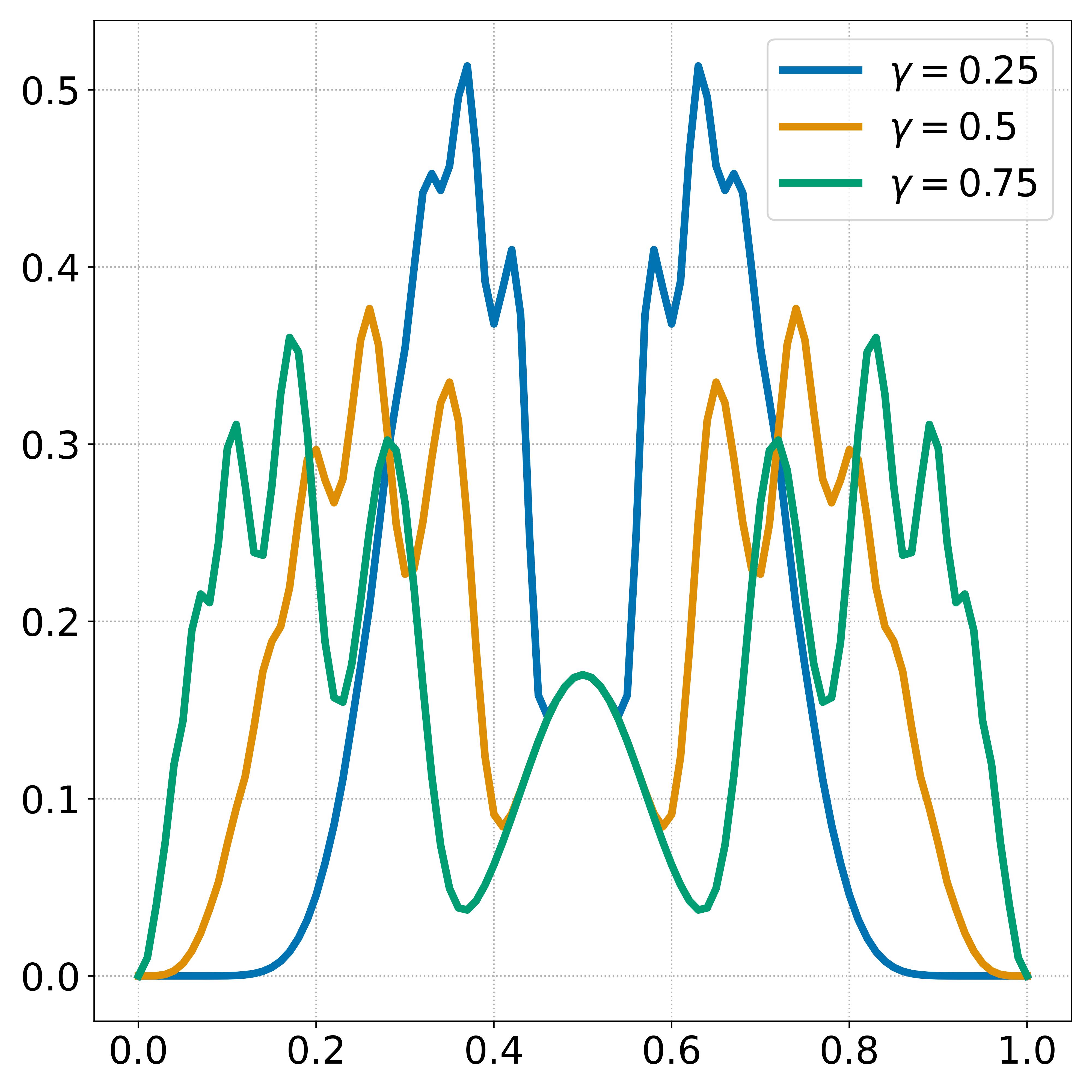}
          \caption{$\phi_A$ at $t = 10$}\label{fig:phiAt10}
          \end{subfigure}
          \begin{subfigure}{0.2\textwidth}
        \includegraphics[width=\textwidth]{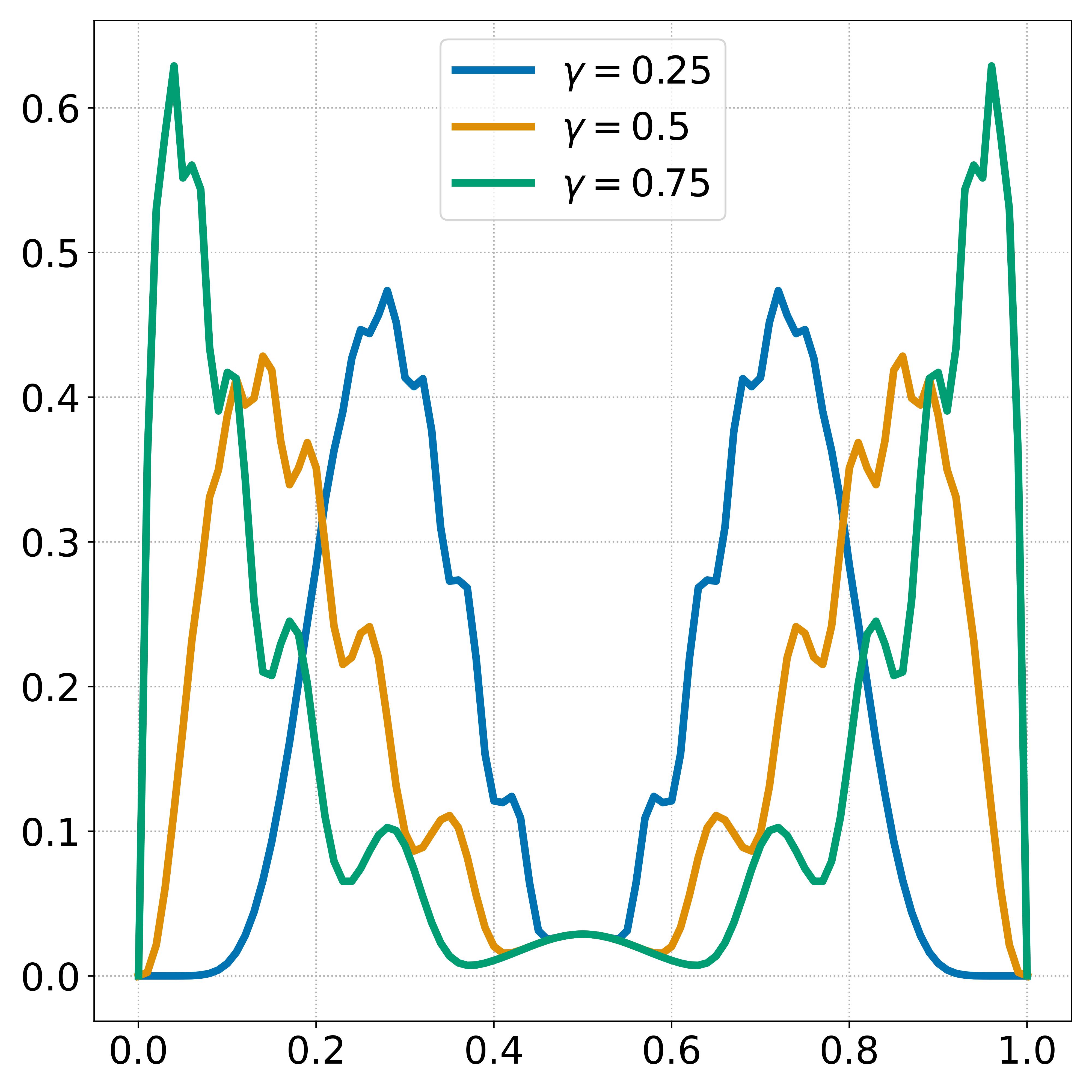}
          \caption{$\phi_A$ at $t = 20$}\label{fig:phiAt20}
          \end{subfigure}\\
           \begin{subfigure}{0.2\textwidth}
        \includegraphics[width=\textwidth]{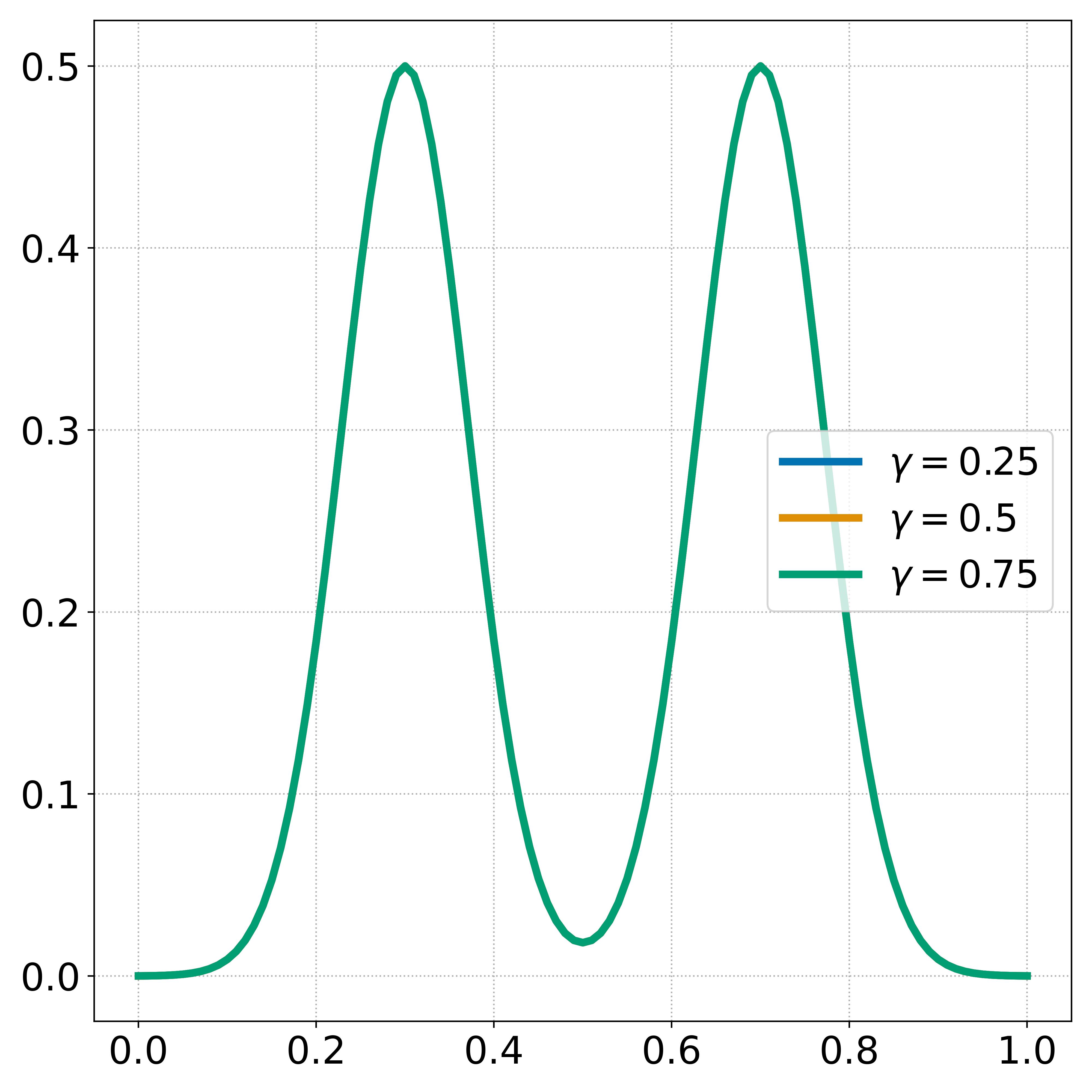}
          \caption{$\phi_T$ at $t = 0$}\label{fig:phiTt0}
          \end{subfigure}
        \begin{subfigure}{0.2\textwidth}
        \includegraphics[width=\textwidth]{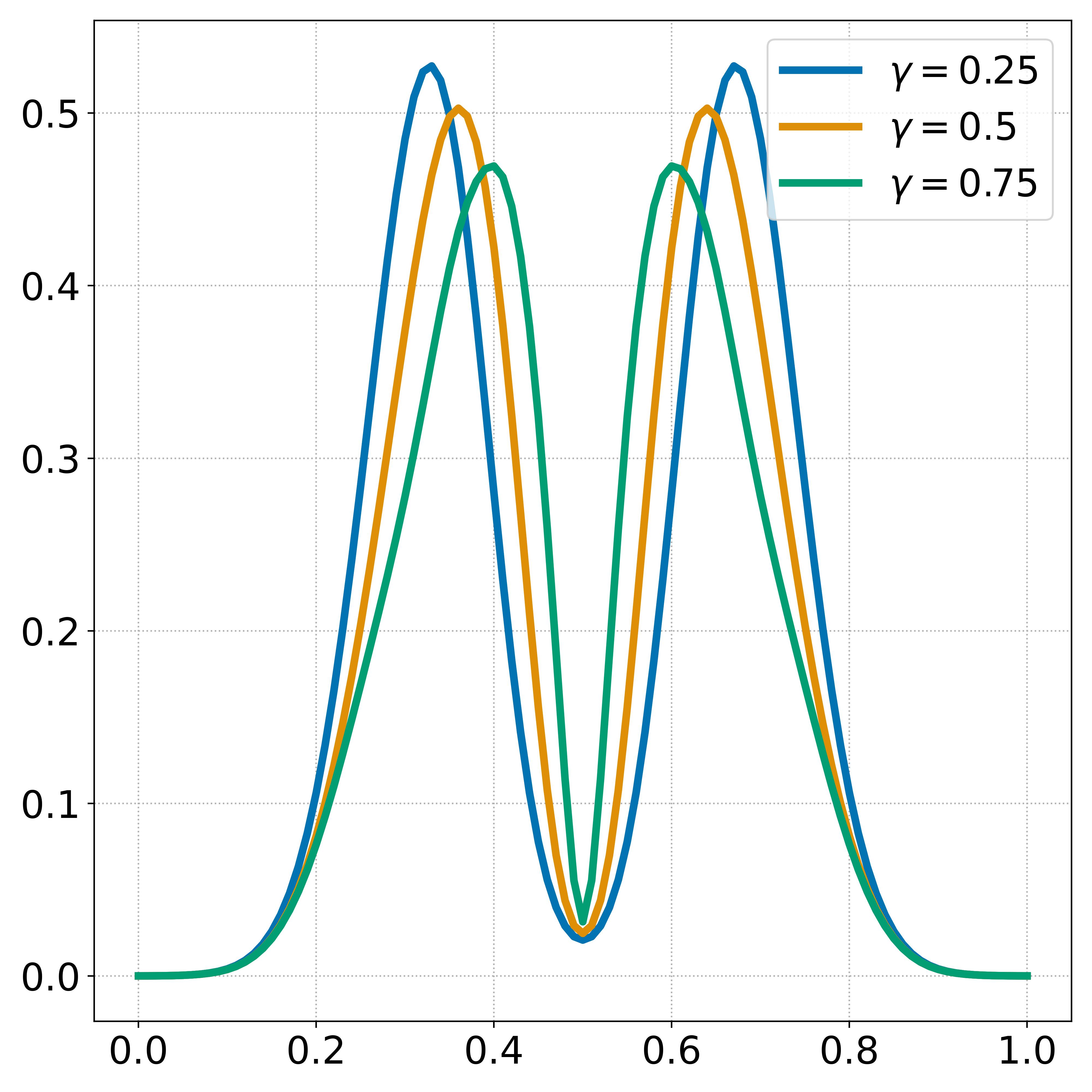}
          \caption{$\phi_T$ at $t = 5$}\label{fig:phiTt5}
          \end{subfigure}
        \begin{subfigure}{0.2\textwidth}
        \includegraphics[width=\textwidth]{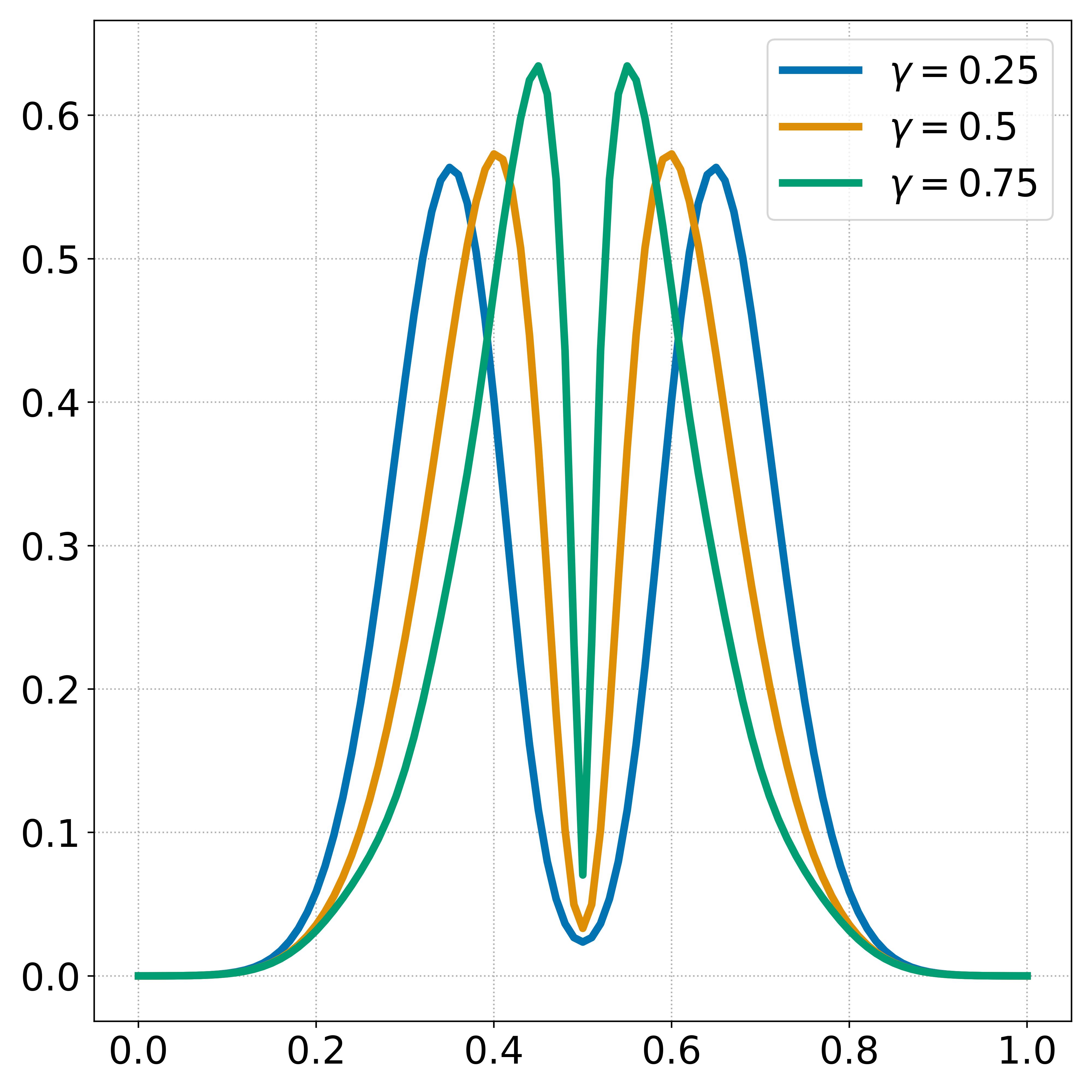}
        \caption{$\phi_T$ at $t = 10$}\label{fig:phiTt10}
          \end{subfigure}
          \begin{subfigure}{0.2\textwidth}
        \includegraphics[width=\textwidth]{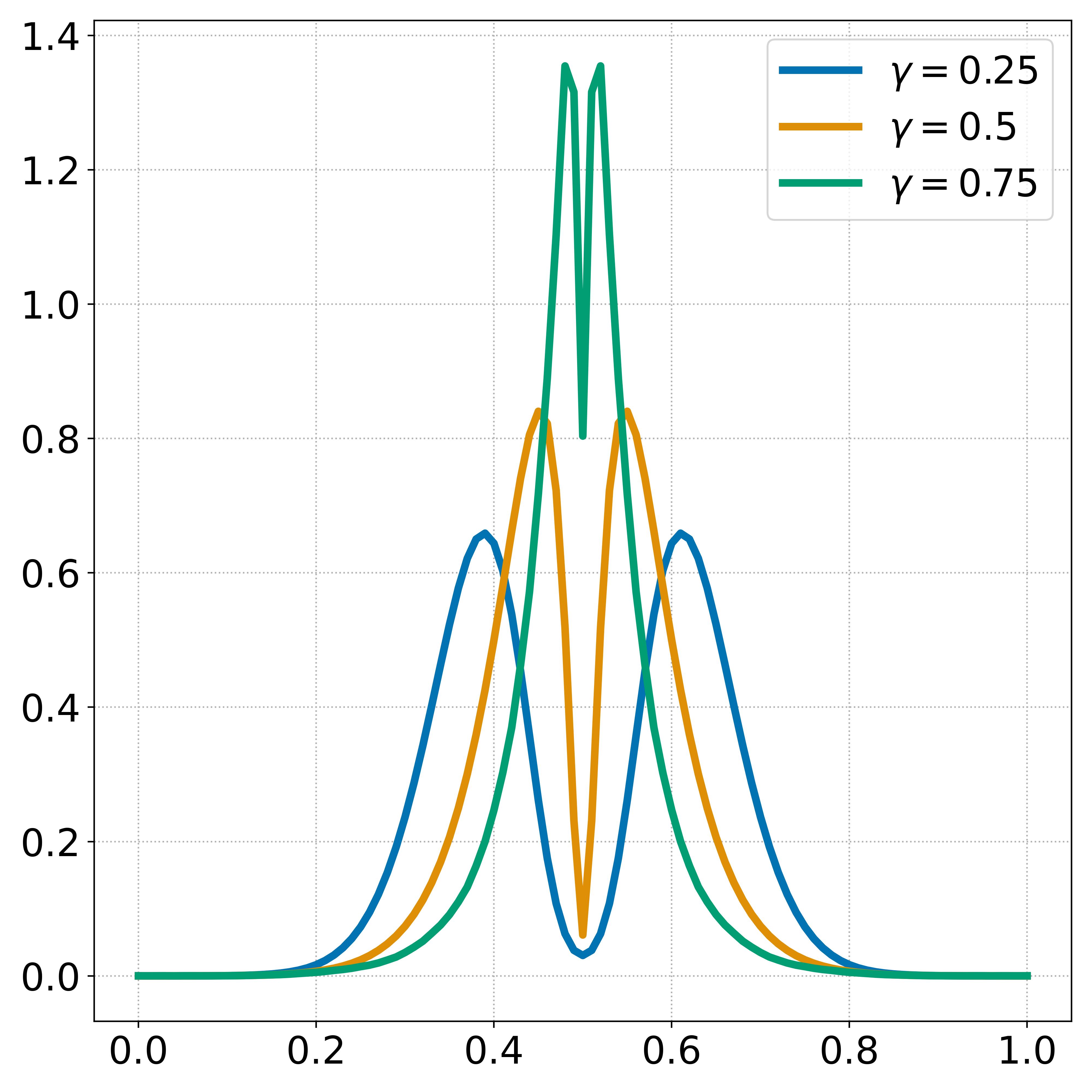}
          \caption{$\phi_T$ at $t = 20$}\label{fig:phiTt20}
          \end{subfigure}\\
\caption{Temporal evolution of the kinetic equation \eqref{EQ:oneparticles} for different transition functions $\phi$ (defined in Table \ref{tab:phitab}), and varying values of $\gamma$. Each row represents the time evolution of the solution for a distinct transition function $\phi$, highlighting differences in agent movement patterns. The initial conditions remain consistent across the simulations for $\phi_L$, $\phi_R$, and $\phi_A$ to allow direct comparison. The simulations are run using MPCM on the time interval $[0,20].$ }
\label{fig:simonepartdiffg}
\end{figure}

\section{Tau-leaping Algorithm}\label{app:taulp}

To simulate system \eqref{eq:integrodiff} with Tau-leaping, we use the following steps:

\begin{enumerate}
\item Divide the microstate into $N_b$ bins $\{u_1, ..., u_{N_b}\}$, each with width $\Delta u$.
\item Represent $f_i(t,u)$ as a collection of particle counts $F_i^l(t)$ in each bin $l$.
\item Precompute all target bins $\phi_{hk}(l,m)$ such that an interaction between bins $l$ and $m$ for populations $h$ and $k$ produces a particle in bin $\phi_{hk}(l,m)$ of population $i$.
\item Use Poisson-distributed random variables to sample the number of reactions occurring in each time step $\tau$.
\end{enumerate}

\begin{algorithm}[H]
\caption{Tau-leaping Simulation for General Kinetic Equation \eqref{eq:integrodiff}}
\begin{algorithmic}[1]
\State Initialize particle counts $F_i^l(0)$ from the initial condition $f_i(0,u_l)$

\State Precompute $\psi_{hk}(l,m) =$ bin index corresponding to transition outcome from $T_{hk}^i(u_l, u_m, ·)$

\For{each time step $t \rightarrow t + \tau$}
\For{each pair of populations $(h,k)$ and bins $(l,m)$}
\State Compute reaction rate:
$     a_{lm}^{hk} = F_h^l \cdot f_k^m \cdot \eta_{hk}(u_l, u_m) \cdot \Delta u
    $
\State Sample number of reactions:
$     N_{lm}^{hk} \sim \text{Poisson}(a_{lm}^{hk} \cdot \tau)
    $
\State Update:
\begin{itemize}
\item $F_h^l ← F_h^l - N_{lm}^{hk}$
\item $F_i^{\phi_{hk}(l,m)} ← F_i^{\phi_{hk}(l,m)} + N_{lm}^{hk}$
\end{itemize}
\EndFor
\EndFor
\end{algorithmic}
\end{algorithm}

The loss term is implicitly captured by the decrease in $F_h^l$ when particles of type $h$ in bin $l$ are removed due to interaction events. The gain term is handled by the redistribution of particles into the destination bin $\phi_{hk}(l,m)$.

The selection of the time step $\tau$ is crucial in the Tau-leaping method to strike a balance between accuracy and computational efficiency. A small $\tau$ ensures that reaction propensities remain approximately constant over the time interval $[t, t+\tau]$, which is a core assumption of the method \cite{gillespie2001approximate}. However, excessively small values reduce the efficiency advantage over the exact SSA. Adaptive schemes that estimate the maximal allowable $\tau$ based on local changes in propensity functions have been proposed \cite{cao2006efficient, anderson2012multilevel}. Still, in the context of discretized kinetic equations where each bin contains a large number of particles and interactions are smooth in time, a fixed $\tau$ may suffice. In this work, we fix $\tau = 0.001$, which is small enough to resolve the timescale of particle interactions and preserve stability, while being large enough to produce significant leaps in system evolution without oversampling individual reaction events. Experimentally, choosing $\tau = 0.001$ yielded the most accurate and stable results for the system under consideration. 

\section{Hybrid Algorithm}\label{app:hybrid}
To improve the efficiency and scalability of Tau-leaping for nonlinear interaction-driven systems, we use a hybrid stochastic–deterministic framework. This method dynamically partitions the computational domain into regions where stochastic updates (via Tau-leaping) are appropriate and regions where deterministic integration suffices. The hybridization significantly accelerates simulations, particularly in high-density regimes where stochastic fluctuations become negligible. This strategy is inspired by the multiscale approach introduced by Yates et al.\cite{yates2020blending}, which balances continuous and discrete updates in biochemical kinetics.

We consider the general integro-differential system \eqref{eq:integrodiff}. Direct simulation via tau-leaping requires sampling interactions and redistributing outcomes in each bin, which can be computationally intensive in dense regions. To mitigate this, we define a hybrid scheme:

\begin{itemize}
\item If $f_h(x)/\Delta u > \theta$ and $f_k(y)/\Delta u > \theta$, we use a deterministic approximation for the number of interactions.
\item If the densities are low, we retain stochasticity via Poisson sampling of reaction events.
\end{itemize}

Each outcome of the interaction $u = \phi_{hk}(x,y)$ is discretized using linear interpolation across the spatial grid. To ensure mass conservation and avoid artifacts introduced by naive binning, once the density of a bin is low, fewer than $1000$ particles, we revert to the Tau-leaping algorithm for those interactions. The choice of the cutoff between deterministic and the stochastic regimes remains an open problem in hybrid simulation methods. It also depends heavily on the equations at hand. For our set of kinetic systems, we found experimentally that using a cutoff of $1000$ particles provided the best balance, yielding mass preservation errors below machine precision.

\vspace{1em}
\noindent The full hybrid algorithm is outlined below.
\begin{algorithm}[H]
\caption{Hybrid Simulation for system \eqref{eq:integrodiff}}
\begin{algorithmic}[1]
\State Discretize each domain $D_i$ into bins $\{u_1, \dots, u_n\}$ with width $\Delta u$
\State Initialize $f_i(t=0, u)$ on the grid for all $i$
\State Set time step $\tau$ and threshold $\theta > 0$
\For{each time step $t = 0$ to $T$ with step size $\tau$}
    \For{each type $i \in \{1,\dots,N\}$}
        \State Initialize $\Delta f_i(u) \gets 0$
    \EndFor
    \For{each pair $(h, k)$}
        \For{each grid point $x \in D_h$, $y \in D_k$}
            \State Compute $\lambda \gets \eta_{hk}(x, y) f_h(x) f_k(y) \Delta u^2$
            \If{ $f_h(x)/\Delta u > \theta$ and $f_k(y)/\Delta u > \theta$}
                \For{each target type $i$, for all $u \in D_i$}
                    \State $\Delta f_i(u) \mathrel{+}= \lambda \cdot T_{hk}^i(x,y,u) \cdot \tau$
                \EndFor
            \Else
                \State Sample $\Delta N \sim \mathrm{Poisson}(\lambda \cdot \tau)$
                \For{each target type $i$, for all $u \in D_i$}
                    \State $\Delta f_i(u) \mathrel{+}= \Delta N \cdot T_{hk}^i(x,y,u)$
                \EndFor
            \EndIf
        \EndFor
    \EndFor
    \For{each $i$ and $u \in D_i$}
        \State Compute loss: $L_i(u) \gets f_i(u) \cdot \sum_k \sum_{x \in D_k} \eta_{ik}(u,x) f_k(x) \cdot \Delta u$
        \State Update: $f_i(u) \gets f_i(u) + \Delta f_i(u) - L_i(u) \cdot \tau$
    \EndFor
\EndFor
\end{algorithmic}
\end{algorithm}


This algorithm ensures positivity and mass conservation (up to interpolation error) and automatically transitions between stochastic and deterministic dynamics depending on local density. It is particularly useful for systems exhibiting localized concentrations or multiscale dynamics.

\section{Evaluating the coefficients $T_{ilmj}^{pq}$} \label{app:evalcoeff}
Before applying quadrature to evaluate the terms  $T_{ilmj}^{pq}$,  we can first analytically transform the integrals into path integrals without delta functions stemming from the transition rates: $T(x,y,u) = \delta_{\phi(x,y) - u}$. Given that 
$$
T_{ilmj}^{pq} = \int_{D_l} \int_{D_m} \eta_{lm}(x,y)T_{lm}^i(x,y,x_j)B_p(x)B_q(y)dxdy,
$$

we project this integral into a subspace of $D_l \times D_m$ by using the theorem from \cite{hormander2007analysis}. We end up with
\begin{align*}
    T_{ilmj}^{pq} &= \int_{D_l} \int_{D_m} \eta_{lm}(x,y)T_{lm}^i(x,y,x_j)B_p(x)B_q(y)dxdy\\
    &= \int_{D_l} \int_{D_m} \eta_{lm}(x,y)\delta_{\phi(x,y) - u}B_p(x)B_q(y)dxdy\\
    &= \int_{\phi(x,y) - x_j = 0} \frac{1}{|\nabla \phi(x,y)|}B_p(x)B_q(y)dxdy.
\end{align*}
Provided $\phi$ does not have a singularity in $D_l \times D_m$, an adaptive quadrature can now be used to evaluate the coefficients.

\textbf{Example: $\phi(x,y) = x - \gamma xy$}

For this peculiar case, the quantity $T_{ilmj}^{pq}$ will then become:
\begin{align*}
     T_{ilmj}^{pq} &= \int_{D_l} \int_{D_m} \eta_{lm}(x,y)\delta_{x - \gamma xy - x_j}B_p(x)B_q(y)dxdy\\
     &= \int_{x - \gamma xy - x_j = 0}  \eta_{lm}(x,y)\frac{1}{|\nabla (x- \gamma xy - u )|}B_p(x)B_q(y)dxdy\\
     &= \int_{x - \gamma xy - x_j = 0} \eta_{lm}(x,y) \frac{1}{\sqrt{(x-\gamma)^2 + \gamma^2}}B_p(x)B_q(y)dxdy\\
     &= \int_{0}^{min(1, \frac{1-x_j}{\gamma})} \eta_{lm}(\frac{x_j}{1-\gamma y}, y) B_p(\frac{x_j}{1-\gamma y})B_q(y)dy
\end{align*}

\section{Simulation of the System Presented in \ref{fig:sys5}}

\begin{figure}[H]
    \centering
    \begin{subfigure}{.21\textwidth}
        \includegraphics[width=\textwidth]{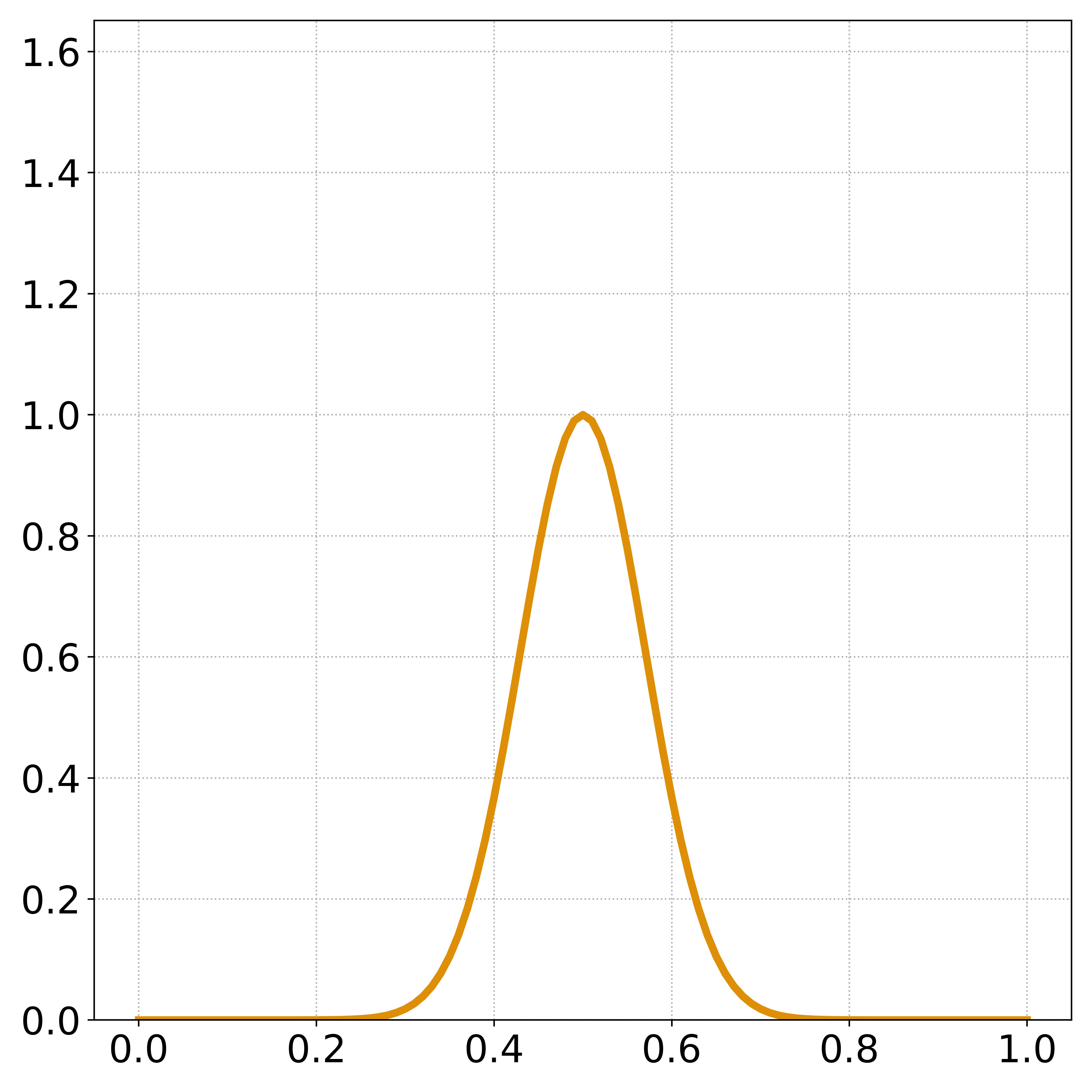}
          \caption{$f_1$ at $t = 0$} \label{fig:sim4f1t0}
          \end{subfigure}
        \begin{subfigure}{.21\textwidth}
        \includegraphics[width=\textwidth]{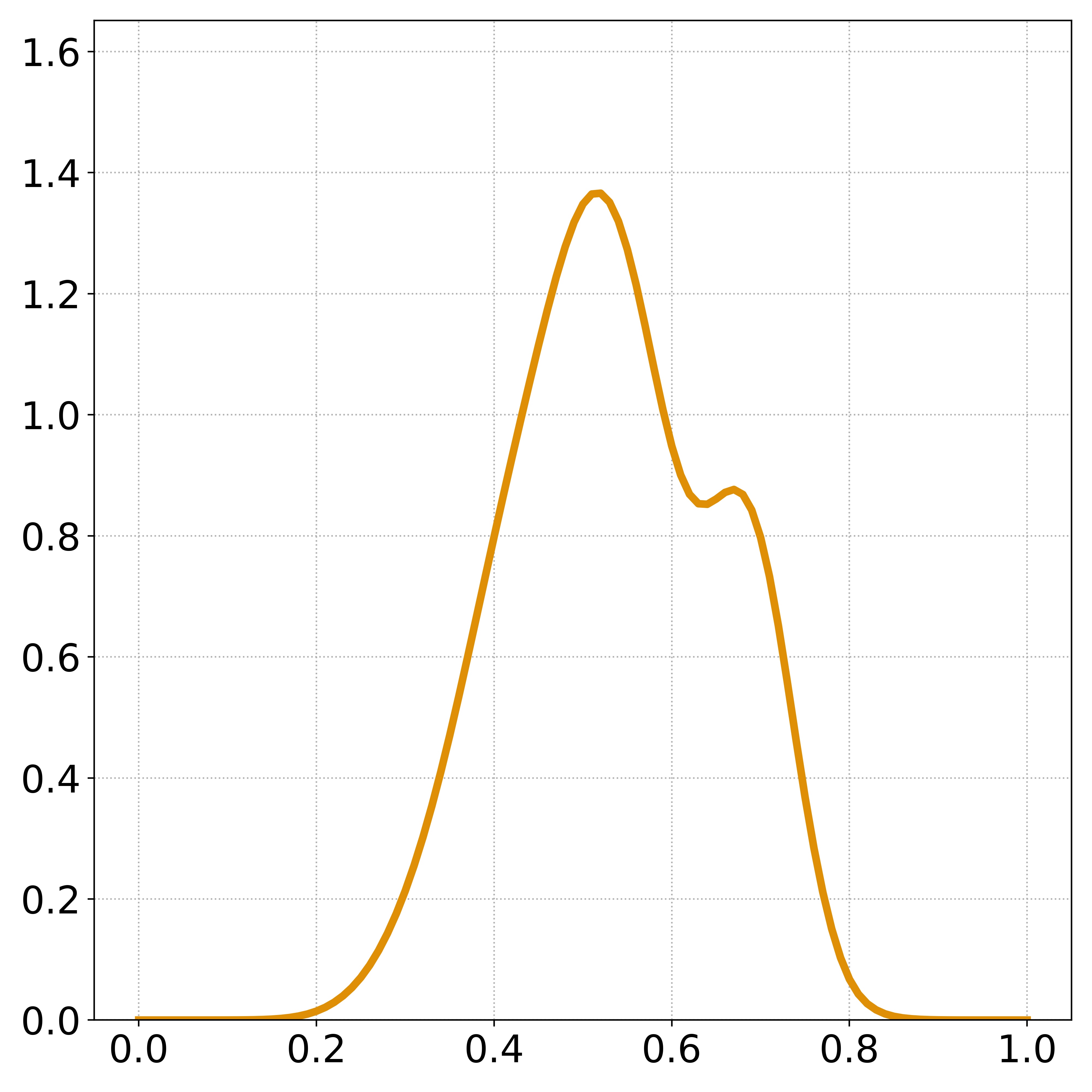}
          \caption{$f_1$ at $t = 5$}\label{fig:sim4f1t5}
          \end{subfigure}
        \begin{subfigure}{.21\textwidth}
        \includegraphics[width=\textwidth]{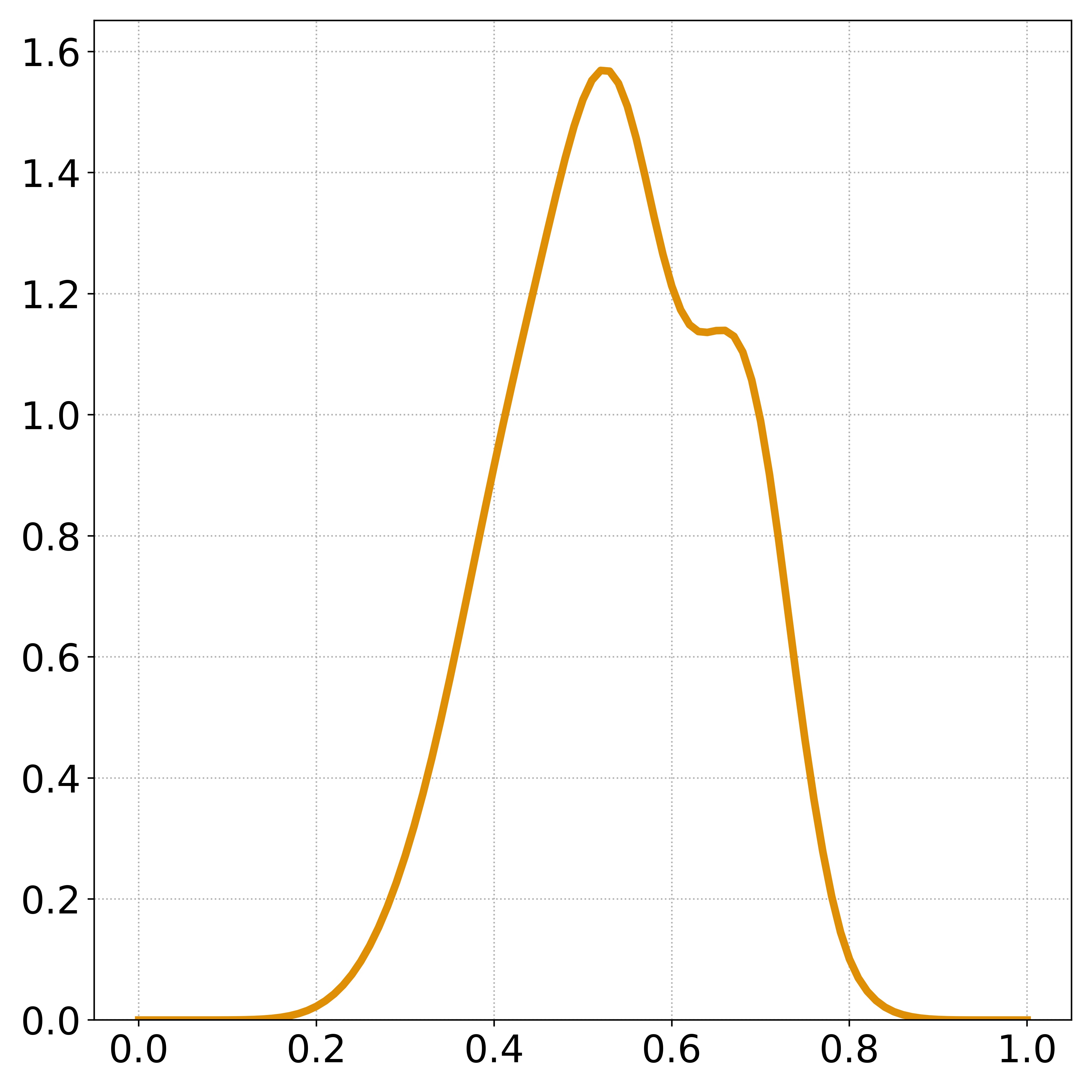}
          \caption{$f_1$ at $t = 10$}\label{fig:sim4f1t10}
          \end{subfigure}
          \begin{subfigure}{.21\textwidth}
        \includegraphics[width=\textwidth]{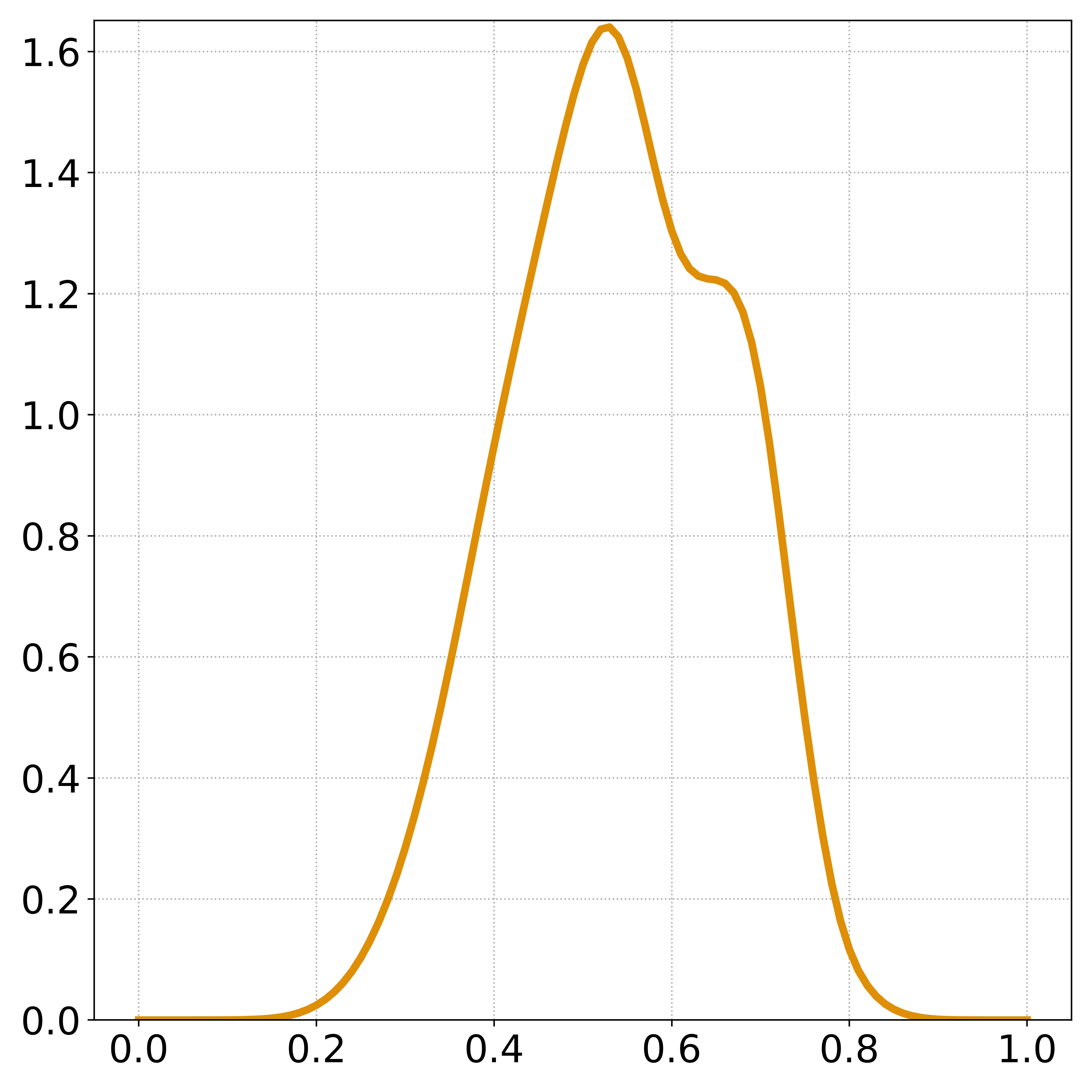}
          \caption{$f_1$ at $t = 20$}\label{fig:sim4f1t20}
          \end{subfigure}\\
           \begin{subfigure}{.21\textwidth}
        \includegraphics[width=\textwidth]{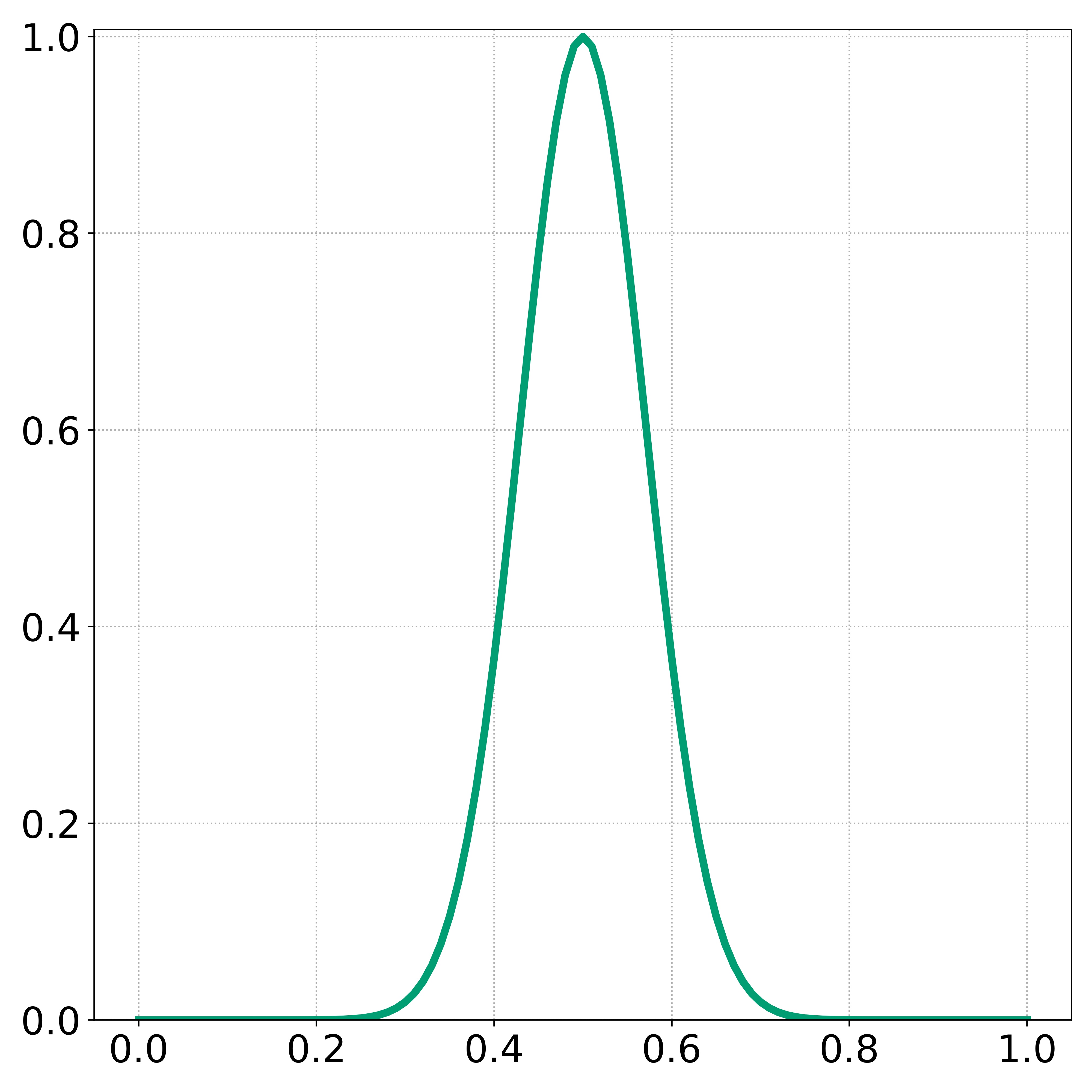}
        \caption{$f_2$ at $t = 0$}
        \label{fig:sim4f2t0}
          \end{subfigure}
        \begin{subfigure}{.21\textwidth}
        \includegraphics[width=\textwidth]{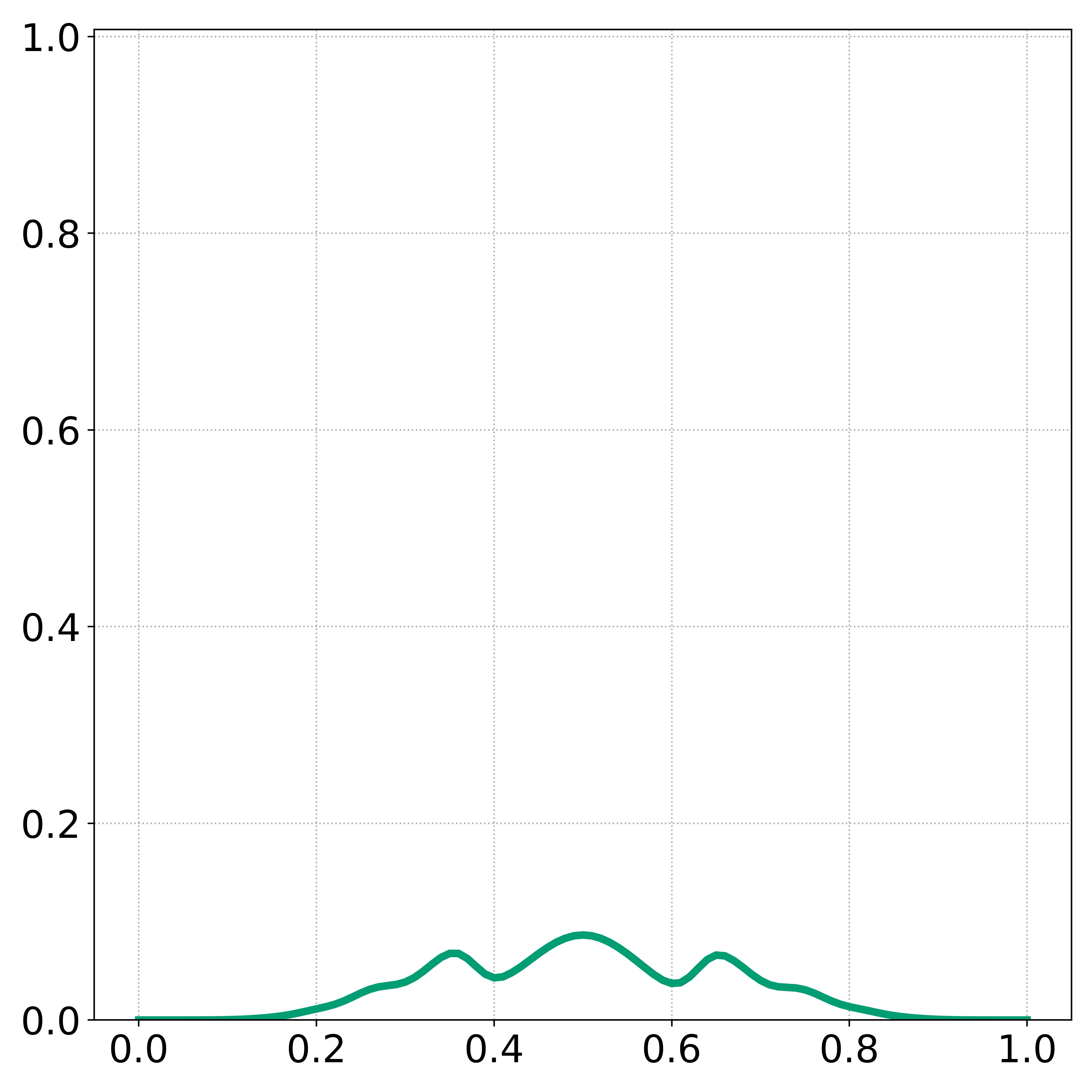}
          \caption{$f_2$ at $t = 5$}\label{fig:sim4f2t5}
          \end{subfigure}
        \begin{subfigure}{.21\textwidth}
        \includegraphics[width=\textwidth]{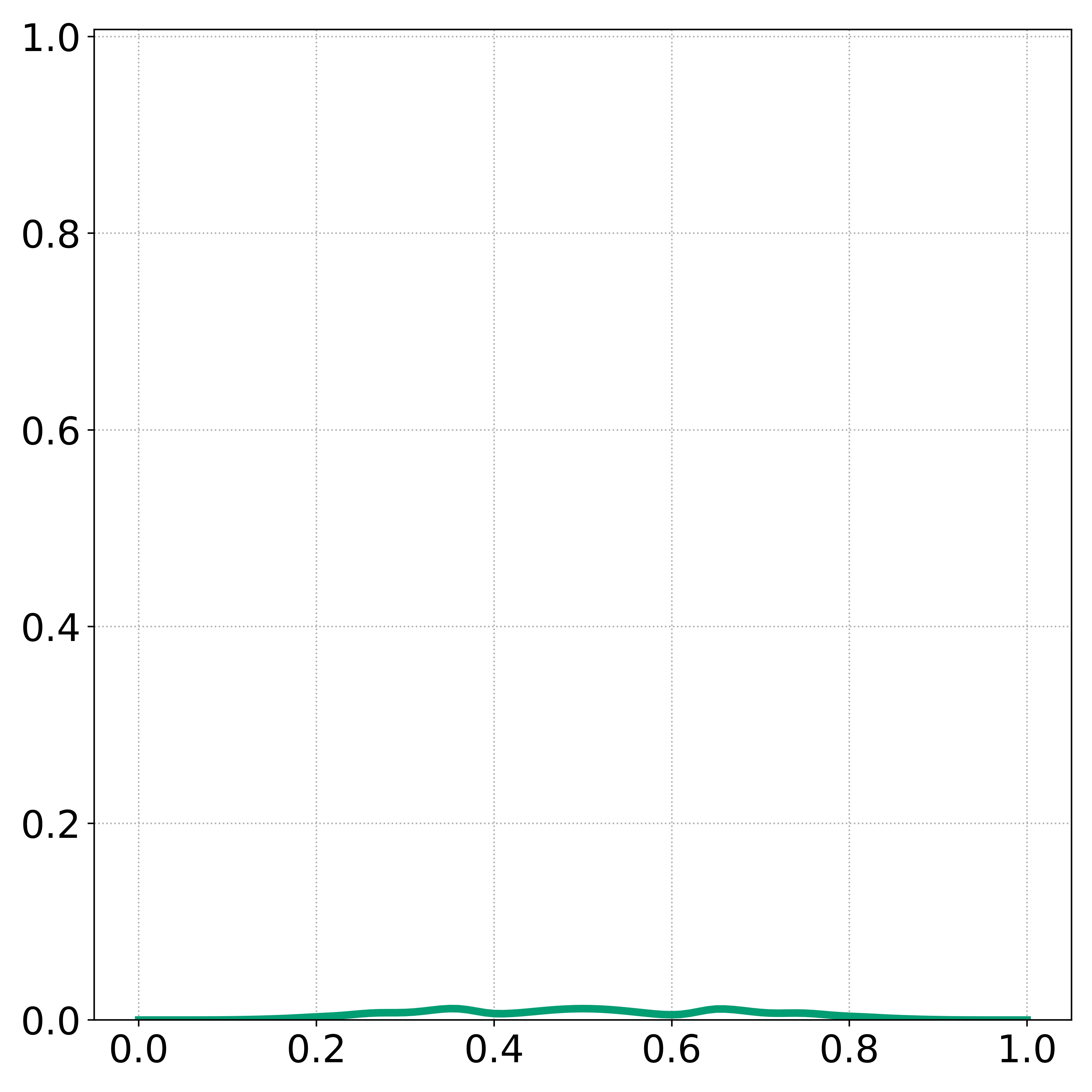}
          \caption{$f_2$ at $t = 10$}\label{fig:sim4f2t10}
          \end{subfigure}
          \begin{subfigure}{.21\textwidth}
        \includegraphics[width=\textwidth]{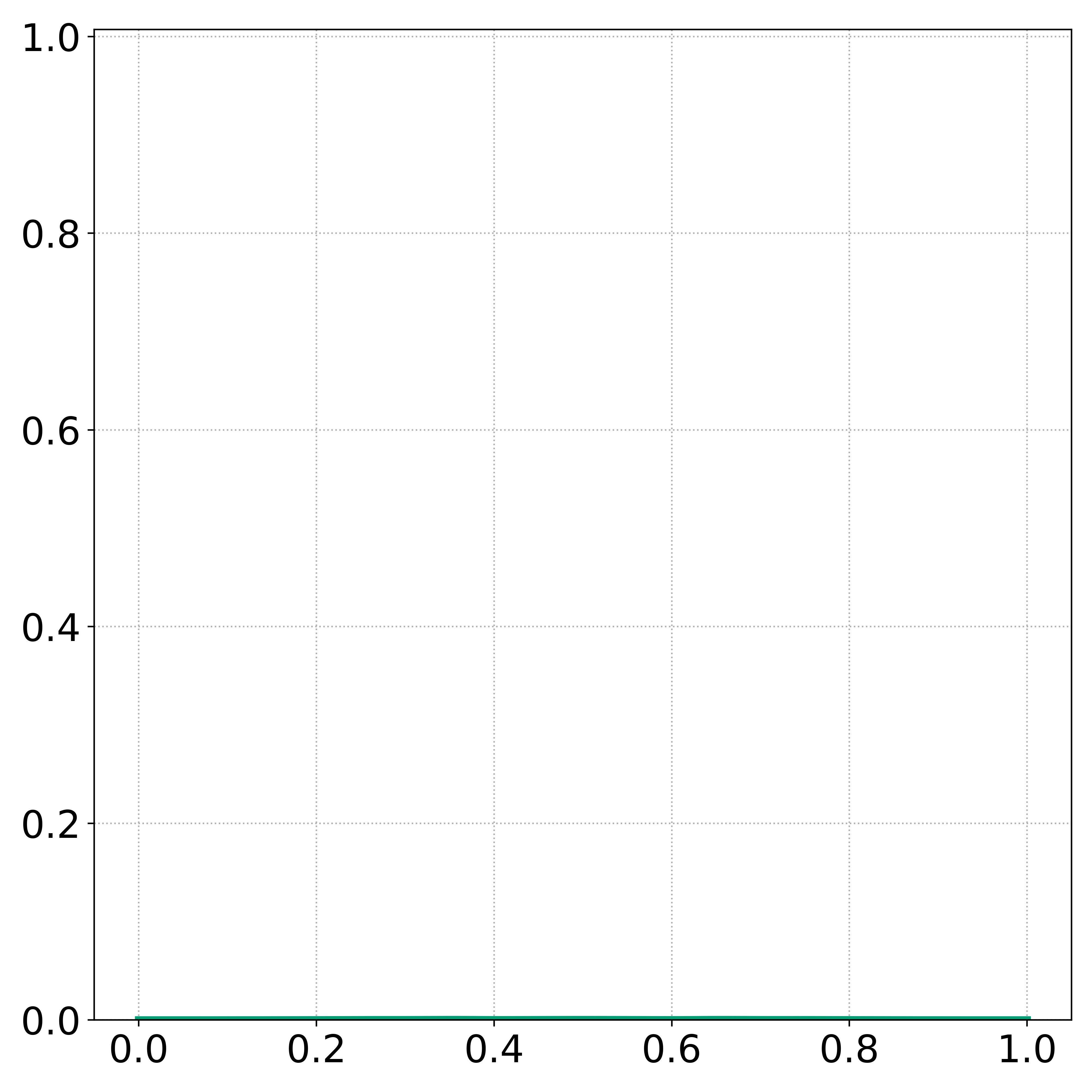}
          \caption{$f_2$ at $t = 20$}\label{fig:sim4f2t20}
          \end{subfigure}\\
           \begin{subfigure}{.21\textwidth}
        \includegraphics[width=\textwidth]{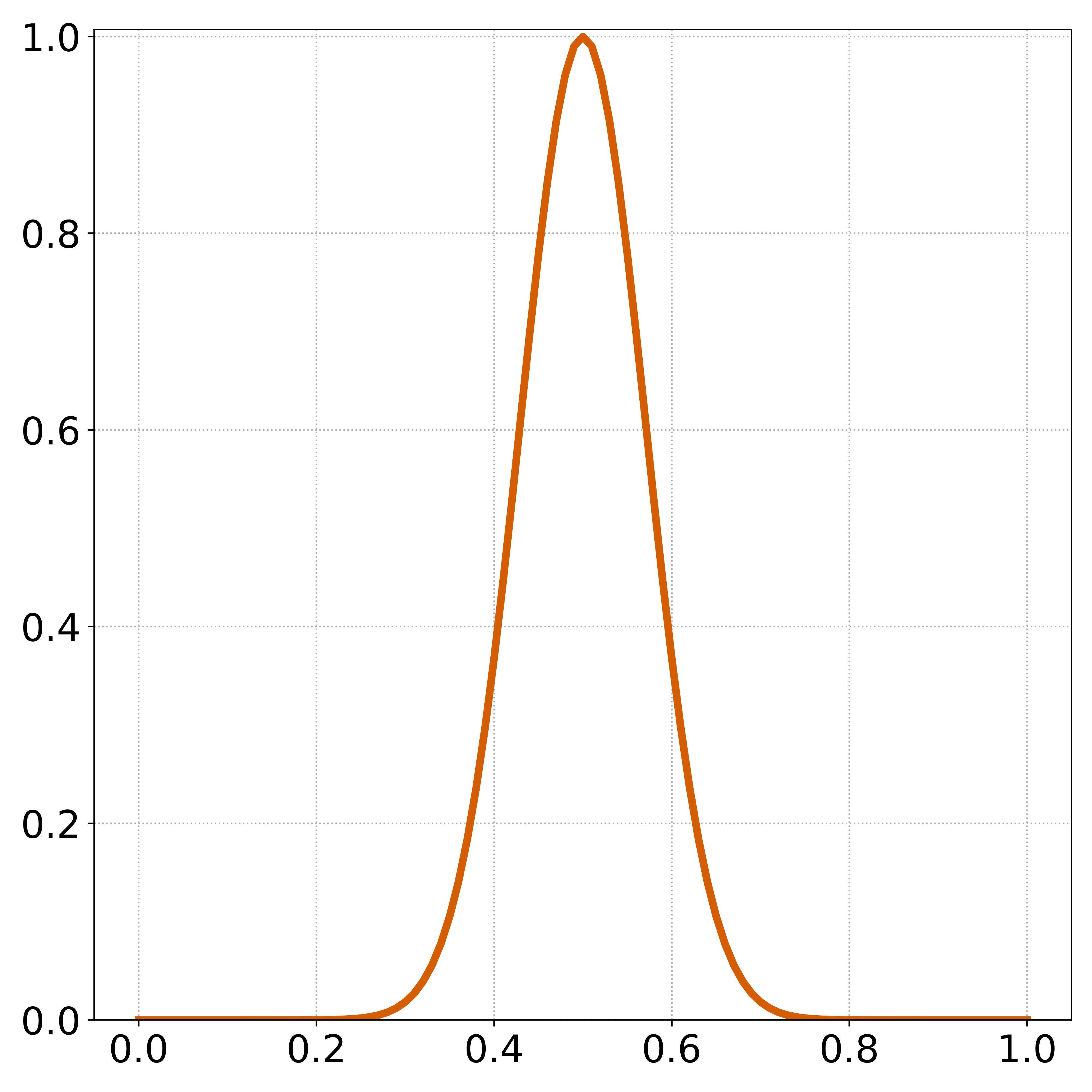}
          \caption{$f_3$ at $t = 0$}\label{fig:sim4f3t0}
          \end{subfigure}
        \begin{subfigure}{.21\textwidth}
        \includegraphics[width=\textwidth]{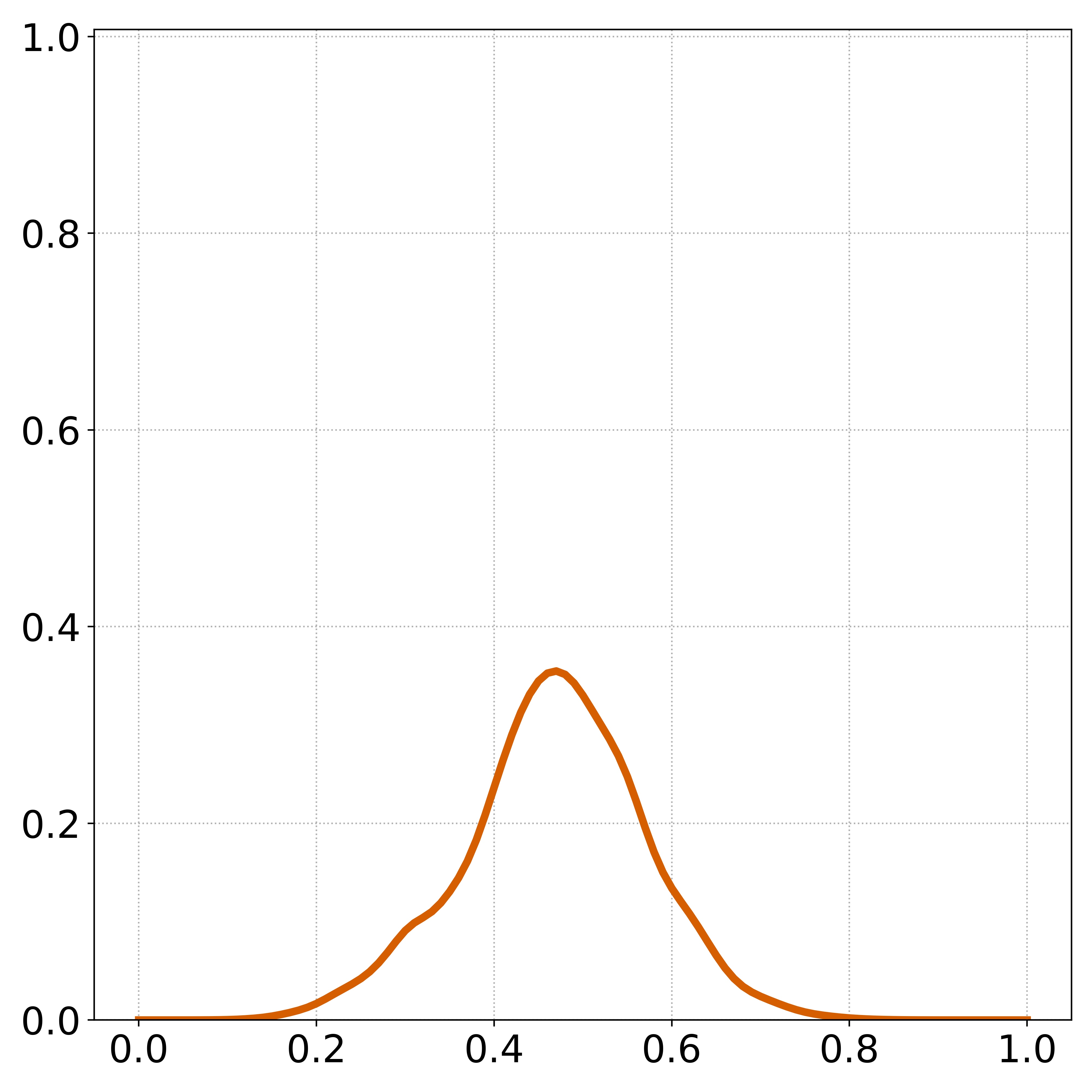}
          \caption{$f_3$ at $t = 5$}\label{fig:sim4f3t5}
          \end{subfigure}
        \begin{subfigure}{.21\textwidth}
        \includegraphics[width=\textwidth]{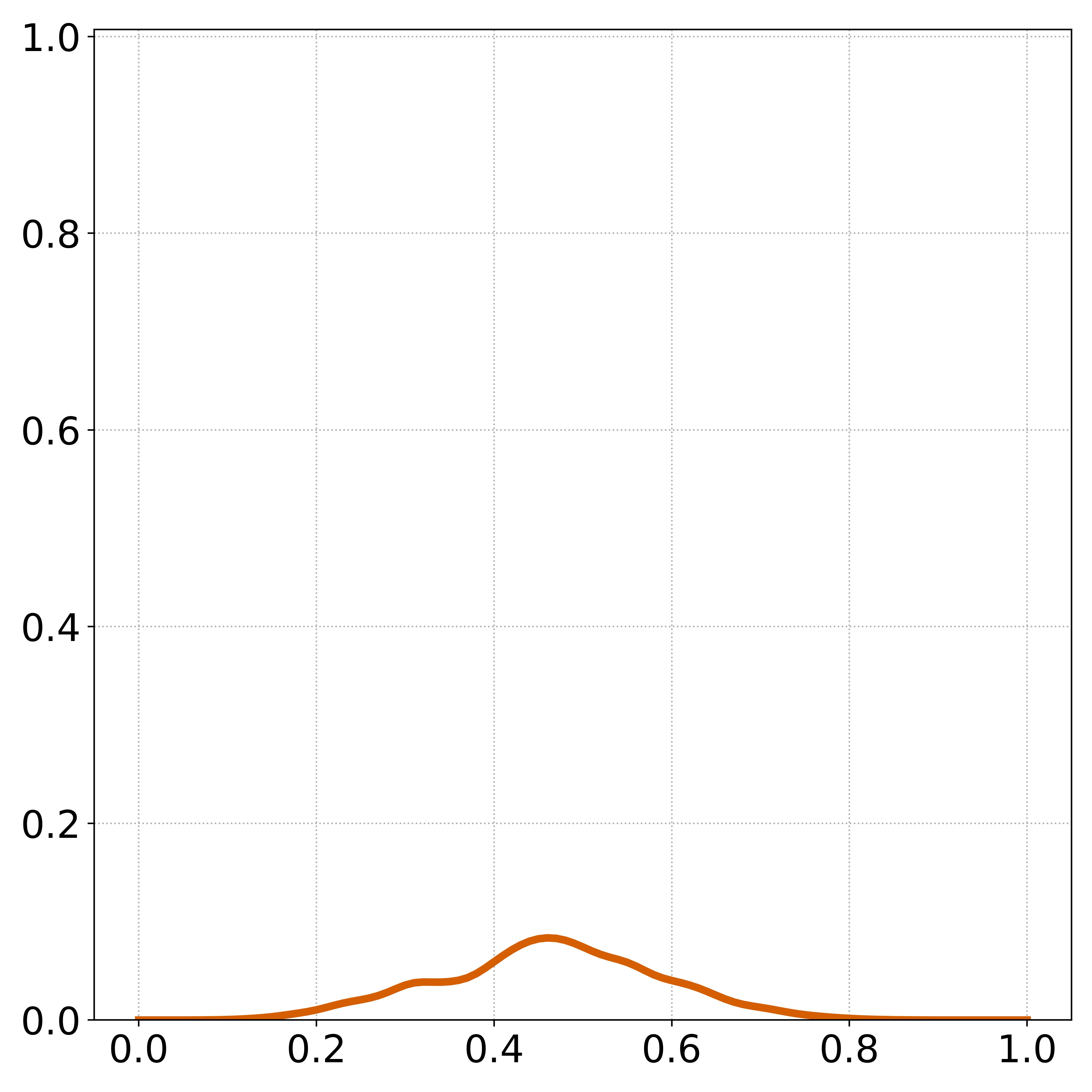}
          \caption{$f_3$ at $t = 10$}\label{fig:sim4f3t10}
          \end{subfigure}
          \begin{subfigure}{.21\textwidth}
        \includegraphics[width=\textwidth]{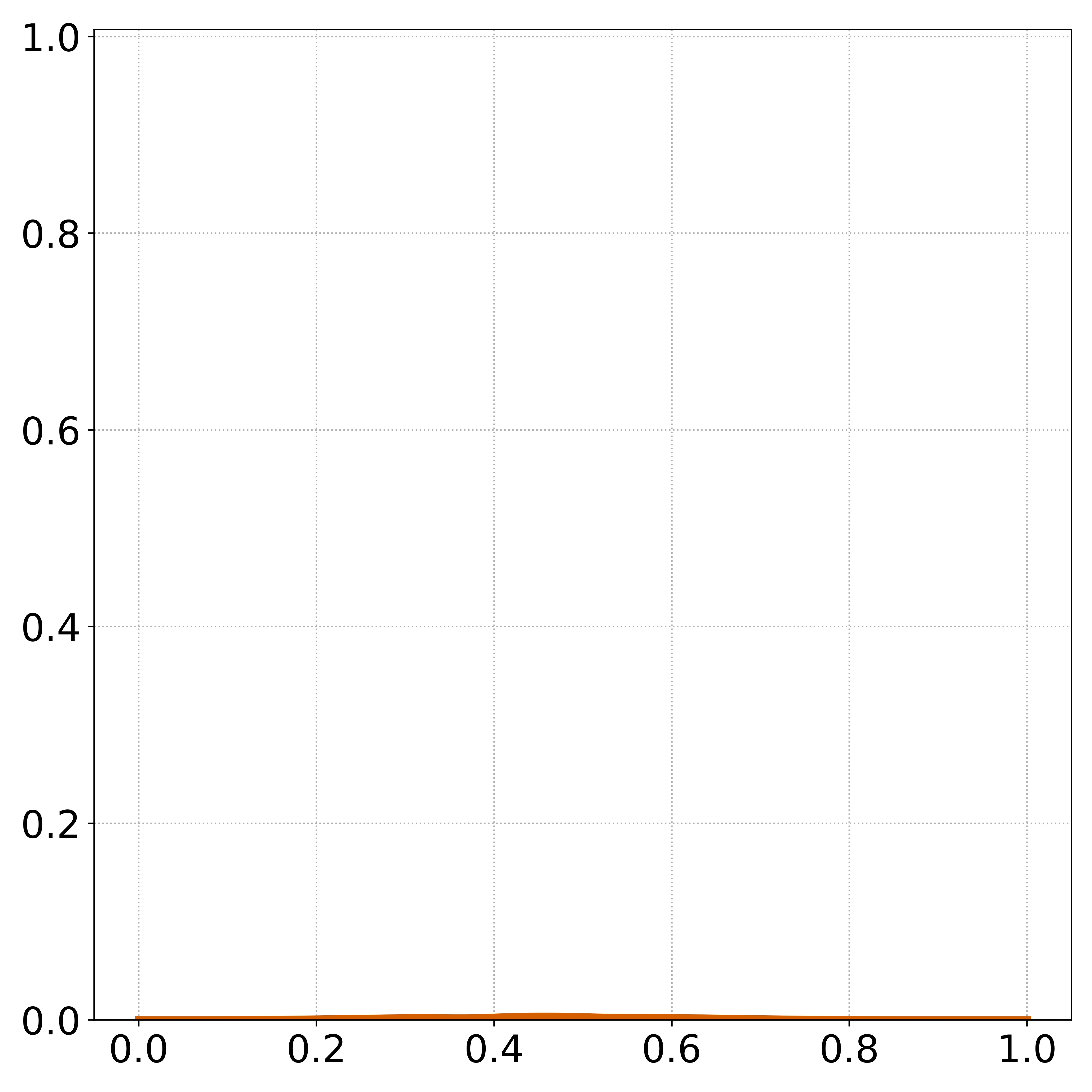}
          \caption{$f_3$ at $t = 20$}\label{fig:sim4f3t20}
          \end{subfigure}\\
           \begin{subfigure}{.21\textwidth}
        \includegraphics[width=\textwidth]{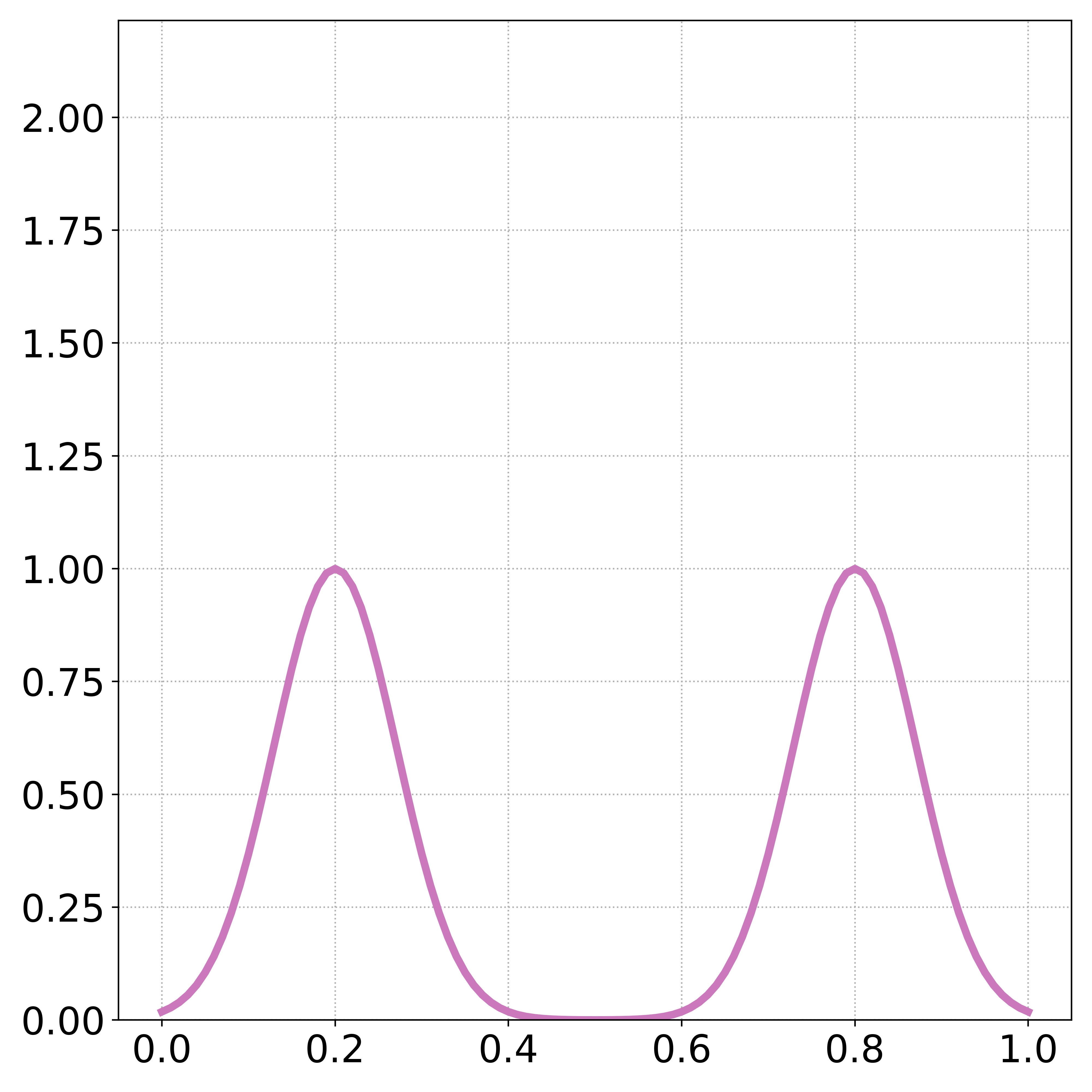}
          \caption{$f_4$ at $t = 0$}\label{fig:sim4f4t0}
          \end{subfigure}
        \begin{subfigure}{.21\textwidth}
        \includegraphics[width=\textwidth]{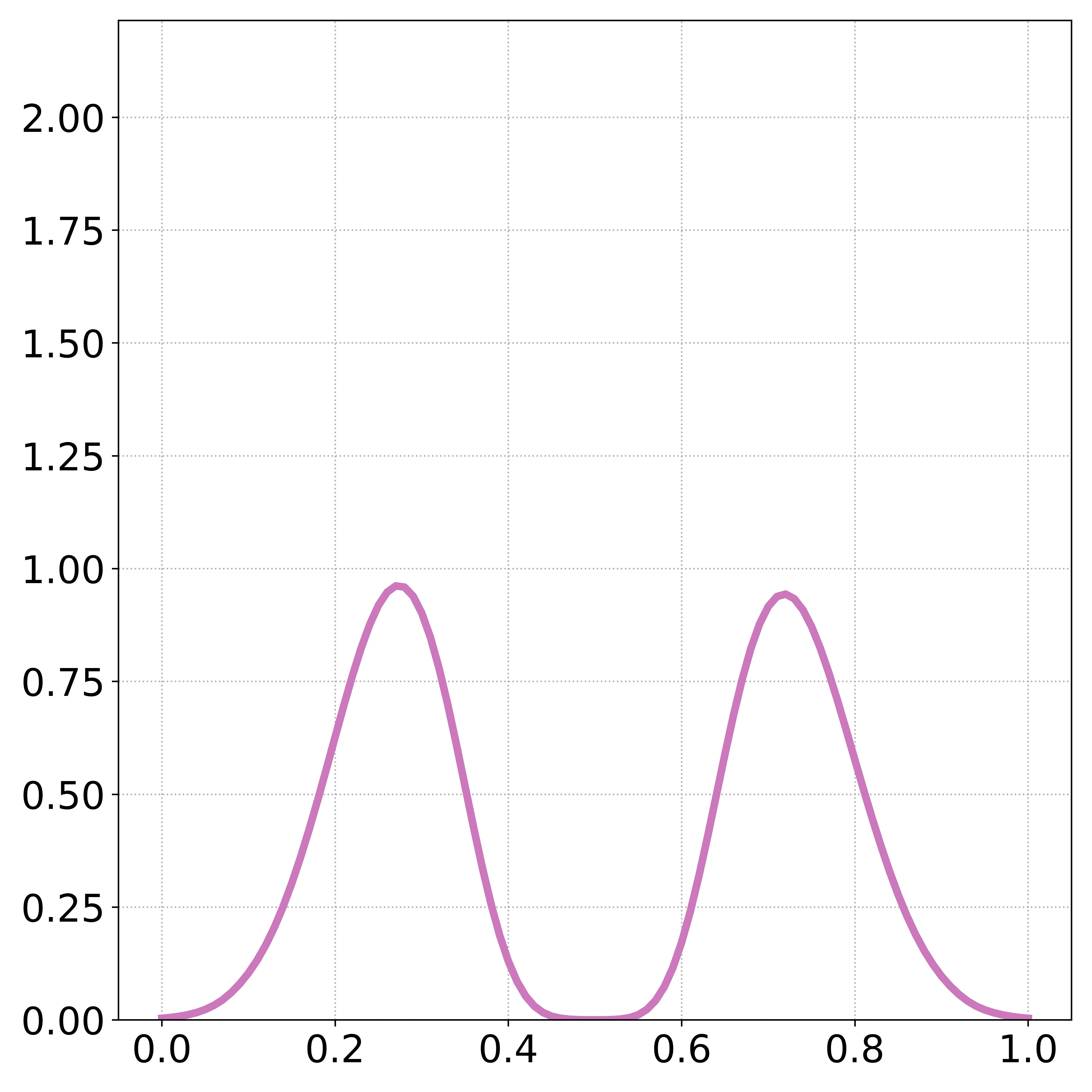}
          \caption{$f_4$ at $t = 5$}\label{fig:sim4f4t5}
          \end{subfigure}
        \begin{subfigure}{.21\textwidth}
        \includegraphics[width=\textwidth]{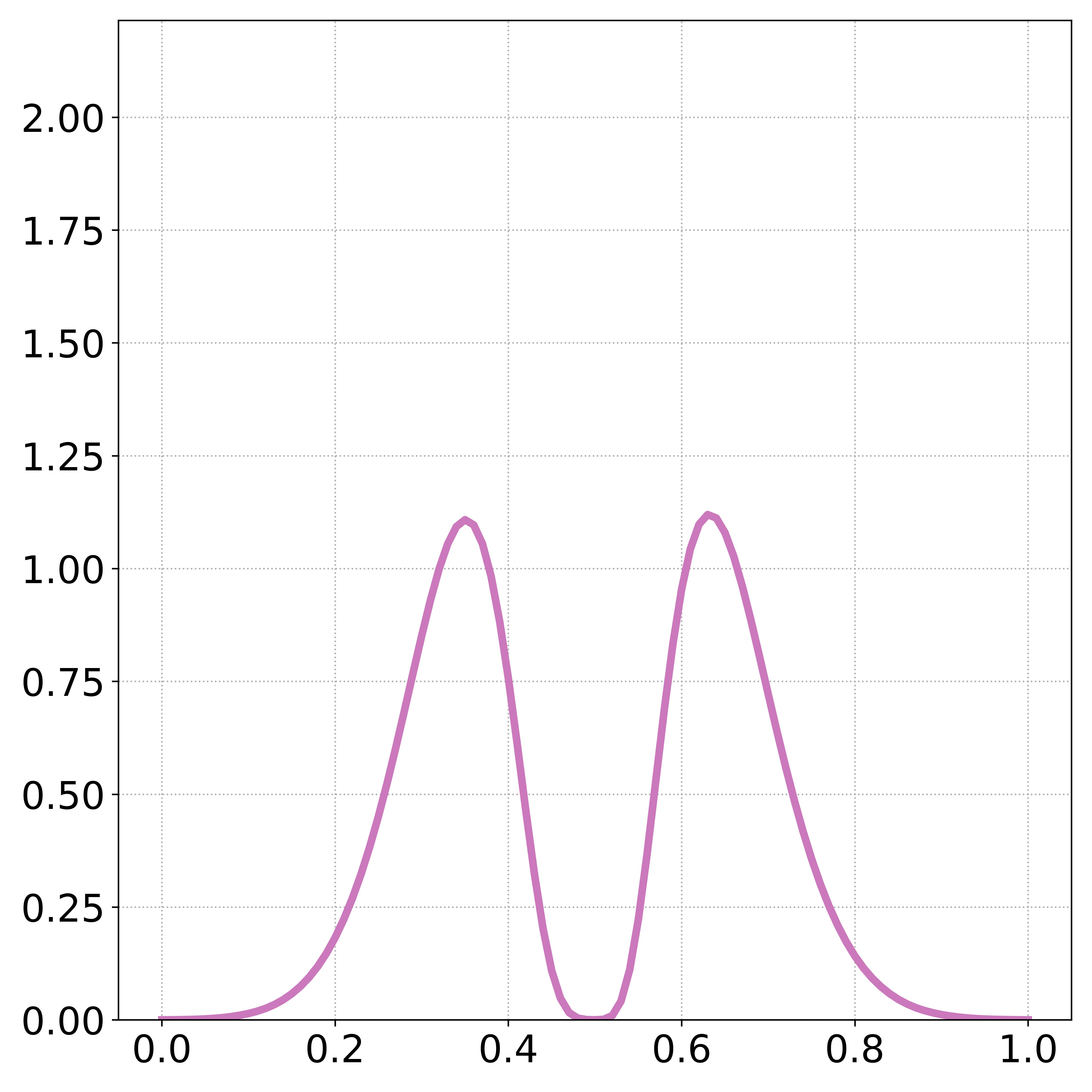}
        \caption{$f_4$ at $t = 10$}\label{fig:sim4f4t10}
          \end{subfigure}
          \begin{subfigure}{.21\textwidth}
        \includegraphics[width=\textwidth]{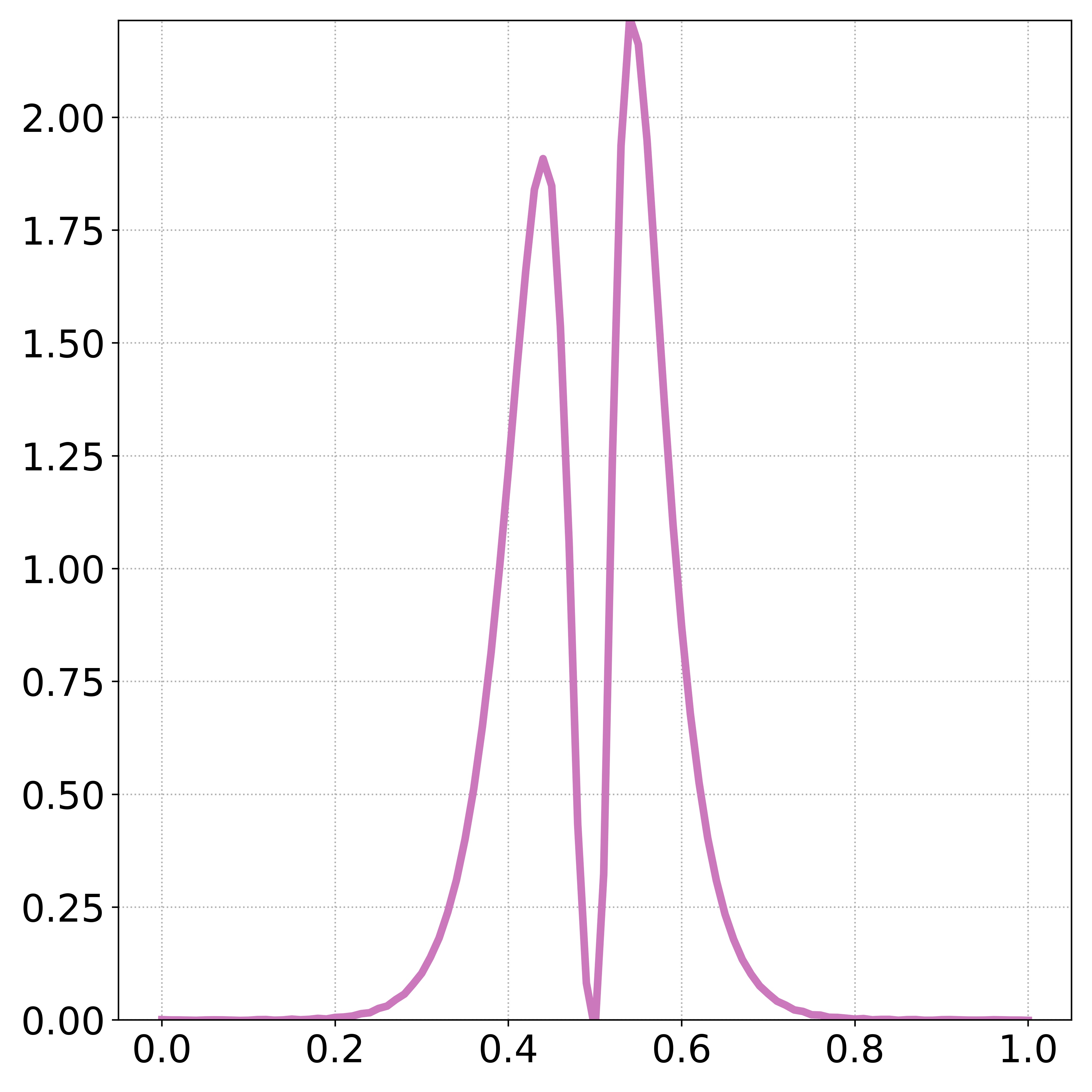}
          \caption{$f_4$ at $t = 20$}\label{fig:sim4f4t20}
          \end{subfigure}\\
           \begin{subfigure}{.21\textwidth}
        \includegraphics[width=\textwidth]{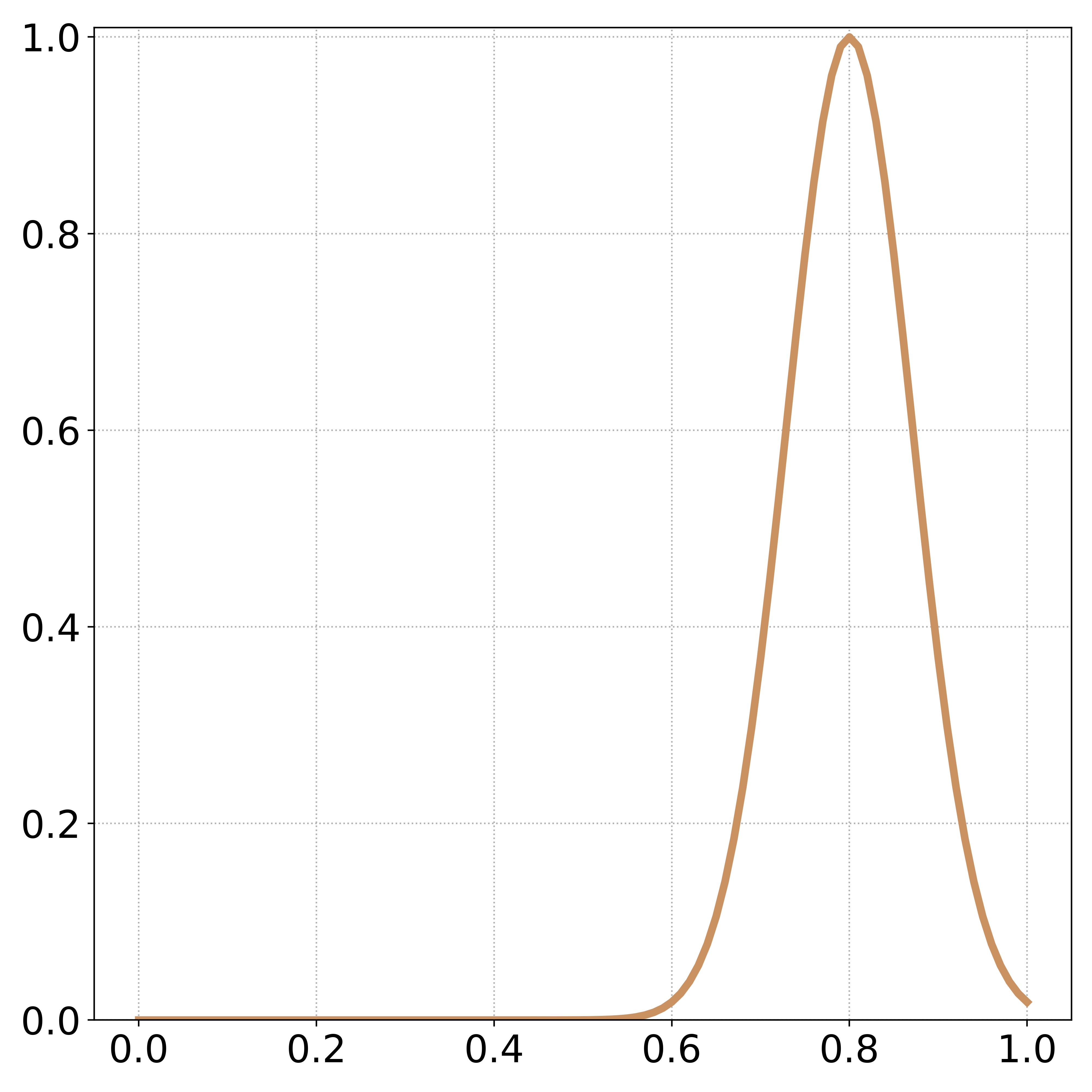}
        \caption{$f_5$ at $t = 0$}\label{fig:sim4f5t0}
          \end{subfigure}
        \begin{subfigure}{.21\textwidth}
        \includegraphics[width=\textwidth]{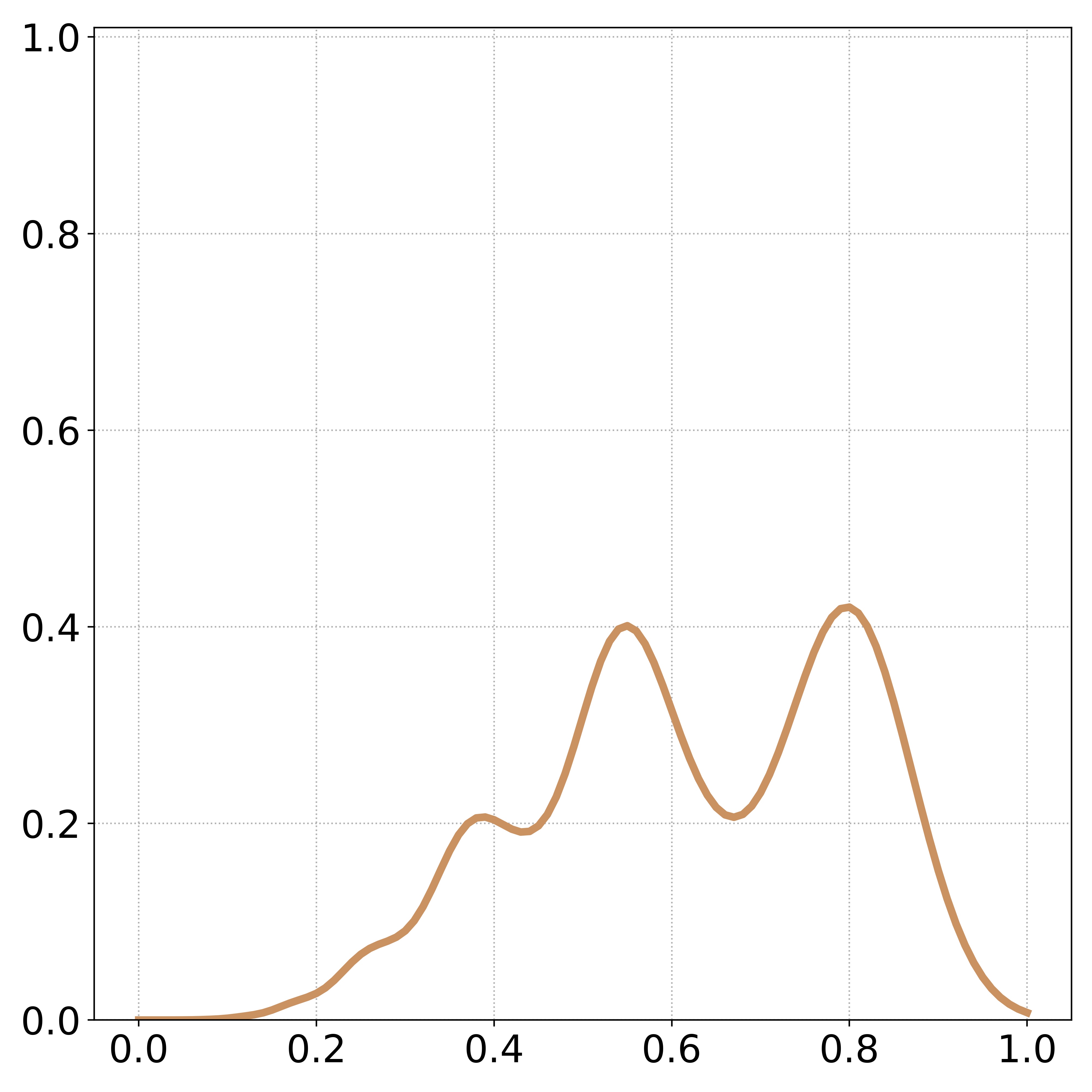}
          \caption{$f_5$ at $t = 5$} \label{fig:sim4f5t5}
          \end{subfigure}
        \begin{subfigure}{.21\textwidth}
         \includegraphics[width=\textwidth]{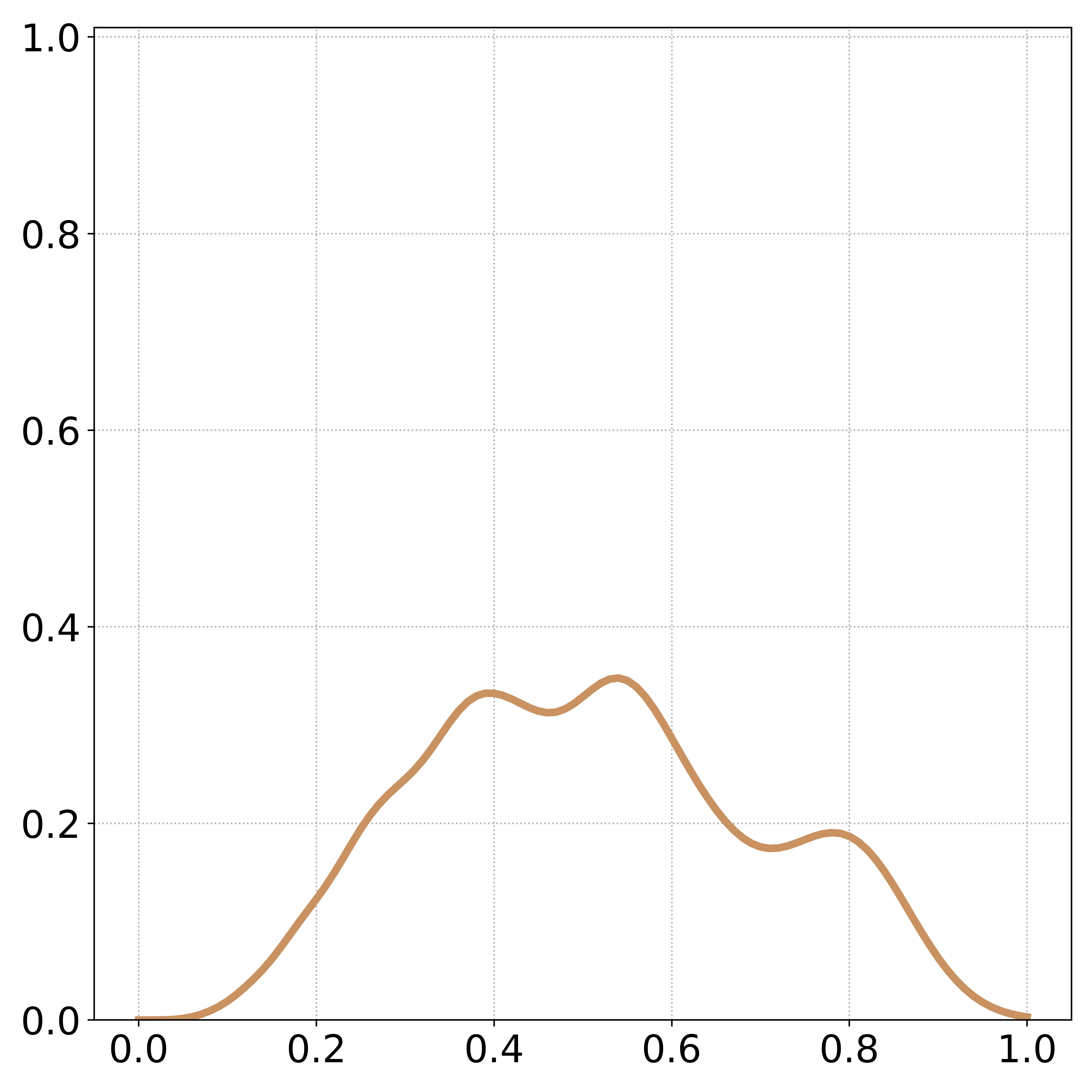}
         \caption{$f_5$ at $t = 10$}\label{fig:sim4f5t10}
          \end{subfigure}
          \begin{subfigure}{.21\textwidth}
        \includegraphics[width=\textwidth]{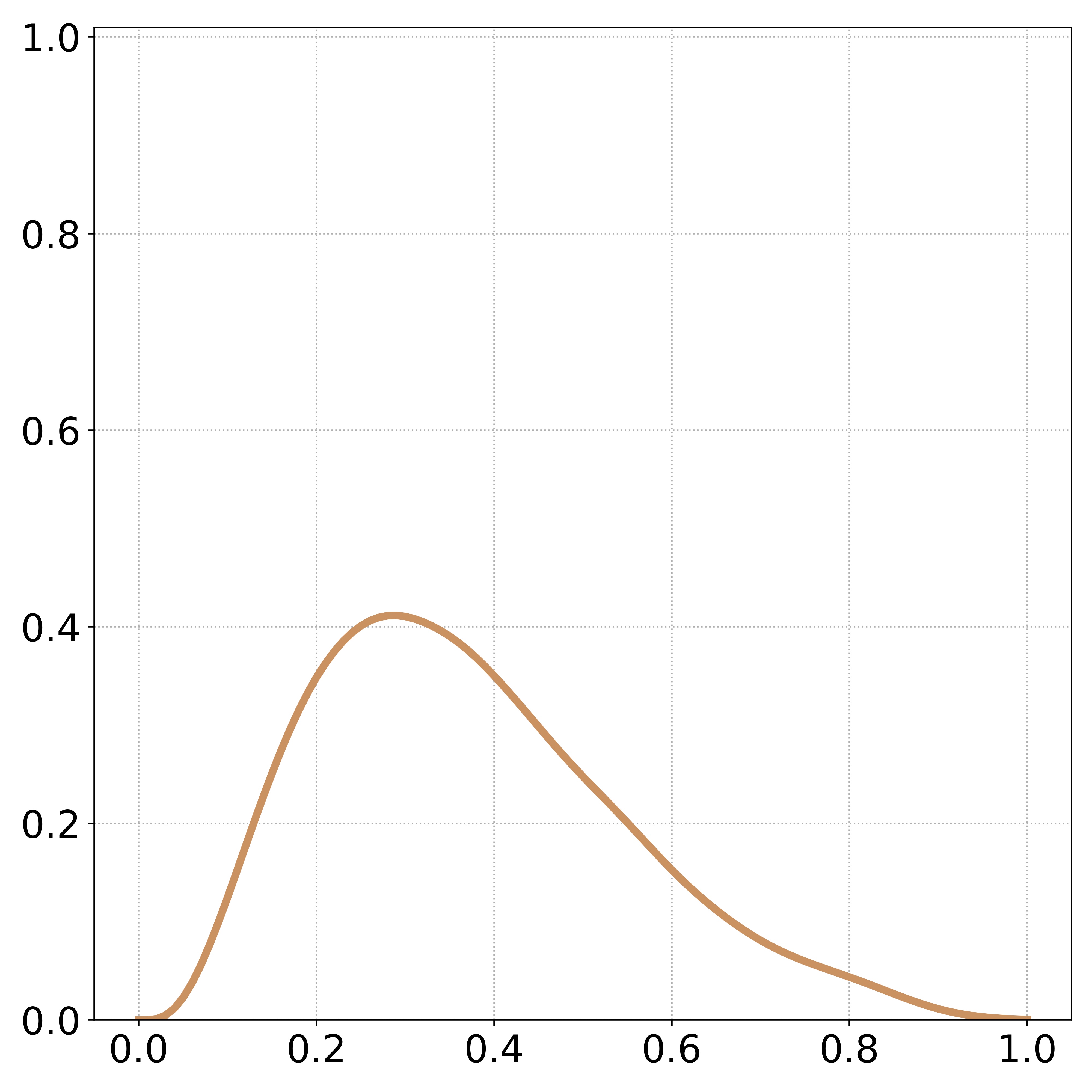}
          \caption{$f_5$ at $t = 20$}
           \label{fig:sim4f5t20}
          \end{subfigure}\\
  \caption{Time evolution of the kinetic system from Figure \ref{fig:sys5} using the transition functions from Table \ref{tab:5partTrans}. The figure shows how the microstate distribution changes over time based on these transitions. Each row represents the evolution of a different subsystem over the interval $[0,1]$.}
  \label{fig:sim4}
\end{figure}

\bibliographystyle{elsarticle-num} 
\bibliography{ref.bib}
\end{document}